\documentclass{alggeom}

\usepackage{setspace}

\usepackage{amssymb}   

\usepackage{amsmath}

\usepackage{mathrsfs}

\usepackage{stmaryrd}

\usepackage{enumitem}

\usepackage[all]{xy}

\usepackage{textcomp}

\usepackage{amsbsy}
\usepackage{verbatim}

\usepackage{amsthm}
\usepackage{url}

\usepackage{color}

\usepackage{xcolor}

\usepackage
[pdfauthor={Brandon Levin and Carl Wang-Erickson},
 pdftitle={A Harder--Narasimhan theory for Kisin modules},
 bookmarks=false]
{hyperref}

\newtheorem{thm}{Theorem}[subsection]
\newtheorem*{thm*}{Theorem}

\newtheorem{thm-defn}[thm]{Theorem/Definition}
\newtheorem{lem}[thm]{Lemma}
\newtheorem{prop}[thm]{Proposition}
\newtheorem{cor}[thm]{Corollary}

\theoremstyle{definition}
\newtheorem{defn}[thm]{Definition}
\newtheorem{defn-lem}[thm]{Definition/Lemma}

\newtheorem{eg}[thm]{Example}

\theoremstyle{remark}
\newtheorem{rem}[thm]{Remark}

\newcommand{\Z}{\mathbb{Z}}
\newcommand{\ZZ}{\mathbb{Z}}
\newcommand{\Q}{\mathbb{Q}}
\newcommand{\QQ}{\mathbb{Q}}
\newcommand{\Qp}{\mathbb{Q}_p}
\newcommand{\R}{\mathbb{R}}

\newcommand{\bG}{\mathbb{G}}

\newcommand{\F}{\mathbb{F}}

\newcommand{\fC}{\mathfrak{c}}

\newcommand{\fL}{\mathfrak{L}}
\newcommand{\fM}{\mathfrak{M}}
\newcommand{\fN}{\mathfrak{N}}

\newcommand{\fP}{\mathfrak{P}}
\newcommand{\fQ}{\mathfrak{Q}}

\newcommand{\fS}{\mathfrak{S}}

\newcommand{\fX}{\mathfrak{X}}
\newcommand{\fY}{\mathfrak{Y}}
\newcommand{\fZ}{\mathfrak{Z}}

\newcommand{\fm}{\mathfrak{m}}

\newcommand{\bF}{\mathbb{F}}

\newcommand{\bK}{\mathbb{K}}

\newcommand{\bQ}{\mathbb{Q}}
\newcommand{\bR}{\mathbb{R}}

\newcommand{\bZ}{\mathbb{Z}}

\newcommand{\cE}{\mathcal{E}}
\newcommand{\cF}{\mathcal{F}}
\newcommand{\cG}{\mathcal{G}}

\newcommand{\cK}{\mathcal{K}}
\newcommand{\cL}{\mathcal{L}}
\newcommand{\cM}{\mathcal{M}}
\newcommand{\cN}{\mathcal{N}}
\newcommand{\cO}{\mathcal{O}}
\newcommand{\cP}{\mathcal{P}}

\newcommand{\cS}{\mathcal{S}}

\newcommand{\phz}{\varphi}

\newcommand{\Zp}{\mathbb{Z}_p}

\newcommand{\val}{\mathrm{val}}

\newcommand{\Gal}{\mathrm{Gal}}
\newcommand{\Hom}{\mathrm{Hom}}

\newcommand{\End}{\mathrm{End}}

\newcommand{\GL}{\mathrm{GL}}

\newcommand{\Gm}{\mathbb{G}_m}
\newcommand{\ad}{\mathrm{ad}}

\newcommand{\Gr}{\mathrm{Gr}}

\DeclareMathOperator{\Char}{char}

\DeclareMathOperator{\Mod}{Mod}

\DeclareMathOperator{\Sym}{Sym}

\DeclareMathOperator{\Spec}{Spec}

\DeclareMathOperator{\rk}{rk}

\DeclareMathOperator{\Fil}{Fil}
\DeclareMathOperator{\Rep}{Rep}
\DeclareMathOperator{\gr}{gr}
\DeclareMathOperator{\coker}{coker}
\DeclareMathOperator{\HN}{HN}
\DeclareMathOperator{\Proj}{Proj}

\newcommand{\ra}{\rightarrow}
\newcommand{\lra}{\longrightarrow}
\newcommand{\iarrow}{\hookrightarrow}

\newcommand{\lrisom}{\buildrel\sim\over\longrightarrow}
\newcommand{\risom}{\buildrel\sim\over\rightarrow}
\newcommand{\rinj}{\hookrightarrow}
\newcommand{\rsurj}{\twoheadrightarrow}

\newcommand{\tor}{\mathrm{tor}}
\newcommand{\fl}{\mathrm{fl}}

\newcommand{\et}{\mathrm{\text{\'et}}}
\newcommand{\Filt}{\mathrm{Filt}}

\newcommand{\Fp}{{\bF_p}}

\newcommand{\lb}{{[\![}}
\newcommand{\rb}{{]\!]}}
\newcommand{\lp}{{(\!(}}
\newcommand{\rp}{{)\!)}}

\newcommand{\abs}[1]{\ensuremath{\left|#1\right|}}

\newcommand{\weight}{\mathrm{weight}}

\makeatletter
\let\c@equation\c@thm
\makeatother

\numberwithin{equation}{subsection}

\begin{document}

\title{A Harder--Narasimhan theory for Kisin modules}

\author{Brandon Levin}
\email{bwlevin@math.arizona.edu}
\address{Department of Mathematics, University of Arizona, Tucson, AZ 95721, USA}

\author{Carl Wang-Erickson}
\email{carl.wang-erickson@pitt.edu}
\address{Department of Mathematics, University of Pittsburgh,
	Pittsburgh, PA 15260, USA}

\subjclass[2010]{11S20 (primary), 14L24, 14G35 (secondary).}
\keywords{algebraic groups, deformation theory, finite flat group schemes, geometric invariant theory,  $p$-adic Hodge theory}

\thanks{
C.W.E.\ thanks the AMS and the Simons Foundation for support for this project in the form of an AMS-Simons Travel Grant.
}

\date{\today}

\begin{abstract} 
We develop a Harder--Narasimhan theory for Kisin modules generalizing a similar theory for finite flat group schemes due to Fargues \cite{fargues2010}.  We prove the tensor product theorem, i.e., that the tensor product of semi-stable objects is again semi-stable.  We then apply the tensor product theorem to the study of Kisin varieties for arbitrary connected reductive groups. 
\end{abstract}

\maketitle

\tableofcontents

\section{Introduction}

\subsection{HN-theory} Let $p$ be a prime number. The study of canonical filtrations on finite flat group schemes (commutative of order a power of $p$) plays an important role in the theory of integral models of Shimura varieties. If $\cG$ is a finite flat group scheme over $\Spec \cO_K$, where $\cO_K$ is a complete discrete valuation ring of residue characteristic $p$, then perhaps the most well-known filtration is the connected-\'etale sequence
\[ 
0 \rightarrow \cG^0 \rightarrow \cG \rightarrow \cG^{\et} \rightarrow 0.
\] 
This sequence can be further refined by considering the largest subgroup scheme $\cG^{\mathrm{mult}} \subset \cG^0$ whose Cartier dual is \'etale (the multiplicative part).  

In \cite{fargues2010}, Fargues introduced a canonical filtration on $0 \subset \cG_1 \subset \ldots \subset \cG_n = \cG$ which further refines the above sequence in the sense that $\cG_1 = \cG^{\mathrm{mult}}$ and $\cG_n/\cG_{n-1} = \cG^{\et}$. The filtration in \emph{loc.~cit.}\ arises from a Harder--Narasimhan theory (\textit{HN-theory}) on the category of finite flat (commutative) group schemes over $\Spec \cO_K$. A HN-theory is defined by giving a map to an abelian category ($\cG \mapsto \cG_K$) together with a notion of rank and degree satisfying a number of properties (see Proposition \ref{prop:HN_main}), generalizing the original setup of vector bundles on an algebraic curve \cite{HN1974}. There are distinguished objects that are called \emph{semi-stable}, meaning that their canonical HN-filtration is trivial. The filtration constructed in \cite{fargues2010} is related to the action of the Hecke operators on certain rigid analytic moduli spaces of $p$-divisible groups associated to Shimura varieties and canonical subgroups (\cite{fargues2007, fargues2011}). 

In a different context, Kisin gave a linear-algebraic description of the category of finite flat group schemes over $\cO_K$ \cite{crfc}, where $\cO_K$ is the ring of integers of a finite extension $K/\Qp$.  More precisely, the category of $p^{\infty}$-torsion finite flat group schemes over $\cO_K$ is anti-equivalent to the category $\Mod_{\fS}^{\phz, [0,1]}$ of \textit{Kisin modules} with height in $[0,1]$. 

In this paper, we begin by showing that the HN-theory introduced in \cite{fargues2010} can be reinterpreted in terms of the category $\Mod_{\fS}^{\phz, [0,1]}$. In fact, the HN-theory extends naturally to the category of Kisin modules with bounded height, denoted $\Mod_{\fS}^{\phz}$.

\subsection{The tensor product theorem} 
An advantage of working in the larger category $\Mod_{\fS}^{\phz}$ is that this category has the structure of a rigid exact category (i.e., duals, tensor products, etc.).  In contrast, there is no good general notion of tensor product between two finite flat group schemes.  Furthermore, Kisin modules with height greater than 1 are related to Galois representations with larger Hodge--Tate weights, and so this category is of interest in its own right.  

While much of the power of an HN-theory to give structure to an additive exact category arises formally from verifying basic axioms about degree and rank (see e.g.~\cite{pottharst}), there is no known general approach to demonstrate its compatibility with additional tensor structure. Critically, one hopes that the tensor product of semi-stable objects is semi-stable. This final statement is known as \emph{the tensor product theorem}. 

\begin{thm}[{(Theorem \ref{thm:tensor})}]
\label{thm:tpintro}
The tensor product of semi-stable Kisin modules of bounded height is again semi-stable. 
\end{thm}

Its proof is one of the main innovations of this work. It is modeled in part on Totaro's proof of the tensor product theorem for filtered isocrystals over a $p$-adic field \cite{totaro1996}, but departs significantly from Totaro's approach. Totaro's strategy heavily relies on the notion of semi-stable subspaces of a tensor product vector space $V \otimes W$, and the fact that the degree of an isocrystal can be calculated in terms of the degree of a subspace with respect to a filtration on $V \otimes W$ arising from filtrations on $V$ and $W$. Moreover, associated to any non-semi-stable (i.e.~unstable) subspace of $V \otimes W$ is a filtration on $V$ and a filtration on $W$ called the Kempf filtration, which characterizes the failure of semi-stability. The Kempf filtration also figures in to Totaro's strategy.

The degree of a Kisin module is not a priori related to a filtration on any vector space. Indeed, $p$-torsion Kisin modules are modules over $\fS/p = k[\![u]\!]$. In order to build off of Totaro's strategy, we bound the degree of a Kisin submodule $\fP \subset \fM \otimes \fN$ by the degree of $\fP \pmod{u}$ as a subspace of $\fM \otimes \fN \pmod{u}$ according to a certain filtration. Not only must we cope with the fact that this is a bound instead of an exact formula, but the more serious obstacle is that submodules of Kisin modules cannot be identified modulo $u$. To resolve this problem, we must relate semi-stability of the generic fiber $\fP[1/u] \subset \fM[1/u] \otimes \fN[1/u]$, where subobjects of $\fP$ are identifiable, to semi-stability modulo $u$. We will establish these algebro-geometric results, calling them ``results of Langton type'' because Langton proved the first result of this kind in another context \cite{langton1975}. These results build upon the full power of geometric invariant theory, including the Kempf--Ness stratification \cite[\S\S12-13]{kirwan1984} and adequate moduli spaces \cite{alper2014}.

It is a well-known consequence of the tensor product theorem that one can then extend a Harder--Narasimhan theory from the setting of `vector bundles' to that of $G$-bundles.  Theorem \ref{thm:tpintro} should have applications to the study of Shimura varieties beyond PEL-type along the lines of \cite{fargues2007, fargues2019}.  Some applications to Shimura varieties were obtained in \cite{irissarry2017, irissarry2016} (see Section \ref{sec:RR} for further discussion). However, we do not pursue that here. Instead, we discuss an application of the tensor product theorem in a different direction. 

\subsection{Kisin modules with $G$-structure} In \cite{levin2015}, the first author developed a theory of Kisin modules with $G$-structure.  In \S\S4-5, We use our HN-theory and the tensor product theorem to study moduli spaces of Kisin modules with $G$-structure.

Specifically, in \cite{mffgsm}, Kisin introduced a projective variety $X_{\overline{\rho}}$ whose closed points parametrize finite flat group schemes over $\Spec \cO_K$ with generic fiber $\overline{\rho}:\Gal(\overline{K}/K) \ra \GL_n(\overline{\F})$.   These were later called \emph{Kisin varieties} by Pappas and Rapoport \cite{PR2009}.   More generally, if $\cM$ is a \'etale $\phi$-module over $\F_p(\!(u)\!)$ and $\nu = (a_1, a_2, \ldots, a_n) \in \Z^n$ ($K$ is assumed to be totally ramified over $\Qp$), then 
\[
X_{\cM}^{\nu}(\F) = \{ \fM[1/u] \cong \cM   \mid \fM \text{ has Hodge type} \leq \nu \}.
\] 
$X_{\cM}^{\nu}$ is projective scheme parametrizing Kisin modules with generic fiber $\cM$ satisfying height conditions given by $\nu$. Kisin varieties resemble affine Deligne--Lusztig varieties in form, but much less is known about their structure.  The most important question for applications to Galois deformation rings and modularity lifting is what are the connected components of $X_{\cM}^{\nu}$.
  
It was observed in \cite[\S2.4.15]{mffgsm} that by considering the connected-\'etale sequence, or, more generally, the \'etale and multiplicative parts, one gets certain discrete invariants on $X_{\cM}^{\nu}$ (see Proposition \ref{prop:upperetrank}).   Conjecture 2.4.16 in \emph{loc.\ cit.}\ asserts that, under some hypotheses, in fact these should be the only discrete invariants. 
 For $\GL_2$, the conjecture says concretely that if $\overline{\rho}$ is indecomposable, then there are at most two components: an ordinary component and a non-ordinary component. 
 This conjecture and a number of generalizations in dimension two were proven by Gee, Hellmann, Imai, and Kisin \cite{gee2006, mffgsm, hellmann2009, imai2010, hellmann2011, imai2012}.  In the case of Kisin modules with descent for $\GL_2$, Caruso--David--M\'ezard prove a related result about connectedness of Kisin varieties \cite{CDM2}. Essentially nothing in any generality is known beyond $\GL_2$ except over mildly ramified fields $K/\Qp$. One motivation for this project was to give a more conceptual explanation for the conjecture, adapting it for other reductive groups as a starting point for a systematic study of Kisin varieties beyond dimension 2. 

Our first result about Kisin varieties is the following: 
\begin{thm}[{(Proposition \ref{HNoverHodge}, Corollary \ref{HNstrata})}] 
 There is a decomposition 
\[
X^{\nu}_{\cM} = \bigcup X^{\nu, P}_{\cM} 
\]
by locally closed reduced subschemes such that the closed points of $X^{\nu, P}_{\cM}$ have HN-polygon $P$.  Furthermore, the (normalized) HN-polygon of any $\fM \in X^{\nu}_{\cM}(\F)$ lies above the Hodge polygon associated to $\nu$ with the same endpoint. 
\end{thm}
The theorem essentially follows from the semi-continuity of the HN-polygon in families with constant generic fiber.   

In \S\ref{sec:kisin_var}, we introduce Kisin varieties $X^{\nu}_{\cP}$ for any connected reductive group $G$ over $\F$ and $\nu$ a cocharacter of $G$, and we prove the analogous result in that setting (Theorem \ref{GHNstrata}). We then show that, by considering unions of certain strata defined group-theoretically, we can produce discrete invariants which generalize the ordinary/non-ordinary components appearing in dimension two. See Theorem \ref{thm:HNcomp}) for details. We provide some examples in \S \ref{sec:examples}.   In some situations, the connected components of these Kisin varieties with $G$-structure can be related to connected components of $G$-valued Galois deformation rings (see \cite[Cor.\ 4.4.2]{levin2015}). 

\subsection{Application to deformation to characteristic 0} 
In \S\ref{sec:flat}, we apply our HN-theory, which a priori applies only to $p$-power torsion Kisin modules, to flat Kisin modules $\fM$ by studying the limit $\fM/p^n$ as $n$ increases. In this we follow Fargues \cite{fargues2019}, who, analogously, showed how his theory for finite flat group schemes could address $p$-divisible groups. Principally, we show that there is a distinct HN-theory on the isogeny category of Kisin modules for which our theory is an ``approximation modulo $p^n$.'' We conclude \S\ref{sec:flat} by showing that there is a relation between an HN-theory arising from de Rham filtration associated to a crystalline representation and the HN-polygon of the associated iso-Kisin module.  Thus, it appears that our HN-theory is a mod $p$ avatar of the Hodge filtration.

\subsection{Related results}  \label{sec:RR}

During the course of this project, we were informed that Macarena Peche Irissarry had obtained partial results towards the tensor product theorem. Her thesis, which appeared while this article was under review, also includes applications to Shimura varieties \cite{irissarry2017, irissarry2016}. Likewise, Cornut informed us that our results gave rise to an alternate approach to the tensor product theorem using convex metric geometry \cite{cornut2018}. Subsequently, Cornut and Peche Irissarry apply the abstract tensor product theorem of \cite{cornut2018} to prove a tensor product theorem for Breuil--Kisin--Fargues modules \cite{CPI2019}.

Recently, there has been a growing interest in studying more general moduli spaces of Kisin modules where one allows the generic fiber to vary.  These moduli were first considered by \cite{PR2009}, but have been further developed by Emerton--Gee in \cite{EG2015, EG2019}, Caraiani--Emerton--Gee--Savitt in \cite{CEGS2018}, and joint work of Caraiani and the first author \cite{CL2015}.  It would be interesting to understand the behavior of the HN-filtrations on these spaces.

\subsection{Notation}  
Let $K/\Qp$ be a finite extension with residue field $k$ and uniformizer $\pi$.  Let $e$ be the ramification degree of $K/\bQ_p$,  $f = [k : \bF_p]$, and $g = [K : \bQ_p]$, so that $g = ef$.  Let $\fS = W(k)[\![u]\!]$ with Frobenius $\phz$ extending the standard Frobenius on $W(k)$ by $\phz(u) = u^p$. Let $E(u)$ be the minimal polynomial of $\pi$ over $W(k)[1/p]$, a monic polynomial of degree $e$.  Let $F = k(\!(u)\!)$ with separable closure $F^{s}$.  Especially in \S\ref{sec:tensor}, we will write $\cO$ for the valuation ring $k\lb u \rb$ of $F$. 

We will often be interested in $\fS$-modules $\fM$ with a Frobenius semi-linear endomorphism. Thinking of $\fM$ and $\fS$ as left $\fS$-modules and giving $\fS$ a right $\fS$-module structure via $\phz$, let $\phz^*(\fM)$ stand for the left $\fS$-module $\fS \otimes_{\phz, \fS} \fM$. A Frobenius semi-linear endomorphism of $\fM$ induces a $\fS$-linear map $\phi_\fM: \phz^*(\fM) \ra \fM$ that we will write as $\phi_\fM$, which we refer to as the linearized structure map of $\fM$. 

Fix a sequence $\pi_n$ such that $\pi_n^p = \pi_{n-1}$ and $\pi_0 = \pi$.  Let $K_{\infty}$ denote the completion of $K((\pi_n)_n)$.  

\section{HN-theory}
\label{sec:HN}

In this section, we develop a Harder--Narasimhan theory on the category of (torsion) Kisin modules with bounded height.  After some background, we demonstrate its basic properties in \S 2.3 (Proposition \ref{prop:HN_main}).  We then deduce the standard consequences (i.e., HN-filtration, HN-polygon, etc.). 

\subsection{Categories of Kisin modules}

\begin{defn}
A \emph{Kisin module} is a pair  $(\fM, \phi_{\fM})$ consisting of a finitely generated $\fS$-module $\fM$ that is $u$-torsion-free along with an isomorphism $\phi_{\fM}:\phz^*(\fM)[1/E(u)] \cong \fM[1/E(u)]$. A morphism of Kisin modules $\fM \ra \fN$ is a $\fS$-linear function $\fM \ra \fN$ commuting with the structure maps $\phi_\fM, \phi_\fN$. 
\end{defn}

The additive exact category structure on Kisin modules is as follows. 
\begin{defn}
The category of Kisin modules $\Mod_{\fS}^{\phz}$ inherits an additive, $\fS$-linear structure as $\fS$-modules, where morphisms are those which are compatible with Frobenius. We say a sequence of Kisin modules $0 \ra \fM_1 \ra\fM_2 \ra \fM_3 \ra 0$ is \emph{strict short exact} when it is exact in the category of $\fS$-modules. A \emph{strict subobject} of $\fM$ is sub-Kisin module $\fN \subset \fM$ which sits in a strict short exact sequence, i.e.\ the $\fS$-module cokernel is $u$-torsion free. 

We write $\Mod_{\fS,\fl}^\phz$ for the full exact subcategory of Kisin modules that are flat as $\fS$-modules.
\end{defn}

The powers of $E(u)$ that appear control the height of $(\fM, \phi_{\fM})$. 
\begin{defn}
Fix integers $a < b$.  A Kisin module $(\fM, \phi_\fM)$ is said to have \emph{height in $[a,b]$} if
\begin{equation} 
\label{eq:heightcond}
E(u)^a \fM \supset \phi_{\fM}(\phz^*(\fM)) \supset E(u)^b \fM
\end{equation}
in $\fM[1/E(u)]$. In particular, a Kisin module with height in $[0,b]$ is said to be \emph{effective}. We denote the category of Kisin modules with height in $[a,b]$ by $\Mod_{\fS}^{\phz, [a,b]}$. 
\end{defn}

Every Kisin module is an object of $\Mod_{\fS}^{\phz, [a,b]}$ for some $a,b$. In other words, every Kisin module has bounded height. 

Effective Kisin modules may be thought of as a finite $u$-torsion free $\fS$-module $\fM$ with a semi-linear endomorphism $\fM \ra \fM$ such that the induced linearized endomorphism $\phi_\fM: \phz^*\fM \ra \fM$ has has cokernel killed by some power of $E(u)$. We will reduce many claims to the effective case. 

\begin{rem} 
\label{rem:rigid}
We define tensor products in $\Mod_{\fS}^{\phz}$ in the natural way.  If $\fM \in \Mod_{\fS}^{\phz}$ is flat over $\fS/p^n$, then the dual is defined to be $\fM^* := \Hom(\fM, \fS/p^n)$ with Frobenius given by $f \mapsto \phz \circ f \circ \phi_{\fM}^{-1}$.  The category of Kisin modules which are flat over $\fS/p^n$ is a rigid exact tensor category. See \cite[\S4.2]{levin2013} for details. 
\end{rem}
   
We will initially set up our HN-theory with effective torsion Kisin modules. 
\begin{defn} 
A torsion \emph{Kisin module with height $\leq h$} is an effective Kisin module $(\fM, \phi_\fM)$ such that $\fM$ is $p$-power torsion and $\coker(\phi_\fM)$ is killed by $E(u)^h$.  Denote this category by $\Mod_{\fS, \tor}^{\phz, [0,h]}$. 
\end{defn}

\begin{rem}
\label{rem:anti-equiv}
The category $\Mod_{\fS, \tor}^{\phz, [0,1]}$ is anti-equivalent to the category of finite flat group schemes over the maximal order $\cO_K$ of $K$ (\cite[Thm.\ 2.3.5]{crfc}).  Likewise, $\Mod_{\fS,\fl}^{\phz,[0,1]}$ is anti-equivalent to the category of $p$-divisible groups over $\cO_K$. 
\end{rem}

\begin{defn}
\label{defn:et}
An effective Kisin module with height $\leq 0$ is called \emph{\'etale}. Equivalently, these are effective Kisin modules $\fM$ where $\phi_\fM : \phz^*\fM \ra \fM$ is an isomorphism. 
\end{defn}

One reason this terminology is appropriate is that a finite flat group scheme is \'etale if and only if its corresponding torsion Kisin module is \'etale. 

\begin{rem} 
If $\fM \in \Mod_{\fS}^{\phz, [a,b]}$, then any $\phz$-stable strict sub-$\fS$-module of $\fM$ has height in $[a,b]$. Consequently, the set of strict subobjects (and, later, the HN-filtration) of a Kisin module are insensitive to changing the range of heights allowed.
\end{rem} 

\begin{eg}
\label{eg:mult}
For $r \in \bZ$, let $(\fS(r), \phi_{\fS(r)})$ be the rank 1 Kisin module consisting of the $\fS$-module $\fS e$ and $\phi_\fS$ defined by $\phi_{\fS(r)}(e) = (c_0^{-1} E(u))^r e$ where  $c_0 = E(0)/p$.  These are Kisin modules of height in $[r,r]$. 
\end{eg}

For any Kisin module $\fM$ with height in $[a,b]$ and any integer $r$, let $\fM(r) := \fM \otimes \fS(r)$ denote the ``$r$th \emph{Tate twist}.''  This terminology is justified because $\fS/p^n(1)$ corresponds to finite flat group scheme $\mu_{p^n}$ under the anti-equivalence of Example \ref{eg:mult}, and hence $\fS(1)$ corresponds to the multiplicative $p$-divisible group.  

\begin{rem}
\label{rem:height_twist}
It is easy to see that $\fM(r)$ has height in $[a + r, b+ r]$. Thus results about effective Kisin modules have sensible extensions to the general case. 
\end{rem}

\subsection{Rank and degree of Kisin modules}

Our Harder--Narasimhan theory for $\Mod_{\fS, \tor}^\phz$ will depend on a notion of rank and degree of a Kisin module, satisfying certain axioms. We will describe this notion of rank and degree, and then describe a few basic properties. The HN-axioms will be proved in \S\ref{subsec:HN-theory}. 

The first piece of a Harder--Narasimhan theory on $\Mod_{\fS, \tor}^\phz$ is a functor to an abelian category, which we now specify. Let $\cO_\cE$ denote the $p$-adic completion of $\fS[1/u]$, which inherits the Frobenius map $\phz$ from $\fS$. This is a complete Noetherian  local ring with residue field $F = k(\!(u)\!)$. 

\begin{defn}
An \emph{\'{e}tale $\cO_\cE$-module} $(\cM, \phi_\cM)$ is a pair consisting of a finite $\cO_\cE$-module $\cM$ with a $\phz$-semi-linear automorphism, its linearized version being $\phi_\cM : \phz^*\cM \risom \cM$. This category is written $\Mod_{\cO_\cE}^\phz$, and the $p$-power torsion full subcategory is written $\Mod_{\cO_\cE, \tor}^\phz$. 
\end{defn}

There is a functor $GF: \Mod_\fS^\phz \ra \Mod_{\cO_\cE}^\phz$, which we call the ``generic fiber functor,'' given by 
\[
\fM \mapsto \fM \otimes_{\fS} \cO_\cE. 
\]
It restricts to a functor $GF: \Mod_{\fS, \tor}^\phz \ra \Mod_{\cO_\cE, \tor}^\phz$. This is a natural analog of the generic fiber of a finite flat group scheme. We will observe that any subobject $\cN \subset GF(\fM)$ induces a strict subobject $\fN := \cN \cap \fM \subset GF(\fM)$ of $\fM$, and that all strict subobjects of $\fM$ are realized in this way. 

\begin{rem}
\label{rem:JMF} In fact $\Mod_{\cO_\cE, \tor}^\phz$ is an abelian tensor category, and it is equivalent to the abelian category of $p$-power torsion representations of $\Gal(F^s/F)$.  By the theory of norm fields, this is also equivalent to representations of $\Gal(\overline{K}/K_{\infty})$ though we will not emphasize this.  See \cite{fontaine1990}. 
\end{rem}

We may now define the rank and degree of a torsion Kisin module $\fM$; note that the rank of $\fM$ depends only upon $\fM \otimes_\fS \cO_\cE$.  For any $A$-module $M$, let $\ell_A(M)$ denote its length if the length is finite.  Recall that $g = [K:\Qp]$. 
\begin{defn}  
\label{defn:rk_deg}
For $\cM$ a torsion \'{e}tale $\cO_\cE$-module, we let the \emph{rank} of $\cM$ to be
\[
\rk(\cM) := \ell_{\cO_\cE}(\cM).
\]
Let $\fM$ be a torsion Kisin module. We define $\rk(\fM)$ to be $\rk(GF(\fM))$. Assuming that $\fM$ is effective, let $\coker(\phi_{\fM})$ be the cokernel of the linearized Frobenius. We set
\[
\deg(\fM)  := \frac{1}{g} \ell_{\Zp}(\coker(\phi_{\fM})).
\]
For general $\fM$, by Remark \ref{rem:height_twist} we choose any $r \in \bZ_{\geq 0}$ such that $\fM(r)$ is effective, and then set
\[
\deg(\fM) := \deg(\fM(r)) - r \cdot \rk(\fM). 
\]
The slope of $\fM$ is the quotient $\mu(\fM) := \deg(\fM)/\rk(\fM)$. 
\end{defn} 

\begin{eg} 
We compute rank and degree for $p$-power torsion quotients of $\fS(r)$. The cokernel is isomorphic to $\fS/E(u)^r$.  We have that $\rk(\fS/p^n(r)) = n$ and $\deg(\fS/p^n(r)) = rn$.  Thus, $\mu(\fS/p^n(r)) = r$. 
\end{eg}

\begin{rem} 
It will be shown (in  Proposition \ref{farguescomp}) that the degree of a Kisin module with height in $[0,1]$ is the same as the degree (as defined by Fargues in \cite[\S3]{fargues2010}) of the corresponding finite flat group scheme. 
\end{rem}

\begin{rem} 
By \cite[Lem.\ 1.2.7]{mffgsm}, the rank of $\cM \in \Mod_{\cO_\cE, \tor}^\phz$ is equivalent to the $\Zp$-length of the $\Gal(\overline{K}/K_{\infty})$-representation associated to $\cM$. 
\end{rem}

When a Kisin module is flat over $\fS/p^n$, its degree can be calculated from its top exterior power. 

\begin{lem} 
\label{wedgeslope}
Let $\fM \in \Mod_{\fS}^{\phz, [0,h]}$ be flat over $\fS/p^n$ of rank $d$. Then $\deg(\fM) = \deg(\wedge^d_\fS \fM)$. 
\end{lem}
\begin{proof}
Filtering $\fM$ by $p^i \fM$, we can reduce to the case where $n = 1$.  There exists basis for $\{ e_j \}$ for $\fM$ such that $\phi_{\fM}(\phz^*(\fM))$ is generated by $u^{m_j} e_j$.  Then $\deg(\fM) = \frac{1}{[k:\Fp]}  \sum m_j = \deg(\wedge^d_\fS \fM)$.  
\end{proof}

In \S\ref{sec:tensor}, we will be very interested in the behavior of Kisin modules under tensor product. This basic fact about how slope adds under tensor product will be foundational. 

\begin{prop}
\label{prop:ten_add}
Let $\fM, \fN \in \Mod_{\fS}^{\phz, [0,h]}$ be flat over $\fS/p^n$. Then $\mu(\fM \otimes_\fS \fN) = \mu(\fM) + \mu(\fN)$. 
\end{prop}
\begin{proof} 
Let $d$ (resp. $d'$) be the rank of $\fM$ (resp. $\fN$) over $\fS/p^n$. The result follows from Lemma \ref{wedgeslope} since
\[
\wedge^{d+d'}_\fS (\fM \otimes \fN) \cong (\wedge^d_\fS \fM)^{d'} \otimes (\wedge^{d'}_\fS \fN)^d. \qedhere
\] 
\end{proof}

\subsection{Harder--Narasimhan theory}
\label{subsec:HN-theory}

Following \cite{pottharst}, to achieve our goal of setting up a Harder--Narasimhan formalism on torsion Kisin modules relative to torsion \'etale $\phz$-modules, we observe that we have 
\begin{enumerate}[leftmargin=3em]
\item An additive faithful functor $GF$ to an abelian category, preserving strict short exact sequences.
\item A notion of rank on the abelian category, denoted $\rk$.
\item A notion of degree of a torsion Kisin module, denoted $\deg$.
\item A notion of closure from the generic fiber which induces a bijection on strict subobjects.  
\end{enumerate}

These structures must satisfy certain axioms, as follows; compare \cite[\S1]{pottharst}. 
 \begin{prop} 
 \label{prop:HN_main}
The notion of rank and degree of torsion Kisin modules given in Definition \ref{defn:rk_deg} define a Harder--Narasimhan theory on $\Mod^\phz_{\fS, \tor}$. That is, they satisfy the following properties.  
\begin{enumerate}
\item The functor $GF$ sending torsion Kisin modules to torsion \'{e}tale $\phz$-modules is exact and faithful, and induces a bijection
\[
\{\text{strict subobjects of } \fM\} \lrisom \{\text{subobjects of } \fM \otimes_\fS \cO_\cE\}
\]
\item Both $\rk$ and $\deg$ are additive in short exact sequences, and $\rk(M) = 0 \iff M = 0$. 
\item When an injection $\fM' \iarrow \fM$ in $\Mod^\phz_{\fS,\tor}$ induces an isomorphism $GF(\fM') \risom GF(\fM)$, then $\mu(\fM') \geq \mu(\fM)$, with equality if and only if $\fM' \risom \fM$. 
\end{enumerate}
\end{prop}
\begin{proof} 
Since $\cO_\cE$ is a flat $\fS$-module, the functor $GF$ is exact and faithful as desired, cf.~\cite[Prop.\ 2.1.12]{crfc}. Because Kisin modules are $u$-torsion free, $GF$ induces a bijection
\begin{align*}
\{\text{strict subobjects of } \fM\}& \lrisom \{\text{subobjects of } GF(\fM)\} \\
\fM' \subset \fM &\mapsto \fM' \otimes_\fS \cO_\cE \subset \fM \otimes_\fS \cO_\cE = GF(\fM)\\
\cM' \cap \fM \otimes 1 &\mapsfrom \cM' \subset GF(\fM).
\end{align*}
This proves (i).

Additivity in exact sequences of $\deg$ follows from the snake lemma, and additivity of $\rk$ follows from the Noetherianness of $\cO_\cE$. If $\rk(M) = 0$, then $M=0$ by Nakayama's lemma. This proves (ii). 

Now, let $f:\fM' \rinj \fM$ be an injective morphism which becomes an isomorphism after inverting $u$, that is, the cokernel $\fN$ is $u^s$-torsion for some $s \geq 1$.  Note that $\fN$ is also killed by some power of $p$, so that $\fN$ has finite $\bZ_p$-length. The sequence $0 \ra \phz^*(\fM') \ra \phz^*(\fM) \ra \phz^*(\fN) \ra 0$ is exact because the Frobenius map $\phz: \fS \ra \fS$ is flat. Let $\phi_{\fN}:\phz^*(\fN) \ra \fN$ be the induced map. We get an exact sequence
\[
0 \lra \ker (\phi_{\fN}) \lra \coker (\phi_{\fM'}) \lra \coker (\phi_{\fM}) \lra \coker(\phi_{\fN}) \lra 0.
\]
Thus,
\[
\deg (\fM') - \deg (\fM) = \ell_{\bZ_p} (\ker (\phi_{\fN}))  - \ell_{\bZ_p} (\coker (\phi_{\fN})) = \ell_{\bZ_p}(\phz^*(\fN)) - \ell_{\bZ_p} (\fN).
\]
The claim that $\deg(\fM') \geq \deg(\fM)$ with equality if and only if $f$ is an isomorphism reduces to Lemma \ref{lem:strict_ineq}.
\end{proof}

\begin{lem}
\label{lem:strict_ineq}
For any finite non-zero $\fS$-module $\fN$, we have $\ell_{\bZ_p}(\phz^*(\fN)) > \ell_{\bZ_p}(\fN)$.
\end{lem}

\begin{proof}
Choose a filtration of $\fN$ as a $\fS$-module that has simple graded pieces; each of these graded factors is isomorphic as a $\fS$-module to $k$. Because the Frobenius endomorphism of $\fS$ is flat, this filtration induces a filtration on $\phz^*(\fN)$ with graded pieces isomorphic to $\fS\otimes_{\phz, \fS} k$, which is in turn isomorphic to the $\fS$-module $k[u]/(u^p)$, which has length $p[k:\F_p]$. 
\end{proof} 

Now that we have verified the axioms of HN-theory for our notions of degree and rank on $\Mod^\phz_{\fS,\tor}$, standard results follow. To begin with, there are semi-stable objects. 

\begin{defn} 
An object $\fM \in \Mod_{\fS}^{\phz, [0,h]}$ is \emph{semi-stable} if for all non-zero strict subobjects $\fM' \subset \fM$, we have $\mu(\fM') \geq \mu(\fM)$.  
\end{defn}

\begin{rem}
Note that the inequality is opposite that of \cite[\S3, Cor.\ 3]{fargues2010}.  This is natural in light of the \emph{anti}-equivalence of Remark \ref{rem:anti-equiv}. 
\end{rem}

The following lemma allows us to always reduce to the case of height in $[0,h]$ by Tate twisting. 
\begin{lem} 
\label{sstt}
Let $\fM \in \Mod_{\fS}^{\phz, [a,b]}$.  Then $\fM$ is semi-stable if and only if $\fM(r)$ is semi-stable for any (equivalently, for every) integer $r$. 
\end{lem}  
\begin{proof}  For any strict subobject $\fN \subset \fM$, the Tate twist $\fN(r) \subset \fM(r)$ is again a strict subobject and this induces a bijection between strict subobjects.   Furthermore, $\mu(\fN(r)) = \mu(\fN) + r$ and similarly for $\fM$ so semi-stability is unaffected by the Tate twist.
\end{proof} 

One of the most useful consequences of the theory is that there exists a canonical filtration whose graded pieces are semi-stable objects, and their slopes are ordered. Note that the ordering is in the opposite direction from \cite{pottharst}. 

\begin{thm}[(HN-filtration)]
\label{thm:HNfilt} 
Let $\fM \in \Mod_{\fS,\tor}^{\phz}$.  There exists a unique filtration $0 = \fM_0 \subset \fM_1 \subset \fM_2 \subset \ldots \subset \fM_n = \fM$ by strict subobjects such that for all $i$, $\fM_{i+1}/\fM_i$ is semi-stable and $\mu(\fM_{i+1}/\fM_i) > \mu(\fM_i/\fM_{i-1})$.
\end{thm} 
\begin{proof}
See \cite[Prop.\ 6.3]{pottharst}. 
\end{proof}

Arguably, the most natural indexing for the HN-filtration is by slopes. 
\begin{prop}
\label{prop:HN_funct}
The HN-filtration is functorial when labeled by slopes. That is, if $\fM^{\leq \alpha}$ denotes the maximal step in the HN-filtration of $\fM \in \Mod^\phz_{\fS,\tor}$ with slope $\leq \alpha$, then for any morphism $\fM \ra \fN$ in $\Mod^\phz_{\fS,\tor}$, the image of $\fM^{\leq \alpha}$ lies in $\fN^{\leq \alpha}$. In particular, when all of the slopes of $\fM$ are strictly less than those of $\fN$, then $\Hom(\fM, \fN) = 0$. 
\end{prop}
\begin{proof}
A standard fact in theory of HN-filtrations (see \cite[Lem.\ 6.5]{pottharst}) is that if $\fN$ and $\fN'$ are both semi-stable and $\mu(\fN') < \mu(\fN)$ then $\Hom_{\phz} (\fN', \fN) = 0$.  By induction on the length of the HN-filtration of $\fM$, we are reduced to the case where $\fM_2 = \fM$ so that $\fM/\fM_1$ is semi-stable. The map induced by $f$ from $\fM_1$ to $\fM/\fM_1$ is zero since $\mu(\fM_1) < \mu(\fM/\fM_1)$ and both are semi-stable.  Thus, $f(\fM_1) \subset \fM_1$. 
\end{proof}

Here is one direct consequence of this functorality. 
\begin{cor} 
\label{cor:endo}
 Let $\fM \in \Mod_{\fS,\tor}^\phz$.  Then any $f \in \End_{\phi_{\fM}}(\fM)$ respects the HN-filtration of $\fM$.  
\end{cor}

\begin{defn}
\label{defn:et_rk}
Let $\fM \in \Mod_{\fS,\tor}^{\phz, [0,h]}$ be an effective Kisin module. As $0$ is the least possible slope of an effective $\fM$, there is a unique strict subobject $\fM^{\leq 0} \subset \fM$. We call $d^\et(\fM) := \rk(\fM^{\leq 0})$ the \emph{\'etale rank} of $\fM$. 
\end{defn}
One can check that $\fM^{\leq 0}$ is the maximal \'etale subobject of $\fM$. 

\begin{eg}
When in addition $\fM$ is effective of height $\leq 1$, the \'etale rank of $\fM$ is the same as the height of the \'etale part of the connected-\'etale exact sequence of the corresponding finite flat group scheme. We observe that our HN-theory generalizes the connected-\'etale exact sequence to Kisin modules. We give the proof that our HN-theory on $\Mod_{\fS,\tor}^{\phz, [0,1]}$ is anti-equivalent to Fargues's HN-theory in Proposition \ref{farguescomp}. 
\end{eg}

While the HN-axioms make $\Mod^\phz_{\fS,\tor}$ very close to an abelian category (see e.g.\ \cite[Prop.\ 2.7]{pottharst}), when one restricts to objects that are semi-stable of a given slope, this is actually an abelian category. 
\begin{prop}
\label{prop:ext_ss}
For any slope $\alpha$, the full subcategory of $\Mod_{\fS, \tor}^\phz$ consisting of semi-stable objects of slope $\alpha$ is an abelian, exact subcategory of $\Mod_{\fS, \tor}^\phz$ that is closed under extensions in $\Mod_{\fS, \tor}^\phz$. As $\Mod_{\fS, \tor}^{\phz, [a,b]}$ is closed under strict subobjects, quotients, and extensions, the same can be said for $\Mod_{\fS, \tor}^{\phz, [a,b]}$. 
\end{prop}
\begin{proof}
See \cite[Prop.\ 4.12]{pottharst}. 
\end{proof}

\begin{cor} 
\label{cor:subquotss}
Let $\fM \in \Mod_{\fS}^{\phz, [a,b]}$.  If $\fM$ is semi-stable, then any strict subobject $($resp. quotient$)$ $\fN$ such that $\mu(\fM) = \mu(\fN)$ is again semi-stable. 
\end{cor}   
\begin{proof}  
This is a direct consequence of Proposition \ref{prop:ext_ss}. Cf.\ \cite[Prop.\ 4.9]{pottharst}. 
\end{proof}

\subsection{The Harder--Narasimhan polygon}

The Harder--Narasimhan polygon, i.e.\ ``HN-polygon,'' of $\fM$ encodes numerical data about the pieces of the HN-filtration. 

\begin{defn} 
\label{HNpoly} 
 For any $\fM \in \Mod_{\fS, \tor}^\phz$, define the \emph{HN-polygon} of $\fM$,  $\mathrm{HN}(\fM)$, to be the convex polygon whose segments have slope $\mu(\fM_{i+1}/\fM_i)$ and length $\rk(\fM_{i+1}/\fM_i)$. The endpoints are $(0,0)$ and $(\rk(\fM), \deg(\fM))$. We will often represent the HN-polygon as a piecewise linear function
\[
\HN(\fM) : [0, \rk(\fM)] \lra [0, \deg(\fM)]
\]
such that $\HN(\fM)(0) = 0$ and $\HN(\fM)(\rk(\fM)) = \deg(\fM)$. 
\end{defn}

\begin{rem} 
Because of our normalization, the HN-polygon $\mathrm{HN}(\fM)$ need not have integer breakpoints; instead $g \cdot\mathrm{HN}(\fM)$ has integer breakpoints. 
\end{rem}

The HN-polygon also expresses a bound on the possible pairs $(\rk(\fM'), \deg(\fM'))$ that can occur for subobjects $\fM' \subset \fM$.
\begin{prop} 
\label{chull1} 
The HN-polygon of $\fM$ is the convex hull of the points 
\[
\{ (\rk(\fM'), \deg(\fM')) \} \subset \bR^2
\] 
for all strict subobjects $\fM'$ of $\fM$. It is also the convex hull of such points for all subobjects $\fM'$ of $\fM$. 
\end{prop} 
 
\begin{proof}
The first claim is equivalent to the statement that for any strict subobject $\fN \subset \fM$, the point $(\rk(\fN), \deg(\fN))$ lies above $\mathrm{HN}(\fM)$.  This follows from a standard induction argument (see e.g.\ \cite[Prop.\ 1.3.4]{DOR2010}). In light of Proposition \ref{prop:HN_main}(iii), one gets the same convex hull when all subobjects $\fM' \subset \fM$ are allowed. 
\end{proof}

\subsection{Base change}
\label{subsec:base_change}

We work modulo $p$ for the remainder of the section. We write the category of $p$-torsion Kisin modules as $\Mod_{\fS_{\Fp}}^{\phz}$, i.e., they are Kisin modules over $\fS_{\bF_p} \cong k\lb u\rb$. It will be useful in \S\ref{sec:tensor} to be able to calculate degrees of Kisin modules ``after base change'' in the following sense. 

Let $F'$ be a finite extension of $F = k(\!(u)\!)$ of degree $h$. Let $\fS'_{\F_p}$ denote the valuation ring of $F'$, which is finite free over $\fS/p\fS$ (see e.g.\ \cite[Prop.\ II.3]{serre1979}) and compatibly equipped with a Frobenius $\varphi'$.   There is a tensor exact functor
$$
\mathrm{BC}_{F'/F}: \Mod_{\fS_{\Fp}}^{\phz} \ra \Mod_{\fS'_{\F_p}}^{\phz}
$$
given by $\fM \mapsto (\fM \otimes_{\fS_{\Fp}} \fS'_{\F_p}, \phi_{\fM} \otimes \varphi')$.   

Set $g' = gh$.  For any effective $\fM' \in \Mod_{\fS'_{\F_p}}^{\phz}$, define $\deg(\fM') = \frac{1}{g'} \ell_{\bZ_p}(\coker \phi_{\fM'})$. 
\begin{prop} \label{prop:BCdeg} Let $\fM \in \Mod_{\fS_{\Fp}}^{\phz}$ be an effective $p$-torsion Kisin module.  Then 
$$
\deg(\fM) = \deg(\mathrm{BC}_{F'/F}(\fM)).
$$
\end{prop}
\begin{proof} Since $\fS'_{\F_p}$ is flat over $\fS_{\F_p}$, the tensor product commutes with the formation of cokernels, using that $\phz^*(\fM) \otimes_{\fS_{\F_p}} \fS'_{\F_p} \cong (\phz')^*(\mathrm{BC}_{F'/F}(\fM))$. 
\end{proof}

\section{The tensor product theorem}
\label{sec:tensor}

\subsection{Tensor products of semi-stable Kisin modules}
\label{subsec:ten_sum}

Here is the main result of this section. 

\begin{thm}
\label{thm:tensor}
Let $\fM, \fN$ be semi-stable torsion Kisin modules of slope $\mu(\fM)$ and $\mu(\fN)$ respectively. Then $\fM \otimes_\fS \fN$ is semi-stable of slope $\mu(\fM) + \mu(\fN)$. 
\end{thm} 

Let us observe immediately that it suffices to prove Theorem \ref{thm:tensor} in a more narrow case.
\begin{lem}
\label{lem:p->all}
If Theorem \ref{thm:tensor} is true for effective $p$-torsion Kisin modules $\fM$ and $\fN$, then its full statement is true.
\end{lem}

\begin{proof}
By twisting by $\fS(r)$ as discussed in Lemma \ref{sstt}, it suffices to restrict ourselves to the case of effective Kisin modules. 

The filtration of a torsion Kisin module by its submodules $p^n \fM$ is decreasing, finite, and exhaustive. Corollary \ref{cor:subquotss} implies that when $\fM$ is semi-stable, all of the graded pieces of this filtration are semi-stable of identical slope. Assuming Theorem \ref{thm:tensor} for $p$-torsion effective Kisin modules, the tensor products $\gr^i \fM \otimes_\fS \gr^j \fN$ are $p$-torsion, effective, and semi-stable of slope $\mu(\fM) + \mu(\fN)$. Next, notice that direct sums of such tensor products surject onto the graded pieces of the filtration on $\fM \otimes_\fS \fN$ by $p^n (\fM \otimes_\fS \fN)$. Applying Corollary \ref{cor:subquotss}, we deduce that the graded pieces of $\fM \otimes_\fS \fN$ are semi-stable and all of them have slope $\mu(\fM) + \mu(\fN)$. Finally, by Proposition \ref{prop:ext_ss}, this implies that $\fM \otimes_\fS \fN$ is semi-stable. 
\end{proof}

\subsubsection{Notation and conventions}

To summarize the proof of Theorem \ref{thm:tensor} in the effective $p$-torsion case, it will be convenient to establish the standard notation that we will use in this section. We write $\cO$ for $k\lb u\rb$ and write $F$ for $k \lp u \rp$, the fraction field of $\cO$. Then $\fM$, $\fN$ are projective $\cO$-modules with the additional data of the semi-linear finite height endomorphism. 

We will let $\cM := \fM \otimes_\cO F$ and $\cN := \fN \otimes_\cO F$, which we consider to be $p$-torsion \'etale $\phz$-modules. The notation $\fM_0$ and $\fN_0$ will refer to various choices of $\cO$-lattices in $\cM$ and $\cN$ respectively. These will not necessarily be $\phz$-stable. 

Often, we will work with $\cO$-lattices modulo $u$, especially the $\cO$-lattices $\fM_0$ and $\fN_0$. We denote by $\overline{\fM}_0$ the finite dimensional $k$-vector space $\fM \otimes_\cO k$, and, similarly, $\overline{\fN}_0 := \fN \otimes_\cO k$. 

We will use undecorated tensor products $\cM \otimes \cN$, $\fM \otimes \fN$, and $\overline{\fM}_0 \otimes \overline{\fN}_0$ to denote the tensor product over $F$, $\cO$, and $k$ respectively. 

\subsubsection{Reduction steps}

To prove Theorem \ref{thm:tensor}, we must show that $\mu(\fQ') \geq \mu(\fM \otimes \fN)$ for all subobjects (i.e.~sub-Kisin modules) $\fQ' \subset \fM \otimes \fN$. We note that it suffices to restrict the $\fQ'$ under consideration to \emph{strict} subobjects, because any subobject $\fQ'$ is contained in a strict subobject $\fQ$, and $\mu(\fQ') \geq \mu(\fQ)$. Because all strict subobjects $\fQ \subset \fM \otimes \fN$ arise from a sub-\'etale $\phz$-module $\cS \subset \cM \otimes \cN$ by intersection as 
\[
\fQ = (\cS \cap \fM \otimes \fN) \quad \text{in } \cM \otimes \cN,
\]
it will suffice to prove $\mu(\cS \cap \fM \otimes \fN)$ for all sub-\'etale $\phz$-modules of $\cM \otimes \cN$. We will often study the reduction of this Kisin module modulo $u$ according to the lattice $\fM_0 \otimes \fN_0 \subset \cM \otimes \cN$, which we denote by
\[
\overline{\cS}_0 := (\cS \cap \fM_0 \otimes \fN_0) \otimes_\cO k,
\]
where the intersection is taken in $\cM \otimes \cN$. 

\subsubsection{Strategy of proof of Theorem \ref{thm:tensor}}

With the notation in place, we will summarize the proof of Theorem \ref{thm:tensor} in the $p$-torsion case. First, we relate the notion of degree of subobjects $\cS \cap \fM \otimes \fN \subset \fM \otimes \fN$ to the degrees of subspaces of filtered $k$-vector spaces. For clarity in exposition, we will sometimes call the latter degree the ``filtration degree'' of $\cS$. Specifically, we define a filtration on $\overline{\fM}_0$ and $\overline{\fN}_0$ such that the degree of $\cS \cap \fM \otimes \fN$ is bounded below by the degree of $\overline{\cS}_0$ as a subspace of $\overline{\fM}_0 \otimes \overline{\fN}_0$ with respect to filtration coming from its tensor factors. 

Having related the degree of a Kisin module to the degree of some subspace of a tensor product vector space, we follow the approach of Totaro \cite{totaro1996}, which relies on geometric invariant theory to compare the slope of the subspace to the slope of the whole tensor product. However, there are additional obstacles to overcome. 

Totaro's approach relies on the notion of a \emph{semi-stable} subspace of a tensor product vector space. While Totaro worked with filtered $\phz$-modules, which lie over $p$-adic fields, we must relate the semi-stability of $\cS$ as a $F$-linear subspace of $\cM \otimes \cN$ to the semi-stability of $\overline{\cS}_0$ as a $k$-linear subspace of $\overline{\fM}_0 \otimes \overline{\fN}_0$. Indeed, Totaro's argument relies on dealing with subobjects that are semi-stable as a vector subspace, and relating degree to semi-stability. But in our case, unlike Totaro's, subobjects live on the generic fiber (i.e.~in $\cM \otimes \cN$) while the degree is computed elsewhere -- on the special fiber (i.e.~in $\overline{\fM}_0 \otimes \overline{\fN}_0$). There are many choices of $\fM_0$, $\fN_0$ such that $\cS$ is semi-stable in $\cM \otimes \cN$ but $\overline{\cS}_0$ is not semi-stable in $\overline{\fM}_0 \otimes \overline{\fN}_0$. We resolve this problem by proving in Corollary \ref{cor:LT_ss} that 
\begin{center}
if $\cS$ is semi-stable, then there exists a choice of $(\fM_0, \fN_0)$ such that $\overline{\cS}_0$ is semi-stable. 
\end{center}
We dub this new result to be of ``Langton type,'' after the work of Langton showing that a generically semi-stable vector bundle over a curve admits an integral model such that its special fiber is also semi-stable \cite{langton1975}. 

In the case that $\cS$ is not semi-stable in $\cM \otimes \cN$, Totaro's strategy relies on studying the resulting \emph{Kempf filtration} on the $F$-vector spaces $\cM$ and $\cN$. This canonical filtration maximizes the failure of semi-stability of $\cS$. It may be readily shown that the Kempf filtration consists of sub-\'etale $\phz$-modules of $\cM$ and $\cN$. In this case, we must prove a Langton type result for the Kempf filtration, Theorem \ref{thm:LT_kempf}. There is a final, additional complication to carry out our adaptation of Totaro's strategy in the case that $\cS$ is not semi-stable. Indeed, the defect between degree and filtration degree works in our favor when $\cS$ is semi-stable; but when $\cS$ is not semi-stable, it works against us. Consequently, additional work is required to eliminate the defect in that case. 

\subsection{Degree and filtrations}
\label{subsec:alt}

In this section, we derive a bound (Proposition \ref{prop:alt_deg}) on the degree of a sub-Kisin module $\fN \subset \fM$ contained in an \'etale $\phz$-module $\cM = \fM \otimes_\cO F$. Namely, we bound $\deg(\fN)$ by the degree of $\overline{\fN}_0$ as a subspace of $\overline{\fM}_0 = \fM_0/ u \fM_0$ according to a certain filtration on $\overline{\fM}_0$. Here $\fM_0$ is an \emph{arbitrary} $\cO$-lattice in $\cM$. In particular, $\fM_0$ need not be $\phz$-stable, i.e., it is not a Kisin module. In order to prove Theorem \ref{thm:tensor}, it will be useful to vary $\fM_0$. 

We will work with finite-dimensional $k$-vector spaces $V$ with increasing exhaustive filtrations indexed by the rational numbers exactly as in \cite[\S2]{totaro1996}: $V_\alpha^i \subset V_\alpha^j$ for $i \leq j$, $V_\alpha^i = 0$ for $i \ll 0$, $V_\alpha^i = \cap_{j \in \bQ_{> i}} V_\alpha^j$, and $V_\alpha^i = V$ for $i \gg 0$. Let $\gr_\alpha^i V := V_\alpha^i/V_\alpha^{i-\varepsilon}$ for small enough $\varepsilon > 0$. This formulation of $\gr_\alpha^i V$ makes sense for all $i \in \bR$, and we insist that $\gr_\alpha^i V \neq 0$ only for $i \in \bQ$. All filtrations will be assumed to be increasing filtrations indexed by $\bQ$, without further comment. 

We recall the notions of slope and degree of subspaces of such a filtered vector space. The \emph{degree} of a subspace $S$ of a vector space $V$ with filtration $\alpha$ is given by 
\[
\deg_\alpha(S) = \sum_{i \in \bQ} i \dim_k \gr_\alpha^i S,
\]
where $S$ is given the natural filtration of a subspace of a filtered vector space. The slope of $S$, written $\mu_\alpha(S)$, is then given by $\deg_\alpha(S)/\dim_k (S)$. 

Let $\cM$ be a $p$-torsion \'etale $\phz$-module. This first definition depends only on the fact that $\cM$ is a finite dimensional $F$-vector space. 

\begin{defn} 
\label{defn:u_filt}
For any two $\cO$-lattices $\fM ,\fL$ in $\cM$, we define an increasing sequence of sublattices $\cF^{i}_{\fL}(\fM)$ of $\fM$, 
\[
\cF^{i}_{\fL}(\fM) = \fM \cap u^{-i} \fL \text{ for } i \in \bZ.
\]
An increasing filtration on $\overline{\fM} := \fM/u \fM$ is then given by $\Fil^i_{\fL} (\overline{\fM}) = \mathrm{Im}(\cF^{i}_{\fL}(\fM))$.

Finally, we define $\deg_{\fL}(\fM)$ to be the degree of $\overline{\fM}$ according to this filtration, that is,
\[
\deg_{\fL}(\fM) = \sum_i i \dim_{\F} \gr^i_{\fL} (\overline{\fM}). 
 \]
\end{defn} 

 It is easy to see that if $\fM, \fL$ are in relative position $(d_1, d_2, \ldots, d_n)$, then $\deg_{\fL}(\fM) = \sum_j d_j$.  In particular, if $\fM \supset \fL$, then $\deg_{\fL}(\fM) = \dim_k \fM/ \fL$.  

The following lemma describes the behavior of this notion of degree for strict subobject of $\fM$, in comparison with this subobject's interaction with the filtration on $\overline{\fM}$. 

\begin{lem} 
\label{lem:intcomp} 
Let $\fM, \fL$ be $\cO$-lattices in $\cM$.  Let $\cS \subset \cM$ be a subspace.   Define $\cO$-lattices
\[
S_{\fM} = \cS \cap \fM, \quad S_{\fL} = \cS \cap \fL
\]
in $\cS$ and consider the induced filtration $\Fil^i_{\fL} (\overline{\cS}_{\fM}) = \overline{\cS}_{\fM} \cap \Fil^i_{\fL}(\overline{\fM})$ on $\overline{\cS}_\fM$.   Then 
\[
\deg_{\cS_{\fL}} (\cS_{\fM}) \geq \sum_i i \dim_{k} \gr_{\fL}^i (\overline{\cS}_{\fM}).
\]  
Furthermore, if there exists a splitting $s:\fM \ra \cS_{\fM}$ of $\cS_{\fM} \subset \fM$ such that $s(\fL) \subset \cS_{\fL}$, then the above inequality is an equality. 
\end{lem}

\begin{proof}
Choose a basis $\{e_1, \ldots, e_k \}$ for $S_{\fM}$ such that the collection $\{ u^{d_i} e_i \}$ generates $S_{\fL}$.  Then $\deg_{S_{\fL}} (S_{\fM}) = \sum_i d_i$.  A straightforward computation shows $\overline{e}_i \in \Fil^{d_i}_{\cL} (\overline{\fM})$.   In particular, we have a natural map  $(\overline{S}_{\fM}, \Fil_{S_{\fL}}) \ra (\overline{S}_{\fM}, \Fil_{\fL})$ of filtered vector spaces inducing the desired inequality, with equality if and only if the map is strict, i.e., $\gr^i_{S_{\fL}}( \overline{S}_{\fM}) \cong \gr^i_{\fL}( \overline{S}_{\fM})$. 

Let $m \in S_{\fM}$.   For the map to be strict, we need that if $u^d m \in \fL$ then $u^d m \in S_{\fL}$ for any integer $d$.  This follows easily if there exists such a splitting.      
\end{proof}

Fix an $\cO$-lattice $\fM_0 \subset \cM$, which is not required to be $\phz_\cM$-stable. Also fix a $p$-torsion Kisin module $\fM \subset \cM$ whose $\phz_\fM$ is induced by $\phz_\cM$. The following key result bounds the degree of strict sub-Kisin modules of $\fM$ by the linear-algebraic degree of the induced subspaces of $\overline{\fM}_0$. 

\begin{prop} 
\label{prop:alt_deg}
Let $\fM_0$ and $\fM$ be as above.  Let $\fL_0 = \phi_{\cM}(\phz^*\fM_0)$, a sublattice of $\cM$. For any sub-\'etale $\phz$-module $\cS \subset \cM$,  we have 
\begin{equation} 
\label{eq:2filt}
\deg(\cS_{ \fM}) \geq \frac{1}{e} \left(\deg_{\fL_0} (\cS \cap \fM_0) + (p-1) \deg_{\fM}(\cS \cap \fM_0)\right). 
\end{equation}
\end{prop} 
\begin{proof}
Let $\cS_{\fM_0} = \cS \cap \fM_0$ and similarly for $\fL_0$.  Let $n$ be the rank of $\cS$, and choose a basis $\beta$ for $\cS_{\fM_0}$.  We can then write $\cS_{\fM} = x \cdot \cS_{\fM_0}$ for some $x \in \GL_n(k(\!(u)\!))$. If $A_0 \in \GL_n(k(\!(u)\!))$ is the matrix for Frobenius on $\cS_{\fM_0}$ with respect to $\beta$, then by Proposition \ref{wedgeslope}, $\deg(\cS_{ \fM_0}) = \frac{1}{e} \val_u(\det(A_0))$.   Semi-linear change of basis says that the matrix for Frobenius on $\cS_{\fM}$ with respect to $x(\beta)$ is given by $A = x^{-1} A_0 \phz(x)$.   Since $\val_u(\det(\phz(x))) = p \val_u(\det(x))$, we get
\[
\deg(\cS_{\fM}) = \deg(\cS_{ \fM_0}) + \frac{p-1}{e} \val_u(\det (x)).
\]
However, $ \val_u(\det (x))$ is just the sum of powers of the elementary divisors of $\cS_{\fM}$ relative to $\cS_{\fM_0}$ and so we have
\[
\deg(\cS_{\fM}) = \frac{1}{e} \left( \deg_{\cS_{\fL_0}}(\cS_{ \fM_0}) + (p-1) \deg_{\cS_{\fM}} (S_{\fM_0})\right).
\]
Applying Lemma \ref{lem:intcomp} to both factors on the right, we arrive at (\ref{eq:2filt}).
\end{proof}

Proposition \ref{prop:alt_deg} provides a new angle of attack at proving that the tensor product of semi-stable Kisin modules is semi-stable. Namely, we can bring to bear existing results on degrees of subspaces of a tensor product. 

\subsection{Semi-stability of subspaces of tensor products of vector spaces}
\label{subsec:ss_kempf}

For \S\S\ref{subsec:ss_kempf}-\ref{subsec:two}, we will work entirely with vector subspaces and their moduli, returning to Kisin modules in \S\ref{subsec:tensor_conclusion} to finish the proof of Theorem \ref{thm:tensor}. Throughout \S\S\ref{subsec:ss_kempf}-\ref{subsec:two}, we will work over a field $k$, which can be any perfect field, including the finite field that $k$ represents in this whole paper. 

In this section, we recall the theory of semi-stable subspaces of a tensor product $M \otimes_k N$ of $k$-vector spaces $M$, $N$, following Totaro \cite[\S2]{totaro1996}. Loosely speaking, semi-stable $S$ are ``generic'' subspaces of $M \otimes_k N$. Non-semi-stable $S$ are also called ``unstable.'' Unstable $S$ are closer to ``special'' subspaces, for example those subspaces of $M \otimes_k N$ of the form $M' \otimes N'$ for subspaces $M' \subset M$, $N'\subset N$. 

We remark that we will work with increasing filtrations on vector spaces indexed by $i \in \bQ$, in contrast with decreasing filtrations as in \cite{totaro1996}. 

We also note that when $M$ and $N$ are vector spaces and $\alpha$ is a filtration on $(M,N)$ (meaning that $\alpha$ is the data of increasing filtrations on both $M$ and $N$), the tensor product takes on the filtration by $\ell \in \bQ$ 
\[
(M \otimes_k N)_\alpha^\ell = \sum_{i+j = \ell} M_\alpha^i \otimes N_\alpha^j.
\]
Recall from \S\ref{subsec:alt} that $\alpha$ then induces a filtration on any subspace $S$ of $M \otimes_k N$, which we will label as $S^\ell_\alpha$. 

\begin{defn}
Let $S$ be a subspace of $M \otimes_k N$, where $M$ and $N$ are finite-dimensional $k$-vector spaces. We call $S$ \emph{semi-stable} if for all filtrations $\alpha$ on $(M,N)$, $\mu_\alpha(S) \geq \mu_\alpha(M \otimes_k N)$. Otherwise, we call $S$ \emph{unstable}. 
\end{defn}

In other words, $S$ is semi-stable if it is ``generic enough,'' i.e., its intersections with subspaces of $M \otimes_k N$ the form $M' \otimes N'$, where $M' \subset M$ and $N' \subset N$, tend to be small. 

When $S \subset M \otimes_k N$ is unstable, there is a unique filtration on $(M,N)$ that maximizes the failure of semi-stability (\cite[Prop.~1]{totaro1996}). This is known as the \emph{Kempf filtration} associated to $S$. The Kempf filtration's indices will only be defined up to scaling; indeed, we consider filtrations $\alpha, \beta$ on $(M, N)$ to be (scalar) equivalent if there exists $x \in \bQ_{> 0}$ such that $M^{ix}_\alpha = M^i_\beta$ and $N^{ix}_\alpha = N^i_\beta$ for all $i \in \bQ$. In order to normalize degrees with respect to equivalent filtrations, the following natural notion of size is natural: 
\[
\abs{\alpha} := \bigg(\sum_{i \in \bQ} i^2 \dim \gr^i_\alpha M + \sum_{j \in \bQ} j^2 \dim \gr^j_\alpha N\bigg)^{1/2}.
\]
We observe that $\abs{\alpha}$ is non-zero when $\alpha$ is not trivial.

\begin{prop}
\label{prop:kempf}
Let $M,N$ be finite dimensional vector spaces over a field $\bK$ (perhaps not perfect). Let $S \subset M \otimes_\bK N$ be an unstable subspace. There is a unique equivalence class of filtrations of $(M, N)$, the \emph{Kempf filtration}, such that 
\begin{equation}
\label{eq:mu_diff}
f(S, \alpha) = \frac{\mu_\alpha(M \otimes_\bK N) - \mu_\alpha(S)}{\abs{\alpha}}
\end{equation}
is maximized over the set of non-trivial filtrations $\alpha$ on $(M,N)$ precisely by the Kempf filtration. Moreover, $\mu_\alpha(M) = \mu_\alpha(N) = 0$. 
\end{prop}

\begin{proof}
See the elementary argument of Totaro \cite[Prop.~1]{totaro1996}, adapting the GIT arguments of Kempf \cite{kempf1978} and Ramanan and Ramanathan \cite{RR1984}. 
\end{proof}
 
\subsection{Two Langton type results}
\label{subsec:LT_overview}

Now that we have defined semi-stability of subspaces of tensor product vector spaces, let us overview what we wish to prove about vector spaces and lattices. The full statements appear in \S\ref{subsec:two}. 

Let $k$ be a perfect field, and let $M,N$ be $k$-vector spaces. Let $F = k\lp u\rp$ and let $\cO = k\lb u \rb$. Our goal is to show that the semi-stability of a subspace $\cS \subset M_F \otimes_F N_F$, or its manner of failure as measured by the Kempf filtration, can be preserved under some specialization of $\cS$. More precisely, consider how a choice of $\cO$-lattices $\fM \subset M_F$, $\fN \subset N_F$ induces a sublattice $\cS \cap \fM \otimes_\cO \fN$ and consequently a reduction modulo $u$, $\overline{ \cS} \subset M \otimes_k N$. Our main results are 
\begin{enumerate}
\item if $\cS$ is semi-stable, then there exists some reduction $\overline{ \cS}$ that is semi-stable (Corollary \ref{cor:LT_ss}).
\item if $\cS$ is unstable with Kempf filtration $\alpha$ on $(M_F, N_F)$, then after passing to the associated graded $F$-vector spaces of $M_F$ and $N_F$, there exists a reduction $\overline{ \cS}^{\mathrm{ss}}$ that has the same Kempf filtration (Theorem \ref{thm:LT_kempf}).
\end{enumerate}

We call these results ``Langton type'' because Langton proved the original result like (i) in the setting of the moduli of vector bundles on an algebraic curve \cite{langton1975}. The goal of the next few sections is to recall results from GIT from which we will deduce (i) and (i) as corollaries. 

\subsection{Background in geometric invariant theory and instability theory}
\label{subsec:GIT}

We continue using the notation of \S\ref{subsec:LT_overview}. The notion of semi-stability of $S \subset M \otimes_k N$ can be studied geometrically when one considers it to be a property of a point of the moduli space of $s$-dimensional vector subspaces of $M \otimes N$, i.e., the Grassmannian $\Gr_s(M \otimes_k N)$. This coincides with the notion of semi-stable and unstable points in geometric invariant theory (GIT). In this section, we will provide background in GIT. Then, in \S\ref{subsec:GrassLT}, we will return to the particular case of the Grassmannian. 

Firstly, we will need to understand the relation between quotient stacks with stability properties and their associated \emph{coarse moduli schemes}. This theory was developed by Mumford \cite{mumford1965}, and we will often refer to \cite{alper2013, alper2014} for the geometric aspects of the theory that we particularly require. 

Secondly, we will require a stratification of the quotient stack by a measure of instability which, in the case of the Grassmannian, will be identical to the measure of \eqref{eq:mu_diff}. This stratification is known as the \emph{Kempf--Ness stratification}. The references are \cite{kempf1978, hesselink1978, kirwan1984, ness1984}. The instability results of \cite{kempf1978} were stated in the form of a stratification in \cite{hesselink1978}, and we will mainly refer to \cite{kirwan1984} to reference those works. Our exposition follows that of \cite{HL2014}. 

Assume we have these data. 
\begin{itemize}[leftmargin=2em]
\item A projective $k$-scheme $X$, where $k$ is a perfect field. 
\item An action of a reductive $k$-algebraic group $G$ on $X$. 
\item An ample line bundle $\cL$ on $X$ that is $G$-linearized, i.e.~there is a $G$-action on $\cL$ (by some character of $G$) that is equivariant for the $G$-action on the base $X$. 
\end{itemize}

We are interested in understanding the quotient algebraic stack $\fX := [X/G]$. A point $x \in X$ is called \emph{semi-stable} relative to $\cL$ when there exists a $G$-equivariant global section $\sigma$ of $\cL^{\otimes n}$ (for some $n > 0$) that does not vanish at $x$. To be precise, we write $\sigma\vert_x$ for the value of $\sigma$ at $x$, which is an element of $\cL^{\otimes n}\vert_x := \cL^{\otimes n}_x/\fm_x \cL^{\otimes n}_x$; $\sigma$ is said to vanish at $x$ when $\sigma\vert_x = 0$. This results in an open locus $X^{ss}_\cL \subset X$ called the semi-stable locus. As this locus is $G$-stable, there is also the semi-stable quotient stack $\fX^{ss}_\cL = [X^{ss}_\cL/G]$. 

In general, $\fX$ and $\fX^{ss}_\cL$ are not schemes. GIT yields a \emph{coarse moduli scheme} for the semi-stable locus. By definition, a coarse moduli scheme for $\fX^{ss}_\cL$ is the best-possible quotient \emph{scheme} approximating $\fX^{ss}_\cL$ (see e.g.\ \cite[Thm.\ 1.10]{mumford1965}). In our setting, the  projective $k$-scheme 
\[
X^{ss}_\cL /\!/ G := \Proj \bigoplus_{n\geq 0} H^0(X, \cL^{\otimes n})^G, 
\]
is the coarse moduli scheme. 

\begin{thm}
\label{thm:AMS}
Let $X$, $G$, and $\cL$ be as above. There exists a morphism $\phi: \fX^{ss}_\cL \ra X^{ss}_\cL /\!/ G$ that is uniquely characterized by the property that any morphism from $\fX^{ss}_\cL$ to a separated scheme factors uniquely through $\phi$. In particular, $\phi$ is universally closed. 
\end{thm}

Any morphism $\phi$ from a stack to a separated scheme with the universal property of the theorem is called the \emph{coarse moduli space} of the stack. The crucial property of coarse quotients, for our purposes, is that the morphism $\phi$ from the stack to its coarse space is an \emph{adequate moduli space morphism}. These morphisms are defined and studied in \cite{alper2013, alper2014}. Such morphisms are universally closed and ``weakly proper'' (as discussed in \S\ref{subsec:val}), by \cite[Main Thm.]{alper2014}. 

Here is the main input we need from this theory. 
\begin{cor}
\label{cor:LT_generic}
Let $X$, $G$, and $\cL$ be as above. Then $\fX^{ss}_\cL$ is an adequate moduli space over the projective $k$-scheme $X^{ss}_\cL/\!/ G$.
\end{cor}

\begin{proof}
See \cite[p.~490]{alper2014}.
\end{proof}

Having discussed the semi-stable locus, we now discuss the Kempf--Ness stratification of the unstable locus $\fX^{us}_\cL := \fX \setminus \fX^{ss}_\cL$ in $\fX$. Assume that $G$ is split for simplicity. By the Hilbert--Mumford numerical criterion (\cite[Thm.~2.1]{mumford1965}), for any unstable geometric point $x$ of $X$, there exists a cocharacter $\alpha: \bG_m \ra G$ such that 
\begin{itemize}[leftmargin=2em]
\item the limit $y := \lim_{t \ra 0} \alpha(t) \cdot x$ exists, and 
\item the character in $X^*(\bG_m) \cong \bZ$ by which $\bG_m$ acts via $\alpha$ on $\cL\vert_y$ is negative. 
\end{itemize}
Notice that $\bG_m$ acts on $\cL\vert_y$ by a character because $y$ is in the fixed locus $X^\alpha$ of $\alpha$ in $X$. The integer $a$ such that this character is $t \mapsto t^a$ will be written ``$\weight_\alpha \cL\vert_y$.'' 

Define a norm on cocharacters $\alpha$ of $G$ by $\abs{\alpha} = \sqrt{\langle \alpha, \alpha \rangle}$, where $\langle\,, \rangle$ is a Weyl group-invariant positive definite bilinear form on the cocharacter lattice of $G$. Then, for any $y \in X^\alpha$, we define the slope 
\[
\mu_\cL(y, \alpha) := \frac{-\weight_\alpha \cL\vert_y}{\abs{\alpha}}. 
\]
Because $\mu_\cL(\cdot, \alpha)$ is constant on any connected component $Z \subset X^\alpha$, we also write $\mu_\cL(Z, \alpha)$ for this value. Write $Y_{Z,\alpha}$ for the locus of points $x$ in $X$ such that the limit $\lim_{t \ra 0} \alpha(t)\cdot x$ exists and lies in $Z$. In this context, we write $\pi$ for the projection to the limit $\pi: Y_{Z,\alpha} \ra Z$. For $x \in Y_{Z,\alpha}$, we define $\mu_\cL(x,\alpha) := \mu_\cL(\pi(x),\alpha)$. 

Now define $M_\cL: X(\bar k) \ra \bR$ to be the maximal slope function 
\[
M_\cL(x) := \sup \{ \mu_\cL(y, \alpha) \mid \alpha: \bG_m \ra G, y = \lim_{t \ra 0} \alpha(t) \cdot x \text{ exists}\}.
\]
By the Hilbert--Mumford criterion, $M_\cL(x)$ is positive if and only if $x \in X(\bar k)$ is an unstable point. The Kempf--Ness stratification on $X$ filters its points by the value of $M_\cL$. 

We construct the Kempf--Ness stratification by following this inductive procedure: Select a pair $(Z_i, \alpha_i)$, where $\alpha_i : \bG_m \ra G$ is a cocharacter, $Z_i \subset X^{\alpha_i}$ is a connected component, $Z_i$ is not contained in any of the previously defined strata, and $\mu_i :=\mu_\cL(Z_i, \alpha_i)$ is positive and maximized among all such pairs. Write $Z_i^\circ$ for the open locus in $Z_i$ not intersecting existing strata, and let $Y_{Z_i, \alpha_i}^\circ := \pi^{-1}(Z_i^\circ)$. The strata indexed by $i$ is then $S_i := G \cdot Y^\circ_{Z_i, \alpha_i}$. Notice that the sequence $\mu_i$ is decreasing, but not necessarily strictly. The procedure terminates -- see Theorem \ref{thm:strata}. 

Let $L_\alpha$ denote the Levi subgroup of $G$ centralizing $\alpha$, and set $L_i = L_{\alpha_i}$. Likewise, denote by $P_{\alpha_i}$ the parabolic subgroup associated to $\alpha_i$, i.e.~the subgroup of $G$ generated by $L_i$ and the positive roots relative to $\alpha_i$.  Then, for $Z_i$ as above, let $P_i := \{p \in P_{\alpha_i} \mid \varpi(p) \cdot Z_i \subset Z_i\}$, where $\varpi : P_{\alpha_i} \rsurj L_{\alpha_i}$ is the usual projection. Then $L_i$ acts on $Z_i$, $P_i$ acts on $Y_i$, and the projection map $\pi: Y_i \ra Z_i$ intertwines these actions. As a result, there is a morphism 
\begin{equation}
\label{eq:ev_0}
ev_0: [Y_i/P_i] \lra [Z_i/L_i] 
\end{equation}
induced by $\pi$. 

The theory of the Kempf--Ness stratification yields 
\begin{thm}
\label{thm:strata}
Assume that $X$ is smooth. The stratification of $X$ into $X^{ss}$ and the set $\{S_i\}$ consists of finitely many locally closed strata defined over $k$, where any union of strata $\bigcup_{i \leq i_0} S_i$ is closed. Moreover, 
\[
\overline{S_i} \subset S_i \cup \bigcup_{\mu_j > \mu_i} S_j
\]
and the canonical morphism $G \times_{P_i} Y_i^\circ \ra S_i$ is an isomorphism for all $i$, yielding an isomorphism of quotient stacks $[S_i/G] \simeq [Y_i^\circ/P_i]$. 
\end{thm}

\begin{proof}
See \cite[Thm.\ 13.5]{kirwan1984} for the case that $k = \bar k$. The additional properties of the stratification are described in \cite[\S12]{kirwan1984}. As remarked on \cite[p.\ 144]{kirwan1984}, when $k$ is perfect, it follows from \cite{hesselink1978} that the strata are defined over $k$. 
\end{proof}

The final fact we record is an analogue of Corollary \ref{cor:LT_generic} with $[Z_i/L_{\alpha_i}]$ in place of $\fX = [X/G]$. In particular, we want to cut  out the open locus $Z^\circ_i$ as the semi-stable locus 
\[
Z^\circ_i = (Z_i)^{ss}_{\cL_i}
\]
for an appropriate $L_{\alpha_i}$-linearized ample line bundle $\cL_i$ on the projective $\bF$-scheme $Z_i$. Recall that $Z^\circ_i$ is the complement in $Z_i$ of Kempf--Ness strata with higher instability measure. 

\begin{prop}
\label{prop:ss=kempf_general}
Using the identification of $X_*(T)$ with $X^*(T)$ arising from $\langle\,, \rangle$, there exists a positive integer $r$ such that $r \alpha_i$ corresponds to a character of $T$ that extends to a character $\chi_i$ of $L_i$. Let $\cL_i := \cL^{\otimes r}\vert_{Z_i}  \otimes \chi_i^{-1}$. Then $Z^\circ_i = (Z_i)^{ss}_{\cL_i}$.
\end{prop}
\begin{proof}
See \cite[Rem.\ 12.21]{kirwan1984} and also \cite[Thm.\ 9.4]{ness1984}.  
\end{proof}
Notice that the restriction of $\cL^{\otimes r}$ to $Z_i$ will not suffice, as all points in $Z_i$ are unstable relative to $\cL$. We can twist $\cL^{\otimes r}\vert_{Z_i}$ by a $\chi_i$ and get a $L_i$-linearized line bundle because the kernel of $\chi_i$ contains the image of $\alpha_i$, so that this kernel acts trivially on $Z_i$.

\subsection{The Kempf--Ness stratification of the Grassmannian of a tensor product}
\label{subsec:GrassLT}

Now we study the Kempf--Ness stratification of the Grassmannian of a tensor product, explaining how the objects $\cL, Y_i, Z_i, P_i, L_i, \cL_i$, etc.\ from \S\ref{subsec:GIT}  are realized in this particular case. This will involve notation that will be applied in the proof of Theorem \ref{thm:tensor}. We will follow Totaro \cite[\S2]{totaro1996}.

The general setup of \S\ref{subsec:GIT} is now specified as follows. 
\begin{itemize}[leftmargin=2em]
\item $X = \Gr_s(M \otimes_k N)$, a projective $k$-scheme parametrizing $s$-dimensional subspaces of $M \otimes N$. Let $m = \dim_k M$ and $n = \dim_k N$. 
\item $G = \GL(M) \times \GL(N)$, acting on subspaces of $M \otimes N$ by translation on each tensor factor.
\item The line bundle on $X$ 
\[
\cL := (\wedge^s S^*)^{\otimes mn} \otimes \big(\wedge^{mn} (M \otimes N)\big)^{\otimes s}
\]
is ample and naturally $G$-linearized. Here $S$ is the universal object, a rank $s$ vector bundle on $X$. 
\item We fix a maximal torus $T \subset G$, and a choice of Weyl-invariant positive definite bilinear form $\langle\,, \rangle : X_*(T) \times X_*(T) \ra \bZ$ with induced norm $\abs{\cdot} : X_*(T) \ra \bR_{\geq 0}$. 
\end{itemize}
The first factor of $\cL$ is well-known to be ample on $X$, and the second factor is a trivial line bundle that changes the $G$-linearization. 

\begin{rem}
As a notational convention, we will use the undecorated tensor product $M \otimes N$ to denote the base change from $\Spec k$ to $X$ of $M \otimes_k N$, and also its specializations as appropriate. For example, ``a point $x$ of $X$ corresponds to a subspace $S_x \subset M \otimes N$.'' 
\end{rem}

Totaro calculates that the instability measure of GIT, written $\mu_\cL(x, \alpha)$, is related to the instability measure $f(S_x,\alpha)$ (see \eqref{eq:mu_diff}) of the subspace $S_x \subset M \otimes N$ corresponding to $x \in \Gr_s(M \otimes_k N)$ \cite[Lem.\ 2]{totaro1996}. Indeed, we relate cocharacters $\alpha$ of $G$ to increasing filtrations of $(M,N)$ by assigning $M_\alpha^i$ to the sum of isotypical spaces for the action of $\alpha$ on $M$ with weight $\leq i$. The same is done for $N$.  

To apply the calculations of \cite[Lem.~2]{totaro1996}, we note that in order to translate between Totaro's use of decreasing filtrations and our use of increasing filtrations, one should use the relation $\deg_\alpha S = -\deg^\mathrm{dec}_{-\alpha} S$, where ``$\deg^\mathrm{dec}_{\beta}$'' refers to the degree with respect to the decreasing filtration associated to a cocharacter $\beta$. Then Totaro calculates that $\mu_\cL(x,-\alpha)$ is the positive multiple 
\[
smn \cdot f(S_x, \alpha) = smn \frac{\mu_\alpha(M \otimes_k N) - \mu_\alpha(S_x)}{\abs{\alpha}} = \frac{s\deg_\alpha(M \otimes_k N) - mn\deg_\alpha(S_x)}{\abs{\alpha}}
\]
of the normalized slope difference function $f(S_x, \alpha)$ according to the filtration induced by $\alpha$.  

The upshot is that $S_x$ is semi-stable if and only if $x \in X^{ss}_\cL$. Likewise, the Kempf--Ness stratification of the unstable locus $X\backslash X^{ss}_\cL$ is enumerated by the Kempf filtration of the subspaces associated to the points of this locus. We record this in Proposition \ref{prop:KN=K}, below, after establishing definitions. 

We now set up notation for the Kempf--Ness stratification of $X$, explaining how the structures of \S\ref{subsec:GIT} are realized on vector spaces when they are applied to the Grassmannian. Because increasing filtrations associated to cocharacters $\alpha$ of $G$ are associated with the action of the cocharacter $-\alpha$ on $X$, from now on we will label objects with $\alpha$ that are actually associated to $-\alpha$ in \S\ref{subsec:GIT}. For example, $P_\alpha$ will denote the parabolic subgroup of $G$ associated to the cocharacter $-\alpha$. 

As usual, write $(M^i_\alpha, N^j_\alpha)$ for the increasing filtration associated to $\alpha$. Also note that the action of $\alpha$ on $(M,N)$ allows us to view $\gr^i_\alpha M$ as a direct summand of $M$, and likewise for $\gr^j_\alpha N \subset N$. We observe that the isotypical space of weight $\ell \in \bZ$ of the action of $\alpha$ on $M \otimes_k N$ is the direct summand 
\begin{equation}
\label{eq:Uell}
U_\ell := \bigoplus_{i+j = \ell} (\gr^i_\alpha M) \otimes_k (\gr^j_\alpha N) \subset M \otimes_\bF N.
\end{equation}
Notice that these summands induce an isomorphism $\bigoplus_\ell U_\ell \cong M \otimes_\bF N$. 

Consequently, connected components of $X^\alpha$ are in bijective correspondence with partitions $s = \sum_{\ell\in \bZ} s_\ell$ of $s$, where $0 \leq s_\ell \leq \dim U_\ell$. Each connected component is given by the fiber product over $\Spec k$  
\[
Z_{(s_\ell)} := \times_\ell \Gr_{s_\ell}(U_\ell),
\]
which naturally injects into $X^\alpha$. 

Having enumerated the connected components of $X^\alpha$, we follow \S\ref{subsec:GIT} and set 
\[
Y_{(s_\ell), \alpha} = \{x \in X \mid \lim_{t \ra 0} \alpha^{-1}(t)\cdot x \in Z_{(s_\ell)}\}.
\]
Moduli-theoretically speaking, $Y_{(s_\ell), \alpha}$ consists of those $S \subset M \otimes N$ such that $\gr^\ell_\alpha S \subset \gr^\ell_\alpha M \otimes N$ has dimension $s_\ell$. And the moduli-theoretic realization of the morphism 
\begin{equation}
\label{eq:lim_map}
Y_{(s_\ell), \alpha} \ni x \mapsto \lim_{t \ra 0} \alpha^{-1}(t)\cdot x \in Z_{(s_\ell)}
\end{equation}
sends $S_x$ to $\bigoplus_\ell \gr^\ell_\alpha S_x$, where this is considered to be a subspace of $M \otimes N$ via the canonical inclusion $\gr^\ell_\alpha S_x \subset U_\ell$ and isomorphism $\bigoplus_\ell U_\ell \cong M \otimes N$. In the sequel, we will call this operation \textit{Kempf semi-simplification}. We will record an important property of Kempf semi-simplification in Lemma \ref{lem:kempf_ss} below.  

We also fix $L_\alpha$ and $P_\alpha$ as the Levi subgroup and parabolic subgroup of $G$ associated to $-\alpha$, as in \S\ref{subsec:GIT}. There is a particular subgroup of $P_\alpha$ associated to each connected component $Z_{(s_\ell)}$ in \S\ref{subsec:GIT}, but because $P_\alpha$ and $L_\alpha$ are connected, this subgroup is always $P_\alpha$.  

We note that $\mu_\cL(x, -\alpha)$ is constant over $x \in Y_{(s_\ell), \alpha}$, given by the formula
\[
smn \cdot f(S_x,\alpha) = \mu_\cL(x, -\alpha) = \frac{mn\cdot \sum_\ell \ell \cdot s_\ell}{\abs{\alpha}},
\]
because $\sum_\ell \ell \cdot s_\ell$ is $\deg_\alpha (S)$ and $\deg_\alpha (M \otimes N) = 0$ (Proposition \ref{prop:kempf}). 

Finally, we define the Kempf--Ness stratification as outlined in \S\ref{subsec:GIT}, producing pairs $(Z_i, \alpha_i)$ of connected components $Z_i = Z_{(s_{\ell,i})}$ of $X^{\alpha_i}$ for $i = 1, 2, \dotsc, N$ with decreasing $\mu_i = \mu_\cL(Z_i, -\alpha_i)$. We have open subschemes $Z^\circ_i := Z^\circ_{(s_{\ell,i})} \subset Z_i$ and $Y^\circ_i := Y^\circ_{(s_{\ell,i}), \alpha_i} \subset Y_{(s_{\ell,i}), \alpha_i}$. The $i$th Kempf--Ness stratum is $S_i := G \cdot Y^\circ_{(s_{\ell,i}), \alpha_i}$. These satisfy Theorem \ref{thm:strata}, which also provides morphisms 
\[
[S_i/G] \lrisom [Y^\circ_i/P_{\alpha_i}] \buildrel^{\pi_i}\over\lra [Z^\circ_i/L_{\alpha_i}]. 
\]

We summarize the discussion above. 
\begin{prop}
\label{prop:KN=K}
Let $k'$ be a field extension of $k$, and let $S_x \subset M_{k'} \otimes_{k'} N_{k'}$ be a $k'$-subspace, where $x \in X(k')$ is the associated point of the Grassmannian. 
\begin{enumerate}
\item $S_x$ is semi-stable if and only if $x \in X^{ss}_\cL(k')$. 
\item If $S_x$ is unstable, then $x \in S_i(k')$ if and only if $\alpha_i$ is conjugate to the cocharacter associated to the Kempf filtration associated to $S_x$. 
\end{enumerate}
\end{prop}
\begin{proof}
By Theorem \ref{thm:strata}, all of the strata are defined over $k$. For any extension $k'/k$, Proposition \ref{prop:kempf} provides a unique Kempf filtration (defined over $k'$) if and only if $S_x$ is unstable. 
\end{proof}

\begin{rem}
In the generality of \S\ref{subsec:GIT}, one can only expect the Hilbert--Mumford criterion to produce a produce a maximal destabilizing cocharacter over perfect fields \cite[\S4]{kempf1978}. (See \cite[\S5]{hesselink1978} for a counterexample over an imperfect field.) However, Totaro's proof of Proposition \ref{prop:kempf} produces a Kempf filtration for $S_x$ over its field of definition $k'$ if and only if $S_x$ is not semi-stable, even if $k'$ is imperfect. 
\end{rem}

It will also be important to understand the behavior of the Kempf filtration under the operation  \eqref{eq:lim_map}. 
\begin{lem}
\label{lem:kempf_ss}
Let $k'$ be a field extension of $k$, and let $S_x \subset M_{k'} \otimes_{k'} N_{k'}$ be an unstable $k'$-subspace with Kempf filtration induced by the cocharacter $\alpha$, i.e.\ $x \in Y_{(s_\ell), \alpha}^\circ$ for some partition $(s_\ell)$ of $s$. Then the image $y \in Z_{(s_\ell)}$ of $x$ under \eqref{eq:lim_map}, which corresponds to the subspace
\[
S_y := \bigoplus_\ell \gr^\ell_\alpha S_x \subset \bigoplus_\ell U_{\ell, k'} \cong M_{k'} \otimes_{k'} N_{k'},
\]
has the same Kempf cocharacter as $x$, namely, $\alpha$. In other words, Kempf semi-simplification preserves the Kempf cocharacter. 
\end{lem}
\begin{proof}
The isomorphism $[S_i/G] \lrisom [Y^\circ_i/P_i]$ indicates that we may conjugate the data $(x,\alpha)$ by an element of $G$ in order to reduce to the case that $\alpha = \alpha_i$ and $(s_\ell) = (s_{\ell,i})$ for some $i$ enumerating the sequence of Kempf--Ness strata. It is equivalent to say that $x \in Y^\circ_i$. By construction, the image $y$ of $x$ in $Z_i$ under \eqref{eq:lim_map} lies in $Z^\circ_i$. As $Z^\circ_i \subset Y^\circ_i$, $y$ has the same Kempf cocharacter $\alpha_i$ as $x$ does. 
\end{proof}

Finally, we record that the more general statement of Proposition \ref{prop:ss=kempf_general} specializes to this statement about the Grassmannian. 

\begin{prop}
\label{prop:ss=kempf}
Choose some $z \in Z_{(s_\ell)}$. Then there exists a positive integer $r$ and a character $\chi_\alpha \in X^*(L_\alpha)$ such that 
\[
\cL_{(s_\ell), \alpha} := \cL^{\otimes r} \otimes \chi_\alpha^{-1}
\]
is an $L_\alpha$-linearized ample line bundle on $Z_{(s_\ell)}$. Moreover, $z$ lies in $Z^\circ_{(s_\ell)}$ if and only if $z$ is a semi-stable point relative to $\cL_{(s_\ell), \alpha}$. 
\end{prop}

\subsection{Valuative criteria and descent of base field}
\label{subsec:val}

In this section, we refine the valuative criteria for universally closed morphisms of stacks to a stronger version that holds for weakly proper morphisms of stacks. Here is a word about our motivation for this. 

Broadly speaking, we want to apply the valuative criterion for universally closed morphisms in order to find that certain $F$-points of $X$ are the generic fibers of $\cO$-points of $X$ in the same Kempf--Ness locus. This is too much to expect on $X$, but can nearly be done on the stack quotient $\fX = [X/G]$. The obstruction is that the valuative criterion for universally closed morphisms of finite type algebraic stacks only holds after taking a finitely generated field extension of $F$, which is far too large of an extension to be useful to us. By using the fact that $\fX^{ss}$ and $[Z^\circ_i/L_{\alpha_i}]$ are not only universally closed over a projective space, but also weakly proper (Definition \ref{defn:WP}), we will prove the valuative criterion with only a finite separable field extension of $F$. This will suffice for our purposes. 

The following is one implication of one example among various formulations of the valuative criterion for universally closed morphisms of algebraic stacks. 
\begin{prop}
\label{prop:uc_vc}
Let $f : \fX \ra \fY$ be a locally finite presentation morphism of algebraic stacks, where $\fY$ is assumed to be locally Noetherian. Assume that $f$ is universally closed. Let $R$ be a complete discrete valuation ring with fraction field $\cK$. Assume that there is a commutative diagram
\[
\xymatrix{
\Spec \cK \ar[r] \ar[d] & \fX \ar[d]^f \\
\Spec R \ar[r]  & \fY
}
\]
Then there exists a finitely generated field extension $\cK'/\cK$ with valuation ring $R'$ such that there is a surjective morphism $\Spec R' \ra \Spec R$ and a diagram
\begin{equation}
\label{eq:vc_uc}
\xymatrix{
\Spec \cK' \ar[r] \ar[d] & \Spec \cK \ar[r] \ar[d] & \fX \ar[d]^f \\
\Spec R' \ar[r] \ar[rru]^h &\Spec R \ar[r] & \fY
}
\end{equation}
extending the first diagram. 
\end{prop}

\begin{proof}
See \cite[Thm.~7.3]{lmb}. 
\end{proof}

We will be able to restrict ourselves to the case that $\cK'/\cK$ is finite by choosing a lift of the special point of $\Spec R$ to $\fX$. The following condition will guarantee that such points exist. Here we use $\abs{\fX}$ to denote that topological space underlying $\fX$ (see \cite[\S5]{lmb}). We also write $\nu$ for the closed point of $\Spec R$. 
\begin{defn}[{(\cite[\S2]{ASV2011})}]
\label{defn:WP}
Call $f: \fX \ra \fY$ \emph{weakly separated} if for every valuation ring $R$ with fraction field $\cK$, and 2-commutative diagrams
\[
\xymatrix{
\Spec \cK \ar[r] \ar[d] & \fX \ar[d]^f \\
\Spec R \ar[r] \ar@<.5ex>[ur]^{h_1} \ar@<-.5ex>[ur]_{h_2} & \fY
}
\]
such that $h_1(\nu)$ and $h_2(\nu)$ are closed in $\abs{\fX_R}$ (where $\fX_R = \fX \times_\fY \Spec R$), then $h_1(\nu) = h_2(\nu)$ in $\abs{\fX_R}$. 

Call $f$ \emph{weakly proper} if it is weakly separated, finite type, and universally closed. 
\end{defn}

It is important to note that the equality $h_1(\nu) = h_2(\nu)$ in $\abs{\fX_R}$ does not imply that there is an isomorphism between $h_1$ and $h_2$, but merely an equality of points in a topological space. Also, while \cite{ASV2011} is written for characteristic 0 purposes, none of the arguments of \S2 of \emph{loc.\ cit.}\ depend on this assumption. 

We can strengthen Proposition \ref{prop:uc_vc} when $f$ is weakly proper. 
\begin{prop}
\label{prop:wp_vc}
In the setting of Proposition \ref{prop:uc_vc}, assume in addition that $f$ is weakly separated, and thus weakly proper. Then the field $\cK'$ may be taken to be a finite separable extension of $\cK$.
\end{prop}
\begin{proof}
By \cite[Lem.~2.7]{ASV2011}, there exists some choice of $h$ as in \eqref{eq:vc_uc} such that $h(\nu)$ is a closed point of $\abs{\fX_{R'}}$. Then by a slicing argument as in \cite[Lem.~2.5]{ASV2011}, one can restrict to the case that $\cK'/\cK$ is finite and separable. 
\end{proof}

We are interested in using the valuative criterion for an adequate moduli space morphism $\phi : \fX \ra Y$ over $\Spec k$. And adequate moduli space morphisms are weakly proper. 
\begin{prop}
\label{prop:ams_vc}
The valuative criterion as in Proposition \ref{prop:wp_vc} holds for adequate moduli space morphisms $\phi : \fX \ra Y$, where $Y$ is a proper $k$-scheme. 
\end{prop}
\begin{proof}
In the case that an adequate moduli space morphism is, moreover, a good moduli space morphism, this is verified in \cite[Prop.~2.17]{ASV2011}. However, the difference between adequate moduli spaces and good moduli spaces is immaterial for this proof, as the critical property is that any geometric fiber of $\phi$ has a unique closed point. This property is true for both good and adequate moduli space morphisms; see e.g.\ \cite[Main Thm.]{alper2014}. 
\end{proof}

\subsection{Two corollaries of Langton type}
\label{subsec:two}

In this section we deduce the results of Langton type outlined in \S\ref{subsec:LT_overview} from the GIT background assembled in \S\S\ref{subsec:GIT}-\ref{subsec:val}. 

Our first Langton type result follows from the general GIT result given in Corollary \ref{cor:LT_generic}, when applied to the Grassmannian $X$ and its $G$-linearized ample line bundle defined in \S\ref{subsec:GrassLT}. In the statement of the result, given a finite separable field extension $F'/F$, we write $\cO'$ for the integral closure of $\cO$ in $F'$, which is an $\cO$-finite DVR (see \S\ref{subsec:base_change}). We will write $u'$ for its uniformizer and $k'$ for its residue field. 
\begin{cor}
\label{cor:LT_ss}
Let $M, N$ be finite dimensional $k$-vector spaces. Let $\cS$ be a semi-stable subspace of $M_F \otimes_F N_F$. Then there exist a finite separable field extension $F'/F$ and $\cO'$-lattices $\fM \subset M_{F'}$ and $\fN \subset N_{F'}$ such that the reduction modulo $u'$ of $\cS_{F'} \cap \fM \otimes_{\cO'} \fN$ is a semi-stable subspace of $(M \otimes_{k} N)_{k'}$. 
\end{cor}
\begin{proof}
By Corollary \ref{cor:LT_generic}, we have an adequate moduli space morphism $\phi: \fX^{ss}_\cL \ra X^{ss}_\cL /\!/ G$ whose base is a projective $k$-scheme. The subspace $\cS \subset M_F \otimes_F N_F$ corresponds to a $F$-point $x$ of $X^{ss}_\cL$. Taking it as a $F$-point of $\fX^{ss}_\cL$ for a moment, it maps to a $F$-point $w = \phi(x)$ of $X^{ss}_\cL /\!/ G$. Because $X^{ss}_\cL /\!/ G$ is a projective $k$-scheme, we may apply the valuative criteria for proper morphisms of schemes and find a unique $\cO$-point $\tilde w$ of $X^{ss}_\cL /\!/ G$ whose generic point is $\phi(x)$. 

We may apply the valuative criterion of weakly proper morphisms recorded in Proposition \ref{prop:wp_vc} to the data of $\phi$, $x$, and $\tilde w$ because adequate moduli space morphisms such as $\phi$ are weakly proper (Proposition \ref{prop:ams_vc}). We get the field $F'$ and DVR $\cO'$ as in the statement, and a $\cO'$-point $\tilde x$ of $\fX^{ss}_\cL$ such that $\tilde x \otimes_{\cO'} F' \simeq x \otimes_F F'$. 

By definition of a quotient stack, the data of $\tilde x$ amounts to a $G$-torsor $\cG$ over $\cO'$ and a $G$-equivariant morphism $\cG \ra X^{ss}_\cL$. Because $G \simeq \GL_m \times \GL_n$, any $G$-torsor over $\cO'$ is trivial by Hilbert theorem 90 (or, $\GL_n$ is ``special in the sense of Serre'' \cite{serre1958}). As a trivial $\cO'$-torsor has a $\cO'$-point, we obtain a point $\hat x \in X^{ss}_\cL(\cO')$ that projects to $\tilde x \in \fX^{ss}_\cL(\cO')$. Also, $\hat x \otimes_{\cO'} F'$ is in the $G(F')$-orbit of $x \otimes_F F'$. Let $g = (g_M, g_N) \in G(F') = \GL(M_{F'}) \times \GL(N_{F'})$ such that $(g_M, g_N) \cdot (x \otimes_F F') = \hat x \otimes_{\cO'} F'$. 

Finally, we translate these statements back into the language of vector spaces and lattices. Let $\fM_0 \subset M_{F'}$ be the standard lattice $M \otimes_k \cO'$, and define $\fN_0$ similarly. Then set $\fM = g_M^{-1} \fM_0$ and set $\fN = g_N^{-1} \fN_0$. Thus $g\cdot \cS_{F'}$ in $M_{F'} \otimes_{F'} N_{F'}$ satisfies the statement of the Corollary relative to the lattices $(\fM_0, \fN_0)$. Translating this relationship by $g^{-1}$, we have the desired result. 
\end{proof}

Here is the second Langton type result that we will require. 
\begin{thm} 
\label{thm:LT_kempf}
Let $\alpha$ be a cocharacter of $G$, and let $M = \bigoplus_i \gr^i_\alpha M$, $N = \bigoplus_j \gr^j_\alpha N$ be the resulting direct sum decompositions of $M$ and $N$. For each $\ell$, choose $\cS_\ell \subset U_\ell \otimes_k F$ and write $\cS = \bigoplus_\ell \cS_\ell \subset M_F \otimes_F N_F$. Assume that $\cS$ is an unstable subspace of $M_F \otimes_F N_F$ with Kempf cocharacter $\alpha$. Then there exists a finite separable field extension $F'/F$ and $\cO'$-sublattices 
\[
\gr^i \fM \subset \gr^i_\alpha M \otimes_k F', \quad 
\gr^j \fN \subset \gr^j_\alpha N \otimes_k F'
\]
such that the reduction $\overline{ \cS}$ (modulo $u'$) of $\cS_{F'} \cap (\oplus \gr^i \fM) \otimes_{\cO'} (\oplus \gr^j \fN)$ is unstable in $(M \otimes_k N)_{k'}$ with Kempf cocharacter $\alpha$. Moreover, the Kempf filtration on $(M_{k'}, N_{k'})$ relative to $\overline{ \cS}$ (resp.\ the Kempf filtration on $(M_{F'}, N_{F'})$ relative to $\cS_{F'}$) is realized by specializing (resp.\ generalizing) the filtrations
\[
\fM^i = \bigoplus_{a \leq i} \gr^a \fM, \quad \fN^j = \bigoplus_{b \leq j} \gr^b \fN
\]
of $(\fM,\fN)$ by strict $\cO'$-sublattices. 
\end{thm}

\begin{proof}
By Proposition \ref{prop:ss=kempf}, the locus $Z^{ss}_{\cL'} = Z_{(s_\ell)}^\circ$ of $Z_{(s_\ell)}$ on which $\alpha$ induces the Kempf filtration is the semi-stable locus associated to an ample $L_\alpha$-linearized line bundle which we will refer to by $\cL'$ here. Exactly as in the proof of Corollary \ref{cor:LT_ss}, it follows that there exists an adequate moduli space morphism from $\fZ^{ss}_{\cL'} := [Z^{ss}_{\cL'}/L_\alpha]$ to a projective $k$-scheme. Then Propositions \ref{prop:wp_vc} and \ref{prop:ams_vc} may be applied to the $F$-point of $\fZ$ induced by the subspace $\cS$ of the statement. Applying the valuative criteria, we have $F'$, $\cO'$, and a $\cO'$-point $\tilde z$ of $\fZ^{ss}_{\cL'}$ such that $\tilde z \otimes_{\cO'} F'$ is isomorphic to $z \otimes_F F'$. Again, as in the proof of \ref{cor:LT_ss}, because $L_\alpha$ is a direct product of $\GL_d$ for various $d$, there is a point $\hat z \in Z^{ss}_{\cL'}$ realizing $\tilde z$. Also, there exists $l = (l_M, l_N) \in L_\alpha(F')$ such that $l \cdot (z \otimes_F F') = \hat z \otimes_{\cO'} F'$. 

Finally, translating this statement back to graded vector spaces and writing $\gr^i \fM_0$ for $(\gr^i_\alpha M) \otimes_k \cO'$ and likewise defining $\gr^j \fN_0$, we find that $(l_M^{-1}\fM_0, l_N^{-1}\fM_0)$ are the desired $\cO'$-lattices in $M_{F'} \otimes_{F'} N_{F'}$. 
\end{proof}

\subsection{Conclusion of the proof of Theorem \ref{thm:tensor}}
\label{subsec:tensor_conclusion}

We now prove the tensor product theorem, using the input from GIT summed up in \S\ref{subsec:two}. 

 \begin{proof}[Proof of Theorem \ref{thm:tensor}]  
 
We recall the notation $F = k\lp u\rp$, $\cO = k \lb u \rb$. The $p$-torsion Kisin modules $\fM, \fN$ are finitely generated free $\cO$-modules with a semi-linear endomorphism of finite height, and $\cM, \cN$ are the \'etale $\phz$-modules arising as $\cM = \fM \otimes_\cO F$, $\cN = \fN \otimes_\cO F$. 

As discussed in \S\ref{subsec:ten_sum}, it suffices to check that 
\begin{equation}
\label{eq:ss_goal}
\mu(\cS \cap \fM \otimes_\cO \fN) \geq \mu(\fM \otimes_\cO \fN)
\end{equation}
for all sub-\'etale $\phz$-modules $\cS \subset \cM \otimes_F \cN$. 

To accomplish this, we use Proposition \ref{prop:alt_deg}, which, for our present purposes, may be read as follows. Given a Kisin module $\fP$, an $\cO$-lattice $\fP_0 \subset \cP := \fP \otimes_\cO F$ (that is not required to be $\phz$-stable), and a sub-\'etale $\phz$-module $\cS \subset \cP$, there are two filtrations $\Fil_1, \Fil_2$ on the $k$-vector space $\overline{\fP}_0 := \fP \otimes_\cO k$ such that 
\begin{equation}
\label{eq:fil_sum}
\deg(\fP \cap \cS) \geq \deg_{\Fil}(\overline{\cS \cap \fP_0}) := a \deg_{\Fil_1}(\overline{\cS \cap \fP_0}) + b \deg_{\Fil_2}(\overline{\cS \cap \fP_0}),
\end{equation}
where $a,b \in \bQ_{> 0}$.  Equality is guaranteed when $\cS = \cP$, or, equivalently, when $\cS \cap \fP = \fP$. Below we will write ``$\deg_{\Fil}(\fP \cap \cS)$'' for the right hand side, expressing the inequality as 
\begin{equation}
\label{eq:alt_input}
\deg(\fP \cap \cS) \geq \deg_{\Fil}(\fP \cap \cS), \text{ or, equivalently, } \mu(\fP \cap \cS) \geq \mu_{\Fil}(\fP \cap \cS)
\end{equation}
with \emph{equality} when $\fP = \fP\cap \cS$. 

\textbf{Case 1: $\cS$ is semi-stable.} First we establish \eqref{eq:ss_goal} in the case that $\cS \subset \cM \otimes_F \cN$ is semi-stable as a vector subspace. In this case, Corollary \ref{cor:LT_ss} guarantees that there exist the following data. There is a finite separable field extension $F'/F$ with integral closure $\cO'$ of $\cO$, and $\cO'$-lattices 
\[
\fM_0 \subset \cM_{F'} , \quad \fN_0 \subset \cN_{F'}
\]
such that 
\[
\overline{\cS}_0 := \overline{\cS_{F'} \cap \fM_0 \otimes_{\cO'} \fN_0} \subset \overline{\fM}_0 \otimes_{k'} \overline{\fN}_0
\]
is a semi-stable subspace. As discussed in \S\ref{subsec:base_change}, the Kisin module structure of $\fM$, $\fN$ and the \'etale $\phz$-module structure of $\cM$, $\cN$, $\cS$ naturally extend to $\fM_{\cO'}$, $\fN_{\cO'}$, $\cM_{F'}$, $\cN_{F'}$, $\cS_{F'}$. Moreover, Proposition \ref{prop:BCdeg} tells us that $\deg(\fP) = \deg(\fP_{\cO'})$ for any Kisin module $\fP$ over $\cO$. Therefore, we may establish \eqref{eq:ss_goal} by writing $\fQ$ for $\cS_{F'} \cap \fM_{\cO'} \otimes_{\cO'} \fN_{\cO'}$ and proving 
\begin{equation}
\label{eq:ss_goal'}
\mu(\fQ) \geq \mu(\fM_{\cO'} \otimes_{\cO'} \fN_{\cO'}). 
\end{equation}

By definition of semi-stable subspace, for any filtration $\Fil$ on $(\overline{\fM}_0, \overline{\fN}_0)$, we have $\deg_{\Fil}(\overline{\cS}_0) \geq \deg_{\Fil}(\overline{\fM}_0 \otimes_{k'} \overline{\fN}_0)$. This contributes the middle inequality in 
\[
\mu(\fQ) \geq \mu_{\Fil}(\fQ) \geq \mu_{\Fil}(\fM_{\cO'} \otimes_{\cO'} \fN_{\cO'}) = \mu(\fM_{\cO'} \otimes_{\cO'} \fN_{\cO'}), 
\]
while the outer two (in)equalities come from \eqref{eq:alt_input}. This proves \eqref{eq:ss_goal'}, as desired. 

\textbf{Case 2: $\cS$ is unstable.} It remains to prove \eqref{eq:ss_goal} when $\cS \subset \cM \otimes_F \cN$ is unstable as a vector subspace.  In this case, there exists a unique Kempf filtration of $(\cM, \cN)$ (Proposition \ref{prop:kempf}), which we label by $\cM^i_\alpha$ and $\cN^j_\alpha$. The Kempf filtration minimizes the (negative) normalized difference 
\[
\frac{\mu_{\alpha} (\cS)  - \mu_{\alpha} (\cM \otimes_F \cN)}{\abs{\alpha}}
\]
over all filtrations of $\alpha$ of $(\cM, \cN)$.  Since the Kempf's filtration is unique up to scaling and $\phi_{\cM \otimes\cN} : \phz^*(\cM \otimes_F \cN) \ra \cM \otimes_F \cN$ is a linear isomorphism, we deduce that the $\cM^i_\alpha$ and $\cN^i_{\alpha}$ are $\phi$-stable whenever $\cS$ is.  That is, they are sub-\'etale $\phz$-modules of $\cM$ and $\cN$.

Consider now the strict sub-Kisin modules 
\[
\fM^i := \fM \cap \cM^i_\alpha \subset \fM, \qquad \fN^j := \fN \cap \cN^j_\alpha \subset \fN.
\]
Because $\fM$ and $\fN$ are assumed to be semi-stable Kisin modules, we know that 
\begin{equation}
\label{eq:ss_cond}
\mu(\fM^i) \geq \mu(\fM) \quad \text{ and } \quad \mu(\fN^j) \geq \mu(\fN). 
\end{equation}

We want to arrange for each of the factors in \eqref{eq:ss_cond} to be precisely computed by $\mu_{\Fil}$ for an appropriate choice of $\fM_0 \subset \cM$ and $\fN_0 \subset \cN$. We will do this after semi-simplifying as follows. 

The Kempf co-character $\alpha$ acting on $(\cM, \cN)$ induces a direct sum decomposition with summands
\begin{equation}
\label{eq:gr_summands}
\gr^i_\alpha \cM \subset \cM, \quad \gr^j_\alpha \cN \subset \cN. 
\end{equation}
These $F$-linear summands are not necessarily \'etale $\phz$-module summands. We put a new \'etale $\phz$-module structure on the vector spaces $\cM$ and $\cN$, which we write as $\cM^\mathrm{new}$, $\cN^\mathrm{new}$ and call the ``Kempf semi-simplification'' of $\cM$, $\cN$, respectively. Namely, each of $\cM,\cN$ is canonically a vector space direct sum with summands as in \eqref{eq:gr_summands}. Here each graded piece has a \'etale $\phz$-module structure because the Kempf filtration consists of sub-\'etale $\phz$-modules, resulting in an \'etale $\phz$-module structure on the sums. This is exactly the analogue for \'etale $\phz$-modules of the operation \eqref{eq:lim_map} on vector spaces. 

Likewise, we replace $\fM$ and $\fN$ with their Kempf semi-simplifications 
\[
\fM^\mathrm{new} := \bigoplus_i (\fM \cap \cM^i_\alpha)/(\fM \cap \cM^{i-1}_\alpha), \quad
\fN^\mathrm{new} := \bigoplus_j (\fN \cap \cN^j_\alpha)/(\fN \cap \cN^{j-1}_\alpha)
\]
and replace $\cS$ with its Kempf semi-simplification 
\[
\cS^\mathrm{new} := \bigoplus_\ell (\cS \cap (\cM \otimes_F \cN)^\ell_\alpha)/(\cS \cap (\cM \otimes_F \cN)^{\ell-1}_\alpha). 
\]

Degree and slope of the tensor factors remains constant under semi-simplification: 
\[
\deg(\fM) = \deg(\fM^\mathrm{new}), \quad \deg(\fN) = \deg(\fN^\mathrm{new}). 
\]
By Lemma \ref{lem:lattice_ss}, we have 
\[
\deg(\cS \cap (\fM \otimes_\cO \fN)) \geq \deg(\cS^\mathrm{new} \cap (\fM^\mathrm{new} \otimes_\cO \fN^\mathrm{new})),
\]
where the left hand intersection is taken in $\cM \otimes_F \cN$ and the right hand intersection is taken in $\cM^\mathrm{new} \otimes_F \cN^\mathrm{new}$. Consequently, we may proceed toward our goal of establishing \eqref{eq:ss_goal} after replacing the former objects by their ``new'' Kempf semi-simplified version. Going forward, we drop ``new'' from our notation and use ``old'' to refer to the objects before Kempf semi-simplification. 

Notice that after Kempf semi-simplification, $\cS \subset \cM \otimes_F \cN$ is a subspace of precisely the form required to apply Theorem \ref{thm:LT_kempf}. Moreover, Lemma \ref{lem:kempf_ss} ensures that we may use the $\alpha$ as the Kempf cocharacter for $\cS$, just as we did before Kempf semi-simplifying. 

Applying Theorem \ref{thm:LT_kempf}, we may fix $\cO'$-lattices 
\[
\gr^i \fM_0 \subset \gr^i_\alpha \cM_{F'}, \quad \gr^j \fN_0 \subset \gr^j_\alpha \cN_{F'}
\]
for all $i,j$, such that when we write $\fM_0 := \oplus_i \gr^i \fM_0$ and $\fN_0 := \oplus_j \gr^j \fN_0$, the Kempf filtration on $(\overline{\fM}_0, \overline{\fN}_0)$ arising from the unstable subspace 
\[
\overline{\cS}_0 := \overline{\cS_{F'} \cap \fM_0 \otimes_{\cO'} \fN_0} \subset \overline{\fM_0} \otimes_{k'} \overline{\fN_0}
\]
consists of subspaces
\[
\overline{\fM}^i_0 := \overline{\cM^i_{\alpha, F'} \cap \fM_0} \subset \overline{\fM}_0, \quad 
\overline{\fN}^j_0 := \overline{\cN^j_{\alpha, F'} \cap \fN_0} \subset \overline{\fN}_0. 
\]

We will now apply the comparison of $\deg$ and $\deg_{\Fil}$ given in \eqref{eq:alt_input} as a consequence of Proposition \ref{prop:alt_deg}. Because $\fM_0$ and $\fN_0$ arise as direct sums with summands $\gr^i \fM_0$, $\gr^j \fN_0$ matched to the summands $\gr^i \fM_{\cO'}$ and $\gr^j \fN_{\cO'}$, and because each element $\fM^i$, $\fN^j$ of the filtration on $(\fM, \fN)$ consists of direct sums of certain of these summands, Proposition \ref{prop:alt_deg} guarantees that we have equality in \eqref{eq:alt_input}. Consequently, we may combine these equalities with the fact that $\fM^\mathrm{old}$ is semi-stable (as recorded in \eqref{eq:ss_cond}) and that $\deg(\fM^i) = \deg((\fM^i)^\mathrm{old})$ to conclude 
\begin{equation}
\label{eq:ss_factors}
\mu(\fM^i_{\cO'}) = \mu_{\Fil}(\fM^i_{\cO'}) \geq \mu_{\Fil}(\fM_{\cO'}) = \mu(\fM_{\cO'}),
\end{equation}
and the analogous inequality for $\fN_{\cO'}, \fN^j_{\cO'}$ in place of $\fM_{\cO'}, \fM^i_{\cO'}$. 

Now we apply \cite[Prop.~2]{totaro1996}, which gives some $c > 0$ such that for any filtration $\beta$ on $(\overline{\fM_0}, \overline{\fN_0})$, we have 
\begin{align*}
&\mu_\beta(\overline{\cS}_0) - \mu_\beta(\overline{\fM}_0 \otimes_{k'} \overline{\fN}_0) \geq \\
&c \cdot \left(\sum_i (\mu_\beta(\overline{\fM}^i_0) - \mu_\beta(\overline{\fM}_0)) \dim_{k'} \overline{\fM}^i_0 + \sum_j (\mu_\beta(\overline{\fN}^j_0) - \mu_\beta(\overline{\fN}_0)) \dim_{k'} \overline{\fN}^j_0\right).
\end{align*}
Taking the sum of these inequalities over the two factors adding up to $\mu_{\Fil}$ in \eqref{eq:fil_sum}, the very same inequality is true when we replace $\mu_\beta$ with $\mu_{\Fil}$. Because each factor $\mu_\beta(\overline{\fM}^i_0) - \mu_\beta(\overline{\fM}_0)$ is equal to $\mu(\fM^i_{\cO'}) - \mu(\fM_{\cO'})$ by applying \eqref{eq:ss_factors}, the right hand side of the inequality is positive. Consequently, the left hand side is positive. Writing $\fQ$ for $\cS_{F'} \cap \fM_{\cO'} \otimes_{\cO'} \fN_{\cO'}$, we can then apply \eqref{eq:alt_input} to conclude 
\begin{align*}
\mu(\fQ) \geq \mu_{\Fil}(\fQ) \buildrel{\mathrm{def}}\over{=} \mu_{\Fil}(\overline{\cS}_0) &\geq \mu_{\Fil}(\overline{\fM}_0 \otimes_{k'} \overline{\fN}_0) \\ 
& \buildrel{\mathrm{def}}\over{=} \mu_{\Fil}(\fM_{\cO'} \otimes_{\cO'} \fN_{\cO'}) = \mu(\fM_{\cO'} \otimes_{\cO'} \fN_{\cO'}),
\end{align*}
establishing \eqref{eq:ss_goal'} to complete the proof. 
\end{proof}  

We thank Christophe Cornut for pointing out the need for the following lemma in the foregoing proof. 
\begin{lem}
\label{lem:lattice_ss}
Let $\cP$ be an \'etale $\phz$-module with lattice $\fP \subset \cP$ and a separated increasing filtration
\[
0 = \cP^0 \subset \cP^1 \subset \dotsm \subset \cP^n = \cP
\]
by sub-\'etale $\phz$-modules. Let $\cS \subset \cP$ be an arbitrary sub-\'etale $\phz$-module. Define 
\[
\cS^\mathrm{new}:= \bigoplus_{i=1}^n (\cP^i \cap \cS)/(\cP^{i-1} \cap \cS), \quad \fP^\mathrm{new} := \bigoplus_{i=1}^n (\cP^i \cap \fP)/(\cP^{i-1} \cap \fP). 
\]
Then 
\[
\deg(\cS \cap \fP) \geq \deg(\cS^\mathrm{new} \cap \fP^\mathrm{new}). 
\]
\end{lem}
\begin{proof}
The lattice $\cS \cap \fP$ can be filtered by intersection with the $\cP^i$, making for a composition series with Jordan-H\"older factors 
\[
\fQ^i := \frac{\cS \cap \fP \cap \cP^i}{\cS \cap \fP \cap \cP^{i+1}}.
\]
We have $\deg(\cS \cap \fP) = \sum_1^n \deg(\fQ^i)$. 

Likewise, the $i$th graded piece of the intersection $\cS^\mathrm{new} \cap \fP^\mathrm{new}$ is 
\[
\fX^i := \frac{\cP^i \cap \cS}{\cP^{i+1} \cap \cS} \cap \frac{\cP^i \cap \fP}{\cP^{i+1} \cap \fP}
\cong \frac{(\cP^i \cap \cS + \cP^{i+1}) \cap (\cP\cap \fP + \cP^{i+1})}{\cP^{i+1}}, 
\]
where the left hand intersection is taken in $\cP^i/\cP^{i+1}$. We have $\deg(\cS^\mathrm{new} \cap \fP^\mathrm{new}) = \sum_1^n \deg(\fX^i)$. 

Given $i$, observe that there is a natural morphism of Kisin modules $f^i: \fQ^i \ra \fX^i$. One may quickly check that $f^i$ is both injective and induces an isomorphism on the generic fiber, implying that $\coker(f^i)$ has finite cardinality. The desired inequality follows from Proposition \ref{prop:HN_main}(iii). 
\end{proof}
 
 \section{HN-filtrations in families} 
 \label{sec:fam}
 
In this section, we study the behavior of HN-polygons and HN-filtrations in families of Kisin modules.  The main result is that in families with constant generic fiber, the HN-polygon is upper semi-continuous (Theorem \ref{semi}).  We also study the HN-filtration over Artinian deformations in preparation for \S \ref{sec:flat}. 

\subsection{Kisin modules with coefficients}
\label{subsec:KM_coefs}

Let $A$ be a finitely generated $\bZ_p$-algebra where $p$ is nilpotent. To begin, we define Kisin modules with coefficients in $A$. Write $\fS_A := \fS \otimes_{\bZ_p} A$. 
\begin{defn}
A \emph{Kisin module over $A$}  (with bounded height) or \emph{$A$-Kisin module} is a finite $\fS_A$-module $\fM_A$ that is projective of constant rank, together with an $A$-linear Kisin module structure map $\phi_{\fM_A}:\phz^*(\fM_A)[1/E(u)] \risom \fM_A[1/E(u)]$ satisfying the height condition \eqref{eq:heightcond} for some $a,b \in \bZ$. Denote this additive exact rigid tensor category by $\Mod_{\fS_{A}}^{\phz}$. 
\end{defn}

Likewise, we will use \'etale $\phz$-modules with coefficients in $A$. We write $\cO_{\cE,A}$ for $\cO_\cE \otimes_{\bZ_p} A$. 
\begin{defn}
An \emph{\'etale $\phz$-module over $A$} or \emph{\'etale $\cO_{\cE,A}$-module} is a finite $\cO_{\cE,A}$-module $\cM_A$ that is projective of constant rank, together with an $A$-linear structure of an \'etale $\phz$-module $\phi_{\cM_A}:\phz^*(\fM_A) \risom \cM_A$. Denote this additive exact rigid tensor category by $\Mod_{\cO_{\cE,A}}^{\phz}$.
\end{defn}

Notice that even when $A$ is a finite field $A = \bF$, the constant rank condition on Kisin modules seems a priori stronger than merely insisting on projectivity. This is the case because $\fS_\bF \cong (\bF \otimes_{\F_p} k)\lb u\rb$, which is a product of rings of the form $\bF'\lb u \rb$ where $\bF'$ is a finite field. However, by the argument of e.g.\ \cite[Lem.\ 1.2.7(4)]{mffgsm}, it turns out that a projective $\fS_A$-module with a Kisin module structure compatible with $A$ is constant rank whenever $\Spec A$ is connected. So the constant rank condition is not serious. The same discussion applies to \'etale $\phz$-modules. 

Only briefly in \S\ref{sec:flat} will we consider Kisin modules with $A$-structure that are not projective as $\fS_A$-modules, and we will call these ``generalized $A$-Kisin modules.'' By the comments above, when $A = \F$ is a finite field, then every Kisin module with $\bF$-structure is automatically projective of constant rank over $\fS_\F$, as they must be $u$-torsion free. 

\subsection{$p$-torsion Kisin modules in families}
\label{subsec:p-fams}

Our HN-theory for Kisin modules from \S\ref{sec:HN} can be applied to $A$-Kisin modules when $A$ has finite cardinality, because in this case an $A$-Kisin module is an object of $\Mod^\phz_{\fS, \tor}$ when the $A$-structure is forgotten. We will begin in the case that $A= \F$ is a finite extension of $\F_p$.  
\begin{prop} 
\label{prop:HNfiltoverF} 
For any $\fM_{\F} \in \Mod_{\fS_{\F}}^{\phz}$, consider the HN-filtration 
\[
\fM_{\F, 0} = 0 \subset \fM_{\F, 1} \subset \fM_{\F,2} \subset \ldots \subset \fM_{\F, n} = \fM_{\F}
\]
with $\fM_{\F, i} \in \Mod_{\fS}^{\phz}$ (see Theorem \ref{thm:HNfilt}). Then, each $\fM_{\F, i} \in \Mod_{\fS_{\F}}^{\phz}$.
\end{prop}

\begin{proof} 
By Corollary \ref{cor:endo}, for every $a \in \bF$, the endomorphism of $\fM_\bF$ induced by $a$ preserves the HN-filtration. Equivalently, by the comments of \S\ref{subsec:KM_coefs}, each $\fM_{\F, i}$ is an $\F$-Kisin module. 
\end{proof}

\begin{rem} 
A consequence of Proposition \ref{prop:HNfiltoverF} is that one can work in the a priori smaller category $\Mod_{\fS_{\F}}^{\phz}$ and one gets the same HN-filtration. 
\end{rem}

\begin{prop} 
\label{basechange} 
If $A$ is any finite $\F$-algebra and $\fM_{\F} \in \Mod_{\fS_{\F}}^{\phz}$, then the HN-filtration of $\fM_{\F} \otimes_{\fS_{F}} \fS_A$ is given by 
\[
0 \subset \fM_{\F, 1} \otimes_{\F} A \subset  \fM_{\F, 2} \otimes_{\F} A \subset \ldots \subset  \fM_{\F, n} \otimes_{\F} A = \fM_{\F} \otimes_{\F} A. 
\] 
In particular, the HN-filtration consists of objects of $\Mod^\phz_{\fS_A}$.
\end{prop}

\begin{proof}  
If $\fN_{\F}$ is any semi-stable object with slope $\mu(\fN_{\F})$, then, as a $\fS$-module, $\fN_{\F} \otimes_{\F} A$ is a direct sum of finitely many copies of $\fN_{\F}$.  The direct sum of two semi-stable objects of slope $\mu$ is semi-stable with slope $\mu$ (Proposition \ref{prop:ext_ss}).  Thus, $\fN_{\F} \otimes_{\F} A$ is semi-stable with $\mu(\fN_{\F} \otimes_{\F} A) = \mu( \fN_{\F} )$. Thus, the base change $\fM_{\F, i} \otimes_{\F} A$ of the HN-filtration has the properties of Definition \ref{thm:HNfilt}, which uniquely characterize the HN-filtration. 

Because $\fM_{\F, i}$ are projective of constant rank over $\fS_\bF$, the same holds for $\fM_{\F,i} \otimes_\bF A$ over $\fS_A$. 
\end{proof}

Next, we renormalize the HN-polygon so that if $\F'$ is a finite extension of $\F$, then $\HN_{\F}(\fM_{\F}) = \HN_{\F'}(\fM_{\F} \otimes_{\F} \F')$.  

\begin{defn} 
Let $\fM_{\F} \in \Mod_{\fS_{\F}}^{\phz}$. We set
$$
\HN_{\F}(\fM_{\F})(x) := \frac{1}{[\F:\F_p]} \HN(\fM_{\F})([\F:\F_p]x)
$$
such that the endpoints of $\HN_{\F}(\fM_{\F})$ are $(0,0)$ and $ \frac{1}{[\F:\F_p]}(\rk(\fM_{\F}), \deg(\fM_{\F}))$.
\end{defn} 

Note that the slopes of $\HN_{\F}(\fM_{\F})$ are the same as those of $\HN(\fM_{\F})$, but the lengths of the segments are scaled by the factor $[\F:\F_p]^{-1}$.  

\begin{prop} 
\label{prop:scalarextension} 
For any $\fM_{\F} \in \Mod_{\fS_{\F}}^{\phz}$ and any finite field extension $\F'/\F$,
$$
\HN_{\F}(\fM_{\F}) = \HN_{\F'}(\fM_{\F} \otimes_{\F} \F').
$$
Moreover, the breakpoints of $\HN_{\F}(\fM_\F)$ lie in $(\Z, \frac{1}{g} \Z)$. 
\end{prop}

\begin{proof} This follows directly from Proposition \ref{basechange}.  
\end{proof}

\subsection{Families with constant generic fiber}

 Fix an \'etale $\cO_{\cE, \F}$-module $\cM_{\F}$. We now address the behavior of HN-filtrations in families of Kisin modules with \emph{constant generic fiber} $\cM_\F$. This is a $\F$-algebra $A$ and an $A$-Kisin module $\fM_A$ such that $\fM_A \otimes_{\fS_A} \cO_{\cE,A} \simeq \cM_\F \otimes_\F A$. 
 
 Let $\Sigma_{\cM_{\F}}$ denote the $\overline{\F}$-stable subobjects of $\cM_{\F} \otimes_{\F} \overline{\bF}$.  Any element of $\Sigma_{\cM_{\F}}$ is defined over some finite extension of $\F$.  

Let $\F'$ be any finite extension of $\F$.  If $\fM_{\F'}$ is a Kisin module over $\F'$ with generic fiber $\cM_{\F} \otimes_\F \F'$, then we have a map
\[
P_{\fM_{\F'}}:\Sigma_{\cM_{\F}} \ra \mathbb{Q}^2
\]
as follows: any $N \in \Sigma_{\cM_{\F}}$ is defined over a some finite field $\F''$ which we can assume contains $\F'$.  We take $P_{\fM_{\F'}}(N)$ to be 
\[
\frac{1}{[\F'':\F_p]}\big(\rk(\fM_{\F'} \otimes_{\F'} \F'' \cap \cN_{\F''}), \deg(\fM_{\F'} \otimes_{\F'} \F'' \cap \cN_{\F''})\big),
\]
which is independent of our choice of $\F''$ by Proposition \ref{prop:scalarextension}. 

\begin{lem} 
\label{chull2} 
For any Kisin module $\fM_{\bF'} \in \Mod^\phz_{\fS_{\F'}}$, $\HN_{\F'} (\fM_{\F'})$ is the convex hull of the image of $P_{\fM_{\F'}}$.
\end{lem}

\begin{proof}
Initially, restrict to subobjects of $\Sigma_{\cM_{\F}}$ which are defined over $\F'$, i.e., $\bF'$-subobjects of $\cM_{\F} \otimes_\F \F'$.  The values of $P_{\fM_{\F'}}$ on these subobjects lie above $\HN_{\F'}(\fM_{\F'})$ by Proposition \ref{chull1}.   Furthermore, by Proposition \ref{prop:HNfiltoverF}, the HN-filtration on $\fM_{\F'}$ is given by strict $\F'$-subobjects; these come from $\F'$-subobjects of $\cM_{\F} \otimes_\F \F'$ (Proposition \ref{prop:HN_main}(i)). From this, we deduce that the breakpoints of $\HN_{\F'}(\fM_{\F'})$ lie in the image of $P_{\fM_{\F'}}$. 

It remains to show that for $\cN \in \Sigma_{\cM_{\F}}$ possibly not defined over $\F'$, $P_{\fM_{\F'}}(\cN)$ lies above $\HN_{\F'}(\fM_{\F'})$.  This follows from the same argument as in Proposition \ref{chull1} using that the HN-filtration is stable under finite base change (Proposition \ref{prop:scalarextension}).
\end{proof}

\begin{thm} 
\label{semi} $($Semi-continuity$)$ 
Let $A$ be a finite type $\F$-algebra. Let $\fM_A \in \Mod^\phz_{\fS_A}$ be an $A$-Kisin module of rank $n$. Assume also that $\fM_A$ has constant generic fiber $\cM_\F$.  Let $P_0:[0, n] \ra \R$ be a convex polygon. Then the set of points $x \in (\Spec A)(\overline{\F})$ such that $\HN_{k(x)}((\fM_A)_x) \geq P_0$ is Zariski closed.
\end{thm}

\begin{proof}
For each point $x \in (\Spec A)(\overline{\F})$, we get a function $P_{\fM_x}:\Sigma_{\cM_{\F}} \ra \mathbb{Q}^2$ as in the discussion before Lemma \ref{chull2}.   Let $Q_x$ denote the projection of $P_{\fM_x}$ onto the second coordinate.  We claim then that for any $\cN \in \Sigma_{\cM_{\F}}$, the integer-valued function $Q_x(\cN)$ of $x$ is upper semi-continuous.

By Proposition \ref{prop:scalarextension}, it is harmless to take a finite extension of the residue field, so we may assume that $\cN$ is defined over $\F$. Consider $\fN_A = \fM_A \cap (\cN \otimes_{\F} A) \subset \cM_{\F} \otimes_{\F} A$ which is stable under $\phi_{\fM_A}$.     Let $\fC_A$ denote the cokernel of $\phi_{\fM_A}|_{\fN_A}$.   The formation of both $\fN_A$ and $\fC_A$ commute with base change on $A$ since everything in sight is $\fS_A$-flat.  In particular, 
$$
\fN_A \otimes_{A, x} \F' = \fM_x \cap (\cN \otimes_\F {\F'}) \subset \cM_{\F} \otimes_\F \F'
$$
for any $x:A \ra \F'$.  

Since for any $x:A \ra \F'$, $Q_x(\cN)$ is the $\F'$-dimension of $(\fC_A)_x$,  to prove the claim, it suffices to show that $\fC_A$ is a coherent $A$-module.  But this follows from the fact that $\fC_A \subset \coker (\phi_{\fM_A})$ (apply the snake lemma and use that everything is well-behaved over $\cO_{\cE, A}$). 

We return to the proof of the theorem.  By Lemma \ref{chull2}, $\HN_x(\fM_x) \geq P_0$ if and only if $Q_x(\cN) \geq P_0(\rk_{\cO_{\cE, \F'}}(\cN))$ for all $\cN \in \Sigma_{M_{\F}}$ and $\F'$ is a field over which $\cN$ is defined. For each $\cN$, this is a closed condition on $\Spec A$ and so applying this for all $\cN$ proves the theorem.
\end{proof}

\subsection{\'Etale rank}
\label{subsec:et_rank}

The first segment of the HN-polygon is of particular interest, as we will show that it cuts out a discrete invariant in mod $p$ families of effective Kisin modules with constant generic fiber. Without loss of generality we will work with effective Kisin modules $\fM$, for which the first segment has length being the \emph{\'etale rank} $d^\et(\fM)$ of Definition \ref{defn:et_rk}. 

As in Theorem \ref{semi}, we continue to let $A$ represent a $\F$-algebra, and let $\F'$ denote field extensions of $\F$. 

\begin{defn} 
\label{defn:etrank} 
Let $\fM_A \in \Mod^{\phz, [0,h]}_{\fS_A}$ be a family of effective Kisin modules of rank $n$ with constant generic fiber $M_\F$. For each point $x:A \ra \bF'$, define $d^{\et}_{\fM_A}(x)$ to be $d^\et((\fM_A)x)/[\F': \F_p]$, the \emph{$\F'$-\'etale rank} of $(\fM_A)_x$. 
\end{defn} 

\begin{cor} 
\label{loweretalerank} 
Let $\fM_A$ be a family of effective Kisin modules with rank $n$ and constant generic fiber. For each nonnegative  integer $d$, the set of points where $d^{\et}_{\fM_A}(x) \leq d$ is closed.  
\end{cor}
\begin{proof}  Consider the polygon $P_0$ starting from $(0,0)$ whose breakpoints are $(d, 0)$ and $(n, \frac{1}{g})$.  For any $x$, if $\HN_{\kappa(x)}((\fM_A)_x) \geq P_0$, then since $\fM_A$ is effective, making all slopes of $(\fM_A)_x$ to be $\geq 0$, $d^{\et}_{\fM_A}(x) \leq d$.  

Conversely, if $d^{\et}_{\fM_A}(x) = d' \leq d$, then $\HN_{\kappa(x)}((\fM_A)_x) $ lies above the polygon with breakpoints $(d', 0)$ and $(n, \frac{1}{[\F:\F_p]} \deg((\fM_A)_x))$ which lies above $P_0$ since $ \frac{1}{[\F:\F_p]} \deg((\fM_A)_x) \geq \frac{1}{g}$. The corollary now follows from Theorem \ref{semi}.  
\end{proof}

For a family with height $[0, d]$,  the function $d^{\et}_{\fM_A}(x)$ is not just lower semi-continuous but is fact locally constant, as was first observed in \cite{mffgsm}:  

\begin{prop}  
\label{prop:upperetrank} 
Let $\fM_A \in \Mod^{\phz,[0,h]}_{\fS_A}$ be a family of Kisin modules of rank $n$. Then for any nonnegative integer $d$, the set of points $x$ of $\Spec A$ where $d^{\et}_{\fM_A}(x) = d$ is open and closed. 
\end{prop}
\begin{proof}
This is the main content in the proof of \cite[Prop.\ 2.4.14]{mffgsm}.
\end{proof}

Proposition \ref{prop:upperetrank} will be important when we consider Kisin varieties in the next section.  Roughly speaking, one expects generalizations of the \'etale rank to account for most of the discrete invariants (see Remarks \ref{rem:all_poly} and \ref{rem:not_all}). 

\section{Applications to Kisin varieties} 
\label{sec:kisin_var}

The goal of this section is to apply the tensor product theorem to the study of Kisin varieties for a reductive group $G$, called $G$-Kisin varieties.  We specify the objects of interest in Proposition \ref{GKisinvar}. The main result is the construction of certain discrete invariants for $G$-Kisin varieties coming from HN-polygons (Theorem \ref{thm:HNcomp}).  For simplicity, we assume for this section that $K/\Qp$ is totally ramified; that is, $k = \F_p$. 

\subsection{Kisin varieties for $\GL_n$}

First, we review the theory of (closed) Kisin varieties for $\GL_n$.  These varieties were originally constructed in \cite{mffgsm} (see \S 2.4.4).  Kisin was working in the height $[0,1]$ setting; however, the construction does not depend in any way on this. They were called \emph{Kisin varieties} in \cite{PR2009} and studied in greater generality (see \S 6). In form, Kisin varieties resemble affine Deligne--Lusztig varieties, but for a different Frobenius. 

\begin{defn}
For $\nu = (a_1, a_2, \ldots, a_n)$ with $a_i \in \Z$ and $a_{i+1} \geq a_i$, a projective Kisin module $(\fM, \phi_{\fM})$ over $\bF'/\bF$ of rank $n$ has \emph{Hodge type} $\nu$ if there exists a basis $\{ e_i \}$  of $\fM$ such that $u^{a_i} e_i$ generates the image of $\phi_{\fM}$.  We define the \emph{Hodge polygon} $P_{\nu}$ to be convex polygon interpolating $(i, a_1 + a_2 + \ldots + a_i)$. 
\end{defn} 

We denote the usual Bruhat order (on dominant cocharacters of $\GL_n$) by $\nu' \leq \nu$.  Namely, $(a_1', a_2', \ldots, a_n') \leq (a_1, a_2, \ldots, a_n)$ if $\sum^d_{i = 1} a_i' \geq \sum^d_{i =1}  a_i$ for all $d \geq 1$ and $\sum^n_{i =1} a_i' = \sum^n_{i =1}  a_i$. Note that $\nu' \leq \nu$ if and only if we have an inequality of Hodge polygons $P_{\nu'} \geq P_\nu$ that have equal endpoints. 

\begin{prop}
Let $(\cM, \phi_\cM)$ be an  \'etale $\cO_{\cE}$-module of rank $n$ over $\F$.  There exists a reduced closed $\F$-subscheme $X^{\nu}_{\cM}$ of the affine Grassmannian of $\cM$ such that for any $\bF'/\bF$
\[
X^{\nu}_{\cM}(\bF') = \{ \F'[\![u]\!] \text{-lattices } \fM \subset \cM \otimes \bF' \mid \fM \text{ has Hodge type } \leq \nu\}.
\]
We call this the (closed) \emph{Kisin variety} associated to $(\cM,  \phi_\cM, \nu)$. 
\end{prop}
\begin{proof} 
See  \cite[Prop.\ 2.4.6]{mffgsm} or \cite[\S6.a.2]{PR2009}. 
\end{proof}

\begin{rem}
\label{rem:nu_height}
Notice that the height of a Kisin module corresponding to a point of $X^\nu_\cM$ is in $[\lfloor{\frac{a_1}{e}}\rfloor, \lceil{\frac{a_n}{e}\rceil}]$, where $\nu = (a_1, \dotsc, a_n)$. Because an affine Grassmannian is ind-projective, and because the bound on height discussed in Remark \ref{rem:nu_height} makes $X^\nu_\cM$ finite type and closed in the affine Grassmannian, $X^\nu_\cM$ is projective, cf.\ \cite[Prop.\ 1.3]{pssdr}. 
\end{rem}

\begin{rem}
\label{rem:lifting_Hodge}
One reason we are interested in $X^\nu_\cM$ and its connected components is the following. If $\fM_{\bF}$ arises from reduction of a crystalline representations with $p$-adic Hodge type $(\mu_{\psi}) \in (\Z^n)^{\Hom(K, \overline{\QQ}_p)}$ with $\mu_{\psi}$ dominant, then the Hodge type of $\fM_{\bF}$ is less than or equal to $\nu = \sum_{\psi \in \Hom(K, \overline{\QQ}_p)}  \mu_{\psi}$. This is a consequence of local model diagram of \cite[Thm.\ 5.3]{CL2015} combined with the description of the special fiber of the local model in Thm.\ 2.3.5 of \emph{loc.\ cit.}\ (see also Prop.\ 5.4 of \emph{loc.\ cit.}). 

Motivated by modularity lifting techniques, one is interested in the connected components of the moduli space of crystalline representations with $p$-adic Hodge type $(\mu_{\psi})$. In the Barsotti--Tate case, the connected components $X^\nu_\cM$ are known to be the same as connected components of the crystalline deformation ring, e.g.\ \cite[\S2.5]{mffgsm}.
\end{rem}

\begin{prop} 
\label{HNoverHodge}
For any point in $\fM \in X^{\nu}_{\cM}(\overline{\F})$ defined over $\F'$, the $($scaled$)$  HN-polygon lies above the Hodge polygon $P_\nu$, i.e., $g \cdot \HN_{\F'}(\fM)(x) \geq \sum_{i \leq x} a_i$, and they have the same endpoints. 
\end{prop}

\begin{proof}  
We write $\bF$ for $\bF'$ for convenience. By Proposition \ref{chull2}, it suffices to show that for any strict $\F$-subobject $\fN \subset \fM$ we have $g \deg(\fN) \geq \sum_{i \leq \rk(\fN)} a_i$.   Let $d = \rk_{\fS_{\F}}(\fN)$.   Consider $\wedge_{\fS_\F}^d(\fN)  \subset \wedge_{\fS_\F}^d (\fM)$.   An elementary calculation shows that 
\[
u^{\sum_{i \leq d} a_i}(\wedge_{\fS_\F}^d(\fM)) \supset \phi(\phz^*(\wedge_{\fS_\F}^d (\fM))).
\]
In particular, the Frobenius on $\wedge^d(\fN)$ is divisible by $u^{\sum_{i \leq d} a_i}$ and so by adapting Lemma \ref{wedgeslope} to work over $\fS_\F$ in place of $\fS$, we have $\deg(\fN) = \deg(\wedge^d(\fN)) \geq \frac{1}{g}  \sum_{i \leq d} a_i$.  The endpoints of $\HN_{\F}(\fM)$ and $P_\nu$ are the same because the $\F$-length of $\coker(\phi_{\fM})$ is $\sum_{i} a_i$. 
\end{proof}

\begin{cor} \label{HNstrata} There is a decomposition
$$
X^{\nu}_{\cM} = \bigcup_{P} X^{\nu, P}_{\cM}  
$$
into locally closed reduced subschemes indexed by concave polygons $P$ such that each closed point $x$ of $X^{\nu, P}_{\cM}$ corresponds to a $\kappa(x)$-Kisin module $\fM_x$ with HN-polygon $\HN_{\kappa(x)}(\fM_x)$ equal to $P$. 
\end{cor}

\begin{proof} This follows from Theorem \ref{semi}. 
\end{proof}

\begin{rem} 
\begin{enumerate}
\item By Proposition \ref{HNoverHodge}, there are finitely many possible HN-polygons appearing in the decomposition in Corollary \ref{HNstrata}. 
\item It is clear that the closure of $X^{\nu, P}_{\cM}$ in $X^{\nu}_{\cM}$  is contained in $\bigcup_{P' \geq P} X^{\nu, P'}_{\cM}$ by semi-continuity, but it is not equal in general (see Proposition \ref{prop:upperetrank}). 
\end{enumerate}
\end{rem}

For any Kisin variety, there is an natural generalization of \'etale rank which was first introduced in \cite[Defn.\ 4.2]{hellmann2009} for $\GL_2$. One replaces the length of the initial horizontal segment of $\HN_{\F}(\fM)$ for an effective Kisin module (as discussed in \S\ref{subsec:et_rank}) with the largest $d_{\max} \geq 0$ such that $g \cdot \HN_{\F}(\fM)(x) = P_\nu(x)$ for $x \in [0, d_{\max}]$ and such that the slope is constant. 

\begin{cor}
 \label{firstsegment} Let $\nu = (a_1, a_2, \ldots, a_n)$, and let $d_{\max}$ be largest integer such that $a_{d_{\max}} = a_1$.  For any $d \leq d_{\max}$, let $S_d = \{P \mid g \cdot P(d) = a_1 d, g \cdot P(d+1) > a_1(d+1) \}$. Then the union
$$
\bigcup_{P \in S_d} X^{\nu, P}_{\cM}  
$$
is a union of connected components of $X^{\nu}_{\cM}$.  
\end{cor}

\begin{proof}
If we take $\cM' = (\cM, u^{-a_1} \phi_{\cM})$, then there is a natural identification
\[
X^{\nu}_{\cM} = X^{\nu'}_{\cM'}
\]
with $\nu' = \nu - (a_1, \ldots, a_1)$.  Furthermore, by the same argument as in Lemma \ref{sstt},  if $\fM \in   X^{\nu}_{\cM}(\F)$ has HN-polygon $P = \HN_\F(\fM)$, then $\fM' = (\fM, u^{-a_1}\phi_{\fM})$ has HN-polygon given by $P'(x) = P(x) - \frac{a_1}{g}x$. By Proposition \ref{prop:upperetrank} we are done.
\end{proof}

\begin{rem} 
By applying Corollary \ref{firstsegment} to all the exterior powers of $\cM$, one can produce further discrete invariants. For example, the largest integer $\ell$ such that $g \cdot \HN(\fM)(x) = P_\nu(x)$ for $x \in [0, \ell]$ (not necessarily of constant slope) is also a discrete invariant.  The complete description of such invariants is given in Theorem \ref{thm:HNcomp}. 
\end{rem}

\subsection{Kisin varieties for a reductive group}

Let $G$ be a connected split reductive group over $\F$ with maximal torus $T$ and Borel subgroup $B \supset T$.  Let $X_*(T)$ (resp. $X^*(T)$) denote the cocharacters (resp. weights) of the torus $T$.    Let $\Delta \subset X^*(T)$ denote the set of simple positive roots, and let $\Delta^{\vee} \subset X_*(T)$ denote the set of simple coroots.  We will use $\leq$ to denote the usual Bruhat order on dominant cocharacters. 

Let $X_*^+(T)_{\QQ}$ denote the subset of (rational) dominant cocharacters, i.e., 
\[
X_*^+(T)_{\QQ} = \{ \lambda \in X_*(T)_{\QQ} \mid \langle \alpha, \lambda \rangle \geq 0, \alpha \in \Delta \}. 
 \] 
Recall the partial ordering on $X^+_*(T)_{\QQ}$ defined in \cite[Lem.\ 2.2]{RR1996}: 
\begin{defn} \label{ordering}  Let $\lambda, \lambda' \in X^+_*(T)_{\QQ}$ then $\lambda \prec \lambda'$ if 
$$
\lambda' - \lambda = \sum_{\alpha^{\vee} \in \Delta^{\vee}} n_{\alpha^{\vee}} \alpha^{\vee}
$$
for  $n_{\alpha^{\vee}} \in \mathbb{Q}$ with $n_{\alpha^{\vee}} \geq 0$. 
\end{defn}    

Equivalently, $\prec$ defines a partial ordering on conjugacy classes of rational cocharacters for any connected reductive group $G$. For $\GL_n$, this is same ordering as the partial ordering on polygons \cite[Prop.\ 2.4(iv)]{RR1996}. Note that if $\nu' \leq \nu$ in the Bruhat order then $\nu' \prec \nu$, but the converse does not hold in general.

\begin{prop} \label{functcrit}  Let $\lambda, \lambda' \in X^+_*(T)_{\QQ}$. Then $\lambda \prec \lambda'$ if and only if either of the following equivalent statements are true:
\begin{enumerate}
\item $\rho(\lambda) \prec \rho(\lambda')$ for all representation $\rho:G \ra \GL(V)$;
\item $\langle \chi, \lambda' - \lambda \rangle \geq 0$ for all dominant weights $\chi \in X^*(T)$.  
\end{enumerate}
 \end{prop}
 
 \begin{proof}  
 See \cite[Lem.\ 2.2]{RR1996} as well as \cite[(12.9) and Prop.\ 12.18]{AB1983}. 
 \end{proof} 

We now introduce Kisin modules with $G$-structure, which we call ``$G$-Kisin modules,'' and the corresponding Kisin varieties.  We refer the reader to \cite[\S 2]{levin2015} for background. For convenience, we work mod $p$ since this is all we need for the application. We also continue to assume that $K/\Qp$ is totally ramified.  
 
\begin{defn} Let $A$ be any $\F$-algebra, then a \emph{$G$-Kisin module over $A$} (with bounded height) is a $G$-bundle $\fP_A$ over $\fS_A := \fS \otimes_{\ZZ_p} A$ together with an isomorphism $\phi_{\fP_A}:\phz^*(\fP_A)[1/u] \cong \fP_A[1/u]$.  Denote the category of such objects by $\mathrm{GMod}^{\phz}_{\fS_A}$.   
\end{defn}

Since we working mod $p$, inverting $u$ is the same as inverting the Eisenstein polynomial $E(u)$.  Because of our assumption that $k = \F_p$, $\fS_{\F'}$ is complete local ring with finite residue field $\F'$.  Since $G$ is connected, any $G$-bundle over $\fS_{\F'}$ is trivial by Lang's theorem. The following is the analogue of the relative position of $\fM$ and $\phi(\phz^*(\fM))$ introduced in the previous section.  

\begin{defn} \label{GHodge} Let $\fP_{\F} \in \mathrm{GMod}^{\phz}_{\fS_{\F}}$.  We say that $\fP_{\F}$ has \emph{Hodge type} $\lambda \in X_*(T)$ if for any choice of trivialization $\beta$ of $\fP_{\F}$, we have 
\[
(\phi_{\fP_{\F}})_{\beta} \in G(\F[\![u]\!]) \lambda(u) G(\F[\![u]\!]).
\] 
where $(\phi_{\fP_{\F}})_{\beta}$ denotes the matrix for Frobenius with respect to $\beta$ and $\lambda(u) \in T(\F(\!(u)\!))$ is induced by $\lambda:\Gm \ra T$.  
\end{defn}  

\begin{defn} An \emph{\'etale $\cO_{\cE}$-module with $G$-structure} over $A$ is a $G$-bundle $\cP$ on $\Spec \cO_{\cE} \otimes_{\ZZ_p} A$ together with an isomorphism $\phi_{\cP}:\phz^*(\cP) \cong \cP$. 
\end{defn}

We can now define the $G$-analogue of Kisin varieties.
\begin{prop} 
\label{GKisinvar} 
Let $(\cP, \phi_{\cP})$ be an \'etale $\cO_{\cE}$-module with $G$-structure over $\F$ and let $\nu \in X_*(T)$. There exists a projective scheme over $\F$ such that for any $\F'/\F$, 
\[
X^{\nu}_{\cP} (\F') = \{ \fP_{\F'}[1/u] \cong \cP \otimes_{\F} \F'  \mid  \fP_{\F'} \in \mathrm{GMod}^{\phz}_{\fS_{\F'}} \text{ has Hodge type} \leq \nu \}.
\]
\end{prop}
\begin{proof}
Kisin varieties for reductive groups were first introduced in \S 6.1.2 of \cite{PR2009}.  Details of the construction appears in \cite[Prop.\ 3.3.9]{levin2015}.   For the convenience of the reader, we include a sketch of the argument. 

Fix a lattice $\fP_{\F, 0}$ over $\Spec \fS_{\F}$ of $\cP$ and let $g_0 \in G(\F(\!(u)\!))$ be the Frobenius on $\fP_{\F, 0}$.   Then $X^{\nu}_{\cP}$ is a closed subscheme of the affine Grassmannian $\Gr_{G}$ of ``lattices'' in $\fP_{\F,0}[1/u]$.   Let $S(\nu) \subset \Gr_G$ denote the closed Schubert variety corresponding to $\nu$.  For any lattice $\fP_{\F} = g \fP_{\F, 0}$ with $g \in  G(\F(\!(u)\!))$,  the condition of having Hodge type $\leq \nu$ is given by $g^{-1} g_0 \phz(g) \in  G(\F[\![u]\!]) \nu'(u) G(\F[\![u]\!])$ for $\nu' \leq \nu$ or, equivalently, $g^{-1} g_0 \phz(g)(\fP_{\F, 0}) \in S(\nu) \subset \Gr_G$.   Since $S(\nu)$ is closed, this defines a closed subscheme of $\Gr_G$.  The fact that $X^{\nu}_{\cP}$ is a scheme  (as opposed to an ind-scheme) can be deduced from the $\GL_n$ case (described in Remark \ref{rem:nu_height}) by embedding $G$ in some $\GL_n$.   
\end{proof}

By the Cartan decomposition for $G(\F(\!(u)\!))$, any $G$-Kisin module over $\F$ has a Hodge type $\lambda$ which is well-defined up conjugation.  If $\fP_{\F}$ has Hodge type $\lambda$, then define the \emph{Hodge vector} of $\fP_{\F}$ to be dominant representative for the conjugacy class of $\lambda$ in $X_*(T)$.  

It is a well-known consequence of the tensor product Theorem \ref{thm:tensor} that the HN-theory for $\Mod_{\fS_\F}^\phz$ carries over to the category of $G$-bundles $\mathrm{GMod}^\phz_{\fS_\F}$.  We briefly recall how to do this, referring to \cite{DOR2010} for the details of each step. 

\begin{prop}  \label{tpwithcoeff}
If $\fM_{\F}$ and $\fN_{\F}$ are semi-stable of slope $\alpha$ and $\beta$, then $\fM_{\F} \otimes_{\fS_{\F}} \fN_{\F}$ is semi-stable of slope $\alpha + \beta$. If $\fM_{\F}$ is semi-stable, then $\Sym^r_{\fS_\F} (\fM_{\F})$ and $\bigwedge^r_{\fS_\F} (\fM_{\F})$ are also semi-stable.  
\end{prop}
\begin{proof} 
Both $\Sym^r_{\fS_\F} (\fM_{\F})$ and $\bigwedge^r_{\fS_\F} (\fM_{\F})$ are subquotients of $\otimes^r_{\fS_\F} \fM_{\F}$ with the same slope so the first part implies the second part using Corollary \ref{cor:subquotss}.  

For the first part, we have that $\fM_{\F} \otimes_{\fS_{\F}} \fN_{\F}$ is a quotient of the tensor product $\fM_{\F} \otimes_{\fS} \fN_{\F}$ (without coefficients) which is semi-stable of the desired slope by Theorem \ref{thm:tensor} and so we can again apply Corollary \ref{cor:subquotss}.  
\end{proof} 

Let $\fM$ be a Kisin module with bounded height.  As before we have an HN-filtration 
\[
0 = \fM_0 \subset \fM_1 \subset \fM_2 \subset \ldots \subset \fM_n = \fM
\]
on $\fM$.  We index the filtration in the natural way following Proposition \ref{prop:HN_funct}, namely, for any $\alpha \in \Q$, we have $\fM^{\leq \alpha} = \bigcup_{i,  \mu(\fM_i/\fM_{i-1}) \leq \alpha} \fM_i$.  We will write $\Filt_{\fS_\F}$ to refer to the additive exact $\otimes$-category of finite projective $\fS_\F$-modules increasingly filtered by strict submodules and indexed as above. The tensor product theorem tells us that the HN-filtration yields an $\otimes$-functor $\Mod^\phz_{\fS_\F} \ra \Filt_{\fS_\F}$.  Exactness is not obvious. But in the context of Kisin module with $G$-structure, it holds true (Theorem \ref{thm:GHN}).  

\begin{prop}  
\begin{enumerate}
\item Let $\fM_{\F}, \fN_{\F}$ be Kisin modules with bounded height.  The HN-filtration on $\fM_{\F} \otimes_{\fS_{\F}} \fN_{\F}$ is induced by the HN-filtrations on $\fM_{\F}$ and $\fN_{\F}$ respectively.  
\item Let $\fM^{\vee}_{\F} = \Hom_{\fS_{\F}}(\fM_{\F}, \fS_{\F})$.  The HN-filtration on $\fM^{\vee}_{\F}$ is dual to the HN-filtration on $\fM_{\F}$.    
\end{enumerate}
\end{prop}

\begin{proof} 
See \cite[Prop.\ 1.3.6]{DOR2010}.
\end{proof}

Let $\fP_{\F}$ be a $G$-Kisin module over $\F$.  This induces an exact $\otimes$-functor 
$$
\Rep_{\F}(G) \ra  \Mod_{\fS_{\F}}^{\phz}
$$
which we denote $V \mapsto \fP_{\F}(V)$.  

\begin{thm} \label{thm:GHN} 
The HN-filtration on $\fP_{\F}(V)$ for all $V \in \Rep_{\F}(G)$ induces an exact $\otimes$-functor $\cF^{HN}: \Rep_{\F}(G) \ra \mathrm{Filt}_{\fS_{\F}}$ on $\fP_{\F}$.   
\end{thm}
\begin{proof}
The only remaining issue is whether $\cF^{HN}$ is strict in exact sequences.  The same argument as in \cite[Thm.\ 5.3.1]{DOR2010} works here as well. 
\end{proof}

Choose a $G$-bundle trivialization of $\fP_{\F}$.  The filtration $\cF^{HN}$ modulo $u$ gives rise to a filtration on the fiber functor, i.e., a tensor exact functor
\[
\omega^{HN}_G:\Rep_{\F}(G) \ra \mathrm{Filt}_{\F}.
\]
\begin{defn}
Associated to $\omega^{HN}_G$ is a conjugacy class of rational cocharacters in $X_*(T)_{\QQ}$.  Define the \emph{HN-vector}  $\lambda_{\HN}$ of $\fP_{\F}$  to be the dominant representative for this conjugacy class.   
\end{defn}

 \begin{prop} 
 Let $\fP_{\F}$ be $G$-Kisin module over $\F$.  If $\fP_{\F}$ has Hodge type $\nu$, then $g \cdot \lambda_{\HN} \prec \nu$. 
 \end{prop}
 \begin{proof} By Proposition \ref{functcrit}, it suffices to check the claim for any representation $V$ of $G$.  Since both the Hodge and HN-polygons are functorial in representations of $G$, we are reduced to Proposition \ref{HNoverHodge}.   
 \end{proof}

\begin{thm} \label{GHNstrata} Let $\cP = \fP_{\F}[1/u]$ for $\fP_{\F} \in \mathrm{GMod}^{\phz}_{\fS_{\F}}$.  There is a decomposition 
$$
X^{\nu}_{\cP} = \bigcup_{g \lambda \prec  \nu} X^{\nu, \lambda}_{\cP} 
$$
by locally closed reduced subschemes indexed by dominant $\lambda \in X_*(T)_{\QQ}$ such that the closed points of $X^{\nu, P}_{\cM}$ are $G$-Kisin modules with HN-vector $\lambda$. 
\end{thm}
\begin{proof}
By construction, the HN-vector of $\fP_{\F} \in X^{\nu}_{\cP}(\F)$ is determined by the HN-polygons of the collection $\{\fP_{\F}(V)\}_{V \in \Rep_{\F}(G)}$.   For any representation $\rho:G \ra \GL(V)$, we have an induced map
$$
\rho_*:X^{\nu}_{\cP} \ra X^{\rho(\nu)}_{\cP(V)}.
$$
The desired locally closed subscheme is then given by $X^{\nu, \lambda}_{\cP}$
\[
\bigcap_{\rho} \rho^{-1}_*( X^{\rho(\nu), P_{\rho(\lambda)}}_{\cP(V)})
\]
where $P_{\rho(\lambda)}$ is the polygon associated to the rational conjugacy class of $\rho(\lambda)$. The intersection is only a finite intersection since the conjugacy class of a fixed $\lambda$ is determined by its image under finitely many representations of $G$.
\end{proof}

We end with generalization of Corollary \ref{firstsegment} which includes as a special case the construction in \cite[\S 2.4.13]{mffgsm}. 

\begin{thm} 
\label{thm:HNcomp} 
For any dominant weight $\chi \in X^*(T)$, let $S_{\chi}$ denote the collection of $\lambda \in X_*^+(T)_ {\QQ}$ such that $\langle \chi, \nu - g \lambda \rangle = 0$.   Then, 
\[
X^{\nu}_{\cP, \chi} := \bigcup_{\lambda \in S_{\chi}} X^{\nu, \lambda}_{\cP}
\]
is both open and closed in $X^{\nu}_{\cP}$.  
\end{thm}

\begin{proof}
Let $V_{\chi}$ denote the representation of $G$ with highest weight $\chi$. Let $\rho:G \ra \GL(V_{\chi}^*)$ be the dual representation and set $\cM = \cP(V_{\chi}^*)$.  Consider the induced map on Kisin varieties
$$
\rho_*:X^{\nu}_{\cP} \ra X^{\rho(\nu)}_{\cM}.
$$  
The dual representation has lowest weight $-\chi$.  In particular, if we choose a dominant representative for $\rho(\nu)$, then $\rho(\nu) = (\langle -\chi, \nu \rangle, \ldots)$ and similarly for the HN-vector of any $\fP \in X^{\nu}_{\cP}$.  The condition that $\langle \chi, \nu - g \lambda \rangle = 0$ is exactly the condition that $\rho(\lambda) \in S_d$ for some $d \geq 1$ where $S_d$ is as in  Corollary \ref{firstsegment}.     Thus, 
\[
\bigcup_{\lambda \in S_{\chi}} X^{\nu, \lambda}_{\cP} = \rho^{-1} (\bigcup_{d\geq 1} X^{\rho(\nu)}_{\cM, d})
\]
which is open and closed. 
\end{proof}

If we assume that the adjoint group $G_{\ad}$ is simple of rank $\ell$, then we can make Theorem \ref{thm:HNcomp} more concrete. Let $\Sigma = \{ \chi_i \}_{1 \leq i \leq \ell}$ denote the fundamental weights. Then for any subset $J \subset \Sigma$, we can define an open and closed subscheme
\[
X^{\nu}_{\cP, J} := (\bigcap_{i \in J} X^{\nu}_{\cP, \chi_i}) \cap (\bigcap_{i \notin J}  (X^{\nu}_{\cP} \backslash X^{\nu}_{\cP, \chi_i})).
\]
Note that if $G$ is not simply connected, the fundamental representations may not be defined on $G$, but the representation with highest weight $n \chi_i$ will be for some integer $\chi$. Therefore, $X^{\nu}_{\cP, \chi_i}$ as defined in Theorem \ref{thm:HNcomp} still makes sense. 

\begin{rem}  
\label{rem:all_poly}
We expect that the $X^{\nu}_{\cP, J}$ captures all discrete invariants which arise from HN-polygons.   
\end{rem}

\begin{rem}  If one considers the case of $G = \GL_n$, the fundamental representations are exterior powers of the standard representation and so Theorem \ref{thm:HNcomp} is equivalent to the local constancy of the \'etale rank (Proposition \ref{prop:upperetrank}).
\end{rem}

\begin{rem}
\label{rem:not_all}
  Even when $G = \GL_2$, the HN-polygon alone is not sufficient to capture all connected components.  For example, when $\overline{\rho}$ is unramified of dimension two with distinct characters, the Kisin variety can have two ordinary components in addition to a non-ordinary component (see \cite[Thm.\ 1.2]{hellmann2009}).  Thus, we do not conjecture that  the $X^{\nu}_{\cP, J}$ are connected, though one might expect this to be true under suitable hypotheses on $\cP$. 
\end{rem}

\section{HN-theory applied to flat Kisin modules} 
\label{sec:flat}

Recall that a \emph{flat} Kisin module is a Kisin module whose underlying $\fS$-module is flat, with the point of emphasis being that it is $p$-torsion free. In this section, our goal is to deduce consequences about flat Kisin modules from the Harder--Narasimhan theory for torsion Kisin modules that we developed in \S\ref{sec:HN}.  One nice consequence is that one gets distinguished representatives of an isogeny class whose HN-polygon does not change under reduction mod $p$  (Theorem \ref{thm:HN_isog}).   In \S 6.5, we show that if $\fM$ is a flat Kisin module, then our HN-filtration on $\fM/p^n\fM$ is an approximation to the HN-filtration which comes from Hodge--Tate filtration on $p$-adic representations of $\Gamma_K = \Gal(\overline{K}/K)$. In particular, see Theorem \ref{thm:HN-HT}. 

\subsection{Deformation theory}

The main case of interest is when we have a flat Kisin module $\fM$, and we study $\fM/p^n\fM$ for $n \geq 1$ deforming $\fM/p\fM$. In this section, we will give a brief indication of what can be said about deformations of a $p$-torsion Kisin module to an $A$-Kisin module for general $A$. Afterward we will focus only on the case $A = \bZ_p/p^n$. 

Let $A$ be a local Artinian $\bZ_p$-algebra with residue field $A/\fm_A \cong \bF$, a finite extension of $\bF_p$. Recall that $\Mod_{\fS_A}^\phz$ refers to the category of $A$-Kisin modules $(\fM_A, \phi_{\fM_A})$, and the underlying $\fS_A$-module $\fM_A$ is projective of constant rank.  
\begin{defn}
Let $\fM_\bF \in \Mod_{\fS_\F}^\phz$. A \emph{deformation} of $\fM_\bF$ to $A$ is an $A$-Kisin module $\fM_A \in \Mod_{\fS_A}^\phz$ such that there is a $\Mod_{\fS_\bF}^\phz$-isomorphism $\fM_A \otimes_A \bF \simeq \fM_\bF$. 
\end{defn}

Here we give the following general description of the behavior of HN-filtrations under deformation to Artinian local $A$ as above. 

\begin{prop}
\label{prop:def_main}
Fix a Kisin module $\fM_\bF \in \Mod^\phz_{\fS_\F}$ with its HN-filtration 
\[
\fM_{\F, 0} = 0 \subset \fM_{\F, 1} \subset \fM_{\F,2} \subset \ldots \subset \fM_{\F, n} = \fM_{\F},
\]
into strict subobjects in $\Mod_{\fS_\F}^\phz$. Let $\fM_A$ be a deformation of $\fM_\bF$ to $A$. Then the HN-filtration on $\fM_A$ into objects of $\Mod_{\fS, \tor}^\phz$ 
\[
\fM_{A, 0} = 0 \subset \fM_{A, 1} \subset \fM_{A,2} \subset \ldots \subset \fM_{A, m} = \fM_{A},
\]
consists of generalized $A$-Kisin modules and $\mu(\fM_A) = \mu(\fM_\bF)$. Moreover, 
\begin{equation}
\label{eq:end_slopes}
\mu(\fM_{\F, 1}) = \mu(\fM_{A,1}) \quad \text{and} \quad \mu(\fM_{\F,n}) = \mu(\fM_{A,m}).
\end{equation}
In particular, $\fM_\bF$ is semi-stable if and only if $\fM_A$ is semi-stable. 
\end{prop}

Recall the terminology ``generalized $A$-Kisin module'' from \S\ref{subsec:KM_coefs}. In particular, the $\fM_{A,i}$ need not be projective $\fS_A$-modules. 

\begin{proof}
The argument in the beginning of the proof of Proposition \ref{prop:HNfiltoverF} applies with $A$ in the place of $\F$ to show that each $\fM_{A,i}$ is $A$-stable. 

Let $\fm_A \subset A$ be the maximal ideal, and let $I \subset A$ be an ideal such that $\fm_A \cdot I = 0$. Notice that we have a short exact sequence of $A$-Kisin modules
\begin{equation}
\label{eq:def-SES}
0 \lra I \cdot \fM_A \lra \fM_A \lra \fM_{A/I} \lra 0,
\end{equation}
Using the fact that $\fM_A$ is flat over $\fS_A$ and the $\fS_A$-module structure on $I\cdot \fM_A$ factors through $\fS_\bF$, we have a canonical isomorphism of Kisin modules $I \cdot \fM_A \simeq I \otimes_\bF \fM_\bF$. 

First we prove that $\mu(\fM_A) = \mu(\fM_\F)$ and also prove that $\fM_\F$ is semi-stable if and only if $\fM_A$ is semi-stable. We induct on the order of nilpotence $i$ of the maximal ideal $\fm_A \subset A$. When $i=1$, there is nothing to prove. Assuming the result for some $i$ and assuming that $\fm_A^{i+1} = 0$, we have that $I = \fm_A^i$ satisfies $\fm_A \cdot I = 0$. Then \eqref{eq:def-SES} is an extension of Kisin modules of identical slope, so the extension has the same slope. Then, we know from Proposition \ref{prop:ext_ss} that $\fM_A$ is semi-stable if and only if $I \cdot \fM_A$ and $\fM_{A/I}$ are semi-stable. 

Similarly, \eqref{eq:end_slopes} can be proved by induction on the order of nilpotence of the maximal ideal of $A$. Thus we begin our proof by letting $\fM_A$ be some deformation of $\fM_\bF$ to $A$, letting $I \subset A$ be an ideal such that $\fm_A \cdot I = 0$, and assuming \eqref{eq:end_slopes} for $\fM_{A/I} := \fM_A \otimes_A A/I$. 

Let $\mu = \mu(\fM_{A,1})$ be the least slope of the HN-filtration of $\fM_A$. Then $\mu < \mu(\fM_{\F, 1})$ is impossible, because in this case Proposition \ref{prop:HN_funct} implies that $\fM_A^{\leq \mu}$ maps to $0$ in $\fM_{A/I}$, meaning that $\fM_A^{\leq \mu} \subset \fM_F \otimes_\bF I$, which is clearly not possible as the least slope of any submodule is also $\mu(\fM_{\F, 1})$. We may then conclude that $\mu = \mu(\fM_{\F, 1})$, because $\fM_{\bF,1} \otimes_\bF I \subset I\cdot \fM_A \subset \fM_A$ has slope $\mu(\fM_{\F, 1})$. Therefore, the first step $\fM_{A,1}$ in the HN-filtration of $\fM_A$ has slope $\mu(\fM_{\F, 1})$. 

Similarly, let $\mu = \mu(\fM_{A,m})$ be the greatest slope of $\fM_A$, and write $\mu'$ for $\mu(\fM_{\F,n})$. Assuming $\mu > \mu'$, then $I \otimes_\F \fM_\F$ maps to $0$ in $\fM_A/\fM_A^{\leq \mu'}$ by Proposition \ref{prop:HN_funct}, as all of the slopes of $\fM_A/\fM_A^{\leq \mu'}$ are strictly greater than those of $\fM_\F$. This makes $\fM_A/\fM_A^{\leq \mu'}$ a quotient of $\fM_{A/I}$. However, all of the slopes of $\fM_{A/I}$ are $\leq \mu'$ by the induction hypothesis, a contradiction. 

To see that $\mu= \mu'$, note that $\fM_A$ surjects onto $\fM_\F/\fM_{\F,n-1}$, while Proposition \ref{prop:HN_funct} would imply that the image of the surjection is $0$ if all of the slopes of $\fM_A$ are $< \mu'$. 
\end{proof}

\subsection{The limit of HN-polygons}

In order to understand the implications of this Harder--Narasimhan theory to flat Kisin modules $\fM$, we will study the behavior of HN-polygons of $\fM/p^n$ as $n$ varies. Here we will assume that $\fM$ is flat as a $\fS$-module and use the fact that $\fM/p^n$ is an extension of $\fM/p^{n-m}$ by $\fM/p^m$. We begin with this general lemma about the behavior of HN-polygons under extension, following Fargues \cite[\S4]{fargues2010}. 

\begin{defn}
Given functions $f: [0,r] \ra [0,d]$, $f': [0,r'] \ra [0,d']$, let $f \ast f' : [0, r+r'] \ra [0,d+d']$ be defined by
\[
x \mapsto \inf f(a) + f'(b)
\]
where the infimum is taken over $(a,b) \in [0,r] \times [0,r']$ such that $a+b = x$. 
\end{defn}
One can check that $f \ast f'$ is convex when $f$ and $f'$ are convex. 

\begin{lem}
\label{lem:ext_HN}
Let $0 \ra \fM' \ra \fM \ra \fM'' \ra 0$ be a short exact sequence of Kisin modules. Then 
\[
\HN(\fM) \geq \HN(\fM' \oplus \fM'') = \HN(\fM') \ast \HN(\fM''),
\]
\end{lem} 
Moreover, there are some cases where $\HN(\fM)$ can be determined from $\HN(\fM')$ and $\HN(\fM'')$.
\begin{proof}
The inequality is \cite[Prop.\ 7.6(a)]{pottharst}. The equality follows from the fact that the HN-filtration of a direct sum of objects consists of the direct sum of each part of the HN-filtration (indexed by slopes), as in \cite[Prop.\ 7.6(b)]{pottharst}. 
\end{proof}

Now we may define the normalized HN-polygon of a flat Kisin module, in analogy to Fargues's definition of a normalized HN-polygon of a $p$-divisible group (\cite[\S4.2, Thm.\ 2]{fargues2010}). 
 
\begin{thm}
\label{thm:HNP_limit}
The sequence of functions from $[0, \rk_\fS \fM]$ to $[0, \deg(\fM/p)]$ given by 
\[
x \mapsto \frac{1}{n} \HN(\fM/p^n)(nx)
\]
converges uniformly increasingly as $n \ra \infty$ to a continuous convex function 
\[
\HN(\fM) : [0, \rk_\fS \fM] \lra [0, \deg(\fM/p)]
\]
such that $\HN(\fM)(0) = 0$ and $\HN(\fM)(\rk_\fS \fM) = \deg(\fM/p)$. 
\end{thm}

\begin{proof}
The uniform increasing convergence of functions $(f_n)_{n \geq 1}$ satisfying the property $f_{n+m} \geq f_n \ast f_m$ for all $n,m \geq 1$ is shown in \cite[\S1.1, Prop.\ 2]{fargues2019}. By Lemma \ref{lem:ext_HN} applied to the exact sequence
\[
0 \lra \fM/p^{m}\fM \buildrel{\cdot p^n}\over\lra \fM/p^{n+m}\fM \lra \fM/p^m\fM \lra 0,
\]
the functions $f_n = \HN(\fM/p^n)(nx)/n$ satisfy this property. 
\end{proof}

It will be useful to know that this limit does not change under an \emph{isogeny}, i.e., a strict map $\fM \ra \fN$ of flat Kisin modules whose kernel and cokernel are $p$-power torsion Kisin modules (i.e., $u$-torsion free). 
\begin{lem}[{(cf.\ \cite[Prop.\ 3, \S1.3]{fargues2019})}]
\label{lem:HNP_isog_invt}
When $\fM$ and $\fM'$ are isogenous flat Kisin modules, $\HN(\fM) = \HN(\fM')$.
\end{lem}
\begin{proof}
We may assume that our isogeny is an injection $f: \fM' \rinj \fM$ with cokernel $\fN$, a $p$-power torsion Kisin module. Writing $\fP$ for the flat Kisin module $\fP := (\fM \oplus \fM')/(f \oplus \mathrm{id}_{\fM'})(\fM')$. We see that for large enough $n$, we have exact sequences
\[
0 \ra \fN \ra \fP/p^n \ra \fM/p^n \ra 0, \qquad 0 \ra \fM'/p^n \ra \fP/p^n \ra \fN \ra 0.
\]

Using Lemma \ref{lem:ext_HN} and the leftmost exact sequence, we find that 
\[
\HN(\fP/p^n) \geq \HN(\fN) \ast \HN(\fM/p^n). 
\]
Let $C \geq 0$ be an upper bound on $\HN(\fM/p^n)(x) - \HN(\fN)(x)$ for $x \in [0, \rk \fN]$ and all $n \geq 1$. We may select $C$ that satisfies this inequality for all $n \geq 1$ because the slopes of $\fM/p^n$ are uniformly bounded. We deduce that $\HN(\fP/p^n)(nx) \geq \HN(\fM/p^n)(nx) - C$ for all $x \in [0, \rk \fM]$. 

The inclusion $\fM'/p^n \rinj \fP/p^n$ implies that $\HN(\fM'/p^n) \geq \HN(\fP/p^n)$, in light of Proposition \ref{chull1} and the fact that any strict subobject of $\fM'/p^n$ is a strict subobject of $\fP/p^n$. Altogether, one has
\[
\frac{1}{n}\HN(\fM'/p^n)(nx) \geq \frac{1}{n} \HN(\fM/p^n)(nx) - \frac{C}{n},
\]
so $\HN(\fM') \geq \HN(\fM)$. 

Because isogeny is a symmetric relation, the lemma follows. 
\end{proof}

\subsection{An intrinsic notion of degree of a flat Kisin module}

We may now define a notion of degree of a flat Kisin module, which will be shown to be compatible with the limit HN-polygon discussed in Theorem \ref{thm:HNP_limit}. First we recall this standard result from commutative algebra, often applied in Iwasawa theory. 
\begin{prop}
\label{prop:char_form}
Let $M$ be a finitely generated torsion $\fS$-module. Then $M$ is pseudo-isomorphic to a module of the form
\begin{equation}
\label{eq:char_form}
\bigoplus_{i=1}^n \frac{\fS}{(f_i^{h_i})}
\end{equation}
where $f_i$ varies over distinguished polynomials (i.e.~generators of the height 1 primes of $\fS$) and a \emph{pseudo-isomorphism} is a morphism of $\fS$-modules with finite cardinality kernel and cokernel. 
\end{prop}

We also require the notion of characteristic ideal of a torsion $\fS$-module. 
\begin{defn}
Let $M$ be a $\fS$-module pseudo-isomorphic to the module of \eqref{eq:char_form}. Then the characteristic ideal $\Char(M)$ of $M$ is the ideal of $\fS$ generated by $\prod_{i=1}^n f_i^{h_i}$. 
\end{defn}

Because the cokernel of $\phi_\fM$ is guaranteed to be $E(u)$-power torsion by definition of Kisin modules, all of the distinguished polynomials $f_i$ appearing in its characteristic ideal are equal to $E(u)$. This allows us to make the following definition of degree. 
\begin{defn}
The degree $\deg_\fS(\fM)$ of a flat effective Kisin module $(\fM, \phi_\fM)$ is the non-negative integer characterized by the equality
\[
\Char(\coker(\phi_\fM)) = (E(u)^{\deg_\fS(\fM)}).
\]
\end{defn}
This notion of degree may be extended to flat non-effective Kisin modules by twisting as Definition \ref{defn:rk_deg}.

Compatibility between $\deg_\fS(\fM)$ and $\deg(\fM/p^n)$ takes the following form. 
\begin{prop}
\label{prop:deg_mod}
Let $\fM$ be a flat Kisin module. Then 
\[
\deg_\fS(\fM) = \deg(\fM/p^n)/n \quad \text{ and } \quad \mu_\fS(\fM) = \mu(\fM/p^n)
\]
for all $n \in \bZ_{\geq 1}$. 
\end{prop}
\begin{proof}
Because $\coker(\phi_\fM)$ has a two-step projective resolution by $\phi_\fM$ itself, it has homological dimension 1. Consequently, as it is $E(u)$-power torsion, it has no $p$-torsion, and the quasi-isomorphism of Proposition \ref{prop:char_form} is an injection 
\begin{equation}
\label{eq:char_form2}
\coker(\phi_\fM) \rinj \bigoplus_{i=1}^m \frac{\fS}{(E(u)^{h_i})} =: T 
\end{equation}
with finite cardinality cokernel $C$. One then checks that the cokernel of the induced injective map $\phi_{\fM/p^n} : \phz^*\fM/p^n \ra \fM/p^n$ is naturally isomorphic to $\coker(\phi_\fM)/p^n$. 

The exact sequence
\[
0 \lra C[p^n] \lra \coker(\phi_\fM)/p^n \lra T/p^n \lra C/p^n \lra 0
\]
allows us to check that $\ell_{\bZ_p}( \coker(\phi_\fM)/p^n) = \ell_{\bZ_p}(T/p^n)$. Recalling that $\deg_\fS(\fM) = \sum_{i=1}^m h_i$, we calculate 
\[
g\cdot \deg(\fM/p^n) = \ell(\coker \phi_{\fM/p^n}) = \ell(\coker(\phi_\fM)/p^n) = \ell(T/p^n)  = g\deg_\fS(\fM)n,
\]
where the last inequality follows from the fact that $\fS/(E(u)^{h_i}, p^n)$ is length $neh_i$ as a $\cO_K$-module. The equality of slopes follows from $\rk_\fS(\fM) = \rk(\fM/p)$ and $\rk(\fM/p^n) = n\rk(\fM/p)$. 
\end{proof}

As a consequence of Proposition \ref{prop:deg_mod}, we can show that our Harder--Narasimhan theory restricted to $\Mod_{\fS, \tor}^{\phz, [0,1]}$ is the same as that on finite flat group schemes from \cite{fargues2010}. Let $\cG \mapsto \fM(\cG)$ denote the anti-equivalence from the category from $p^\infty$-torsion finite flat group schemes over $\cO_K$ to $\Mod_{\fS}^{\phz, [0,1]}$ as in Remark \ref{rem:anti-equiv}. 
\begin{prop} 
\label{farguescomp}  
Let $\cG$ be a finite flat group scheme. If $\cG_i$ is the HN-filtration on $\cG$ from \cite[\S4]{fargues2010}, then $\fM(\cG^i)$ is the HN-filtration of $\fM(\cG)$. 
\end{prop}

\begin{proof}
It suffices to compare the degrees of $\cG$ and $\fM(\cG)$.  We can write $\cG = \ker(\Psi:\cG_1 \ra \cG_2)$ where $\Psi$ is an isogeny of $p$-divisible groups.  For some $n \gg 0$, we have
\[
0 \ra \cG \ra \cG_1[p^n] \ra \cG_2[p^n] \ra 0.
\]  
Using the behavior of degree in exact sequence, we reduce to the case of $\cG_i[p^n]$, i.e., truncated Barsotti--Tate groups.  By \cite[\S3, Ex.\ 2]{fargues2010}, $\deg(\cG_i[p^n]) = n d_i$ where $d_i$ is the dimension of $\cG_i$. 

Let $\fM_i = \fM(\cG_i)$.  Since $\fM_i$ is a flat Kisin module with height in $[0,1]$, each $h_i $ in \eqref{eq:char_form2} is either $0$ or $1$.  If $d_i'$ is the number of non-zero $h_i$, then $\deg(\fM_i/p^n) = n d_i'$.  By comparison of Hodge--Tate weights (\cite[Cor.\ 2]{tate1967}), we conclude that $d_i  = d_i'$ and hence $\mu(\cG_i[p^n]) = \mu(\fM_i/p^n)$.  
\end{proof}

We may now define a new Harder--Narasimhan theory on the isogeny category of flat Kisin modules. Recalling the formal setup of \S\ref{subsec:HN-theory}, the target of the generic fiber functor will be the category of \'etale $\phz$-modules over $\cE$, where $\cE := \cO_\cE[1/p]$. 
\begin{defn}
An \emph{\'etale $\phz$-module over $\cE$} is a finite dimensional $\cE$-vector space $D$ equipped with a semi-linear endomorphism $\phz$ such that there exists a $\phz$-stable $\cO_\cE$-lattice $\cM \subset D$ that is an \'etale $\phz$-module over $\cO_\cE$. We write the abelian category of such objects as $\Mod^\phz_\cE$. 
\end{defn}
The category $\Mod^\phz_\cE$ is equivalent to the category $\bQ_p$-linear continuous representations of $G_{K_\infty}$, cf.\ Remark \ref{rem:JMF}. 

\begin{defn}
The rank $\rk_\fS(\fM)$ of a flat Kisin module $(\fM, \phi_\fM)$ is the dimension over $\cE$ of the \'etale $\phz$-module $\fM \otimes_\fS \cE$. The slope is given by $\mu_\fS(\fM) := \deg_\fS(\fM)/\rk_\fS(\fM)$. 
\end{defn}
This rank is equal to the rank of $\fM$ as a $\fS$-module. 

When we try to verify the Harder--Narasimhan axioms listed in Proposition \ref{prop:HN_main} for the additive exact category of flat Kisin modules (as opposed to its isogeny category), we find that they are not quite true. Namely, non-strict subobjects inducing the same generic fiber do not necessarily have strict inequality in the degree. To have an HN-theory, we pass to isogeny category. 
\begin{prop}
\label{prop:HN_false}
Let rank and degree be as defined above. 
\begin{enumerate}
\item The functor $\otimes_\fS \cE$ sending flat Kisin modules to \'{e}tale $\phz$-modules over $\cE$ is exact and faithful, and induces a bijection
\[
\{\text{strict subobjects of } \fM\} \lrisom \{\text{subobjects of } \fM \otimes_\fS \cE\}
\]
\item Both $\rk_\fS$ and $\deg_\fS$ are additive in short exact sequences, and $\rk_\fS(M) = 0 \iff M = 0$. 
\item If an injection $\fM' \iarrow \fM$ induces an isomorphism $\fM' \otimes_\fS \cE \lrisom \fM \otimes_\fS \cE $, then $\fM'[1/p] \risom \fM[1/p]$ and $\deg_\fS(\fM') = \deg_\fS(\fM)$. 
\end{enumerate}
\end{prop}
\begin{proof}
We see in \cite[Prop.\ 2.1.12]{crfc} that $\otimes_{\fS} \cO_\cE$ is fully faithful between the exact categories of flat Kisin modules and flat \'etale $\cO_\cE$-modules. Then, applying $\otimes_{\cO_\cE} \cE$ is clearly faithful and exact.  $\cE$ is a flat $\fS$-module just as $\cO_\cE$ is, and $\fM$ is now assumed to be $\fS$-flat just as $\fM$ was assumed to be $u$-torsion free before. The bijection from strict subobjects to subobjects follows from the fully faithful property of $\otimes_\fS \cO_\cE$ along with \cite[Lem.\ 2.1.15]{crfc}. 

One may quickly check that $\rk_\fS$ and $\deg_\fS$ are additive in short exact sequences. 

Let $\fM' \rinj \fM$ be an injection of flat Kisin modules as in (iii) with cokernel $\fN$. Clearly $\fN$ is a torsion $\fS$-module, and by \emph{loc.\ cit.}, $\fN$ is $p$-power torsion. Moreover, the two-step presentation of $\fN$ by flat $\fS$-modules implies that it has homological dimension 1, and consequently has no $u$-torsion. 

Just as in the proof of Proposition \ref{prop:HN_main}, we derive an exact sequence 
\[
0 \lra \ker (\phi_{\fN}) \lra \coker (\phi_{\fM}) \lra \coker (\phi_{\fM'}) \lra \coker(\phi_{\fN}) \lra 0
\]
from the snake lemma. Because $\fN$ and $\phz^*\fN$ have characteristic ideals that are powers of $(p)$, the same is true for $\coker(\phi_{\fN})$ and $\ker(\phi_\fN)$, as characteristic ideals are multiplicative in exact sequences. Therefore the powers of $(E(u))$ in the characteristic ideals of $\coker(\phi_{\fM'})$ and $\coker(\phi_{\fM})$ are identical, i.e., $\deg_\fS(\fM') = \deg_\fS(\fM)$, as desired. 
\end{proof}

It follows from Proposition \ref{prop:HN_false} that the essential image of the isogeny category $\Mod_\fS^{\phz} \otimes \bQ_p$ in $\Mod^\phz_\cE$ is an abelian category equipped with a notion of rank and degree. Consequently, the identity functor from $\Mod_\fS^{\phz} \otimes \bQ_p$ to itself satisfies the Harder--Narasimhan axioms for $\rk_\fS$ and $\deg_\fS$. Indeed, note that $\rk_\fS(\fM)$ and $\deg_\fS(\fM)$ depend only upon $\fM[1/p]$ when $\fM$ is a flat Kisin module. 

Consequently, we have the standard result that there is a unique filtration 
\begin{equation}
\label{eq:HN_filt_isog}
0 = \fM[1/p]_0 \subset \fM[1/p]_1 \subset \dotsc \subset \fM[1/p]_r = \fM[1/p]
\end{equation}
such that $\mu_\fS(\fM[1/p]_{i+1}/\fM[1/p]_i) > \mu_\fS(\fM[1/p]_i/\fM[1/p]_{i-1})$. 

\begin{lem}
\label{lem:K_ss}
Let $\fM$ be a flat Kisin module. The following are equivalent.
\begin{enumerate}
\item $\fM/p$ is semi-stable.
\item $\fM/p^n$ is semi-stable for all $n \geq 1$. 
\item For any $p$-power torsion quotient Kisin module $\fN$ of $\fM$, $\mu(\fN) \leq \mu_\fS(\fM)$.
\end{enumerate}
\end{lem}
\begin{proof}
The equivalence of (i) and (ii) follows from Proposition \ref{prop:def_main}.

By Proposition \ref{prop:deg_mod}, we have $\mu(\fM/p^n) = \mu_\fS(\fM)$. Any $p$-power torsion Kisin module of $\fM$ is a quotient Kisin module of $\fM/p^n$ for some $n$. We immediately get the equivalence of (ii) and (iii). 
\end{proof}

\begin{defn}
Call a flat Kisin module $\fM$ \emph{semi-stable} when any of the equivalent conditions of Lemma \ref{lem:K_ss} are true. 
\end{defn}

\begin{rem}
Notice that one can immediately derive the tensor product theorem for flat Kisin modules from the equivalence of Lemma \ref{lem:K_ss} and the tensor product theorem for $p$-torsion Kisin modules (Theorem \ref{thm:tensor}). 
\end{rem}

This notion of semi-stability is different than the notion native to the isogeny category given in \eqref{eq:HN_filt_isog}, and they have the following relation to each other. 
\begin{prop}[{(cf.\ \cite[Prop.\ 6, \S2.1]{fargues2019})}]
\label{prop:ss_isog}
Let $\fM$ be a flat Kisin module. The following statements are equivalent.
\begin{enumerate}
\item The isogeny class of $\fM$ has a semi-stable member, which we call $\fN$. 
\item For any strict subobject $\fM'[1/p]$ of $\fM[1/p]$, $\mu_\fS(\fM') \geq \mu_\fS(\fM)$, i.e.~$\fM[1/p]$ is semi-stable in $\Mod^{\phz}_\fS \otimes \bQ_p$. 
\end{enumerate}
\end{prop}

\begin{proof}
We have observed that the slope of a flat Kisin module depends only on its isogeny class. Assuming (i) and choosing a strict subobject $\fM' \subset \fM$, it follows that $\mu_\fS(\fM') = \mu_\fS(\fM'[1/p] \cap \fN)$, where this intersection is taken after choosing an isomorphism $\fN[1/p] \simeq \fM[1/p]$. Because $\fN$ is semi-stable, so is $\fN/p$. Consequently, $\deg((\fM'[1/p] \cap \fN)/p) \geq \deg(\fN/p)$. Then (ii) follows, by using Proposition \ref{prop:deg_mod} to calculate degrees of flat Kisin modules by calculating their degree modulo $p$.

Now assume (ii). If $\fM$ is semi-stable we have established (i), so we address the case that $\fM$ is not semi-stable. Lemma \ref{lem:K_ss} tells us that $\deg(\fM/p^n)$ is not semi-stable for some $n \geq 1$. Consequently, there exists some $p$-power torsion quotient Kisin module of $\fM$ with all HN-slopes strictly greater than $\mu_\fS(\fM)$. Consider the inverse system of all such quotients of $\fM$. We claim that the HN-rank of these quotients is bounded above. For if it is not bounded above, the fact that the number of generators as a $\fS$-module is bounded above implies that there exists a $p$-torsion free quotient Kisin module $\fM''$ of $\fM$, realizable as a quotient of the inverse limit of the inverse system. We then have that $\mu_\fS(\fM'') > \mu_\fS(\fM)$, because we can calculate these slopes modulo $p$, and $\mu(\fM''/p) > \mu(\fM/p)$ by assumption. Using the additivity of $\deg_\fS$ in exact sequences, we deduce that $\fM[1/p]$ is not semi-stable, a contradiction. 

There is a unique maximal element $\fP$ to the inverse system we have defined. One way to see this is to start by applying the bound on the rank of members of the system to deduce that all of them are quotients of $\fM/p^n$ for some sufficiently large $n$. Then, $\fP$ is the quotient of $\fM/p^n$ by the largest step of the HN-filtration with slopes $\leq \mu_\fS(\fM)$. By Proposition \ref{prop:HN_funct}, all elements of the inverse system are quotients of $\fP$. 

Now consider the kernel Kisin module $\fN$ of the surjection $\fM \rsurj \fP$. We know that $\mu_\fS(\fM) = \mu_\fS(\fN) = \mu(\fN/p^n)$, and that any quotient of $\fN/p^n$ has slope no more than $\mu_\fS(\fM)$. Consequently, $\fN/p^n$ is semi-stable and, by Lemma \ref{lem:K_ss}, $\fN$ is semi-stable and isogenous to $\fM$, as desired. 
\end{proof}

\subsection{Flat Kisin modules of type HN}

Not all flat Kisin modules have a filtration by strict semi-stable subobjects satisfying the axioms of the HN-filtration (since, a priori, the filtration lives in the isogeny category). Those which do are called \emph{type HN}, following \cite[Defn.\ 5, \S2.3]{fargues2019}. 
\begin{defn}
\label{defn:type_HN}
A flat Kisin module $\fM$ is said to be of \emph{type HN} when there exists an increasing filtration $\fM_i$ by strict subobjects such that each graded piece $\fM_i/\fM_{i-1}$ is a semi-stable flat Kisin module and 
\begin{equation}
\label{eq:HN_type_filt}
\mu (\fM_1/\fM_0) < \mu(\fM_2/\fM_1) < \dotsm < \mu (\fM_r/\fM_{r-1}).
\end{equation}
\end{defn}

The following observation will be useful. 
\begin{lem}
\label{lem:HN_iff_eq}
A flat Kisin module $\fM$ is of type HN if and only if $\HN(\fM) = \HN(\fM/p)$. 
\end{lem}
\begin{proof}
If $\fM$ is of type HN, then Lemma \ref{lem:K_ss} implies that the reduction modulo $p^n$ of the filtration \eqref{eq:HN_type_filt} has the properties of the canonical HN-filtration of $\fM/p^n$ for every $n \geq 1$. It follows that $\HN(\fM) = \HN(\fM/p)$. Conversely, $\HN(\fM) = \HN(\fM/p)$ implies that the increasing sequence of HN-polygons of Theorem \ref{thm:HNP_limit} is a constant sequence. It follows that $\rk((\fM/p^n)^{\leq \mu}) = n \rk((\fM/p)^{\leq \mu}$ for any $\mu$ and all $n \geq 1$, and that $(\fM/p^n)^{\leq \mu}$ is an extension of $\fM/p^{n-1}$ by $\fM/p$. Setting $\fM^{\leq \mu} := \varprojlim_n (\fM/p^n)^{\leq \mu}$, one can easily check that $\fM^{\leq \mu}/\fM^{< \mu}$ is a flat semi-stable Kisin module, so that the $\fM^{\leq \mu}$ realize the filtration of Definition \ref{defn:type_HN}.
\end{proof}

A generalization of Proposition \ref{prop:ss_isog} is to show that 
\begin{thm}
\label{thm:HN_isog}
Any flat Kisin module is isogenous to a flat Kisin module of type HN. 
\end{thm}

The theorem will follow from this generalization of the ``descent algorithm'' of \cite[\S3]{fargues2019}.
\begin{lem}
\label{lem:fargues3.1}
Let $(\fM_i)_{i=1}^r$, $(\fM'_i)_{i=1}^{r-1}$, $(\fQ_i)_{i=1}^{r-1}$ be sequences of flat Kisin modules such that $\fM = \fM_1$. For each $i = 1, \dotsc, r-1$, assume that we have an isogeny $\fM'_i \ra \fM_i$ and an exact sequence 
\begin{equation}
\label{eq:isogES}
0 \lra \fM_{i+1} \lra \fM'_i \lra \fQ_i \lra 0.
\end{equation}
Then $\fM$ is isogenous to a flat Kisin module $\fN$ with an increasing filtration such that $\Fil^0 \fN = 0$, $\Fil^r \fN = \fN$, and for $i = 1, \dotsc, r-1$, 
\[
\Fil^i \fN/\Fil^{i+1} \fN \simeq \fQ_i, \quad \text{ and } \Fil^r\fN / \Fil^{r-1} \fN \simeq \fM_r.
\]
\end{lem}
\begin{proof}
Using induction on $r$, and applying the result for $r-1$ to the flat Kisin module $\fM_2$, we have an isogeny $\fN_2 \rinj \fM_2$ and a filtration on $\fN_2$ such that $\Fil^i \fN_2 / \Fil^{i+1} \fN_2 = \fQ_i$ for $i = 2, \dotsc, r$ and $\Fil^1 \fN_2 = 0$. We can now construct $\fN$ lying in a diagram 
\[
\xymatrix{
0 \ar[r] & \fM_2 \ar[r] \ar[d] & \fM'_1 \ar[r] \ar[d] & \fQ_1 \ar[r] \ar[d]^\sim & 0 \\
0 \ar[r] & \fN_2 \ar[r] & \fN \ar[r] & \fQ_1 \ar[r] & 0 
}
\]
by defining it to be the quotient of $\fM'_1 \oplus \fN_2$ by the diagonal image of $\fM_2$. The desired filtration on $\fN$ is the pullback of the filtration on $\fN_2$. 
\end{proof}

\begin{proof}[Proof of Theorem \ref{thm:HN_isog}.]
We will apply Lemma \ref{lem:fargues3.1}, where the $\fQ_i$ will be semi-stable flat Kisin modules and the $\fM_i$ arise from the HN-filtration \eqref{eq:HN_filt_isog} on $\fM$. The maximal slope graded factor $\gr_r \fM = \fM/\fM_{r-1}$ induces a quotient flat Kisin module $\fM \rsurj \gr_r \fM$. By Proposition \ref{prop:ss_isog}, $\gr_r \fM$ is isogenous to a semi-stable Kisin module $\fQ_r$ via $\fQ_r \rinj \gr_r \fM$. Now let $\fM'_r$ be the kernel of $\fM_r \ra \gr_r \fM/\fQ_r$. We see that we have a short exact sequence of the form \eqref{eq:isogES} where $i = r$. 

We iterate this same argument, applying it to $\fM_{i}$ in the place of $\fM = \fM_r$ for $i$ from $r-1$ to $1$. The hypotheses of Lemma \ref{lem:fargues3.1} are verified, so because the $\fQ_i$ are semi-stable flat Kisin modules, we have shown that $\fM$ is isogenous to a flat Kisin module of type HN. 
\end{proof}

This HN-polygon is the limit of the HN-polygons of the $p^n$-torsion quotients of a flat Kisin module. 
\begin{prop}
When $\fM$ is a flat Kisin module, we have an equality $\HN_\fS(\fM[1/p]) = \HN(\fM)$ of polygons.
\end{prop}
\begin{proof}
First we remark that we may replace $\fM$ with an isogenous flat Kisin module of type HN. Such a replacement exists by Theorem \ref{thm:HN_isog}. The polygon $\HN_\fS(\fM[1/p])$ remains the same by Proposition \ref{prop:HN_false}(iii), and $\HN(\fM)$ remains the same by Lemma \ref{lem:HNP_isog_invt}. 

By Lemma \ref{lem:HN_iff_eq}, $\HN(\fM) = \HN(\fM/p)$. The argument of Lemma \ref{lem:HN_iff_eq} also implies that $\HN(\fM/p) = \HN_\fS(\fM[1/p])$. 
\end{proof}

\subsection{Comparison with the Hodge--Tate polygon}

In this section, we work with $p$-adic Galois representations, which we take to be continuous Galois group actions on finite-dimensional as $\Q_p$-vector spaces without further comment. We show that the HN-polygon of a flat Kisin module is a generalization of the Hodge--Tate polygon of a Hodge--Tate $\Gamma_K$-representations. A priori, the Galois representations to which this HN-polygon applies are the $\Gamma_{K_\infty}$-representations of bounded $E$-height, as this category is equivalent to the isogeny category of flat Kisin modules. In the case that a $\Gamma_{K_\infty}$-representation arises from a crystalline $\Gamma_K$-representation $\rho$ (and consequently has finite $E$-height), the HN-polygon lies above the Hodge--Tate polygon of $\rho$, generalizing an equality in the $p$-divisible group case due to Fargues \cite[\S5.3, Cor.\ 2]{fargues2019}. Note that in this crystalline case, $\rho$ is characterized by $\rho\vert_{\Gamma_{K_\infty}}$ \cite[Thm.\ 0.2]{crfc}. 

First we recall fundamental definitions in $p$-adic Hodge theory. Let $\mathbb{C}_p$ be the completion of an algebraic closure of $K$, which admits a natural $K$-linear action of $\Gamma_K$. Let $t$ be a basis vector for $\mathbb{C}_p(1)$, the base change to $\mathbb{C}_p$ of the $\bQ_p$-linear cyclotomic character with the natural semi-linear action of $\Gamma_K$. Let $B_{HT} = \bigoplus_{i \in \bZ}  \mathbb{C}_p(i) \simeq  \mathbb{C}_p(t, 1/t)$ be the $\bZ$-graded Hodge--Tate period ring. 

When $V$ is a $\bQ_p$-linear representation of $\Gamma_K$, $V$ is called \emph{Hodge--Tate} when $\dim_K (V \otimes_{\bQ_p} B_{HT})^{\Gamma_K} = \dim_{\bQ_p} V$. Such representations form an abelian category. We call $i \in \bZ$ a \emph{Hodge--Tate weight} of $V$ with multiplicity $m_i > 0$ when $\dim_K (V \otimes_{\bQ_p}  \mathbb{C}_p(-i))^{\Gamma_K} = m_i$, and we set $\deg_{HT}(V) = \sum_{i \in \bZ} i \cdot m_i$. 

\begin{defn}[{(cf.\ \cite[\S5.2.2]{fargues2019})}]
\label{defn:HTHN}
The \emph{Hodge--Tate polygon} of a Hodge--Tate $p$-adic $\Gamma_K$-representation is the HN-polygon arising from the HN-theory on the category of $\bQ_p$-linear Hodge--Tate $\Gamma_K$-representations induced by setting the degree of $V$ to be $\deg_{HT}(V)$. We write $\HN(HT(V))$ for this polygon. 
\end{defn}

As a result, a Hodge--Tate $\Gamma_K$-representation $V$ has a canonical increasing filtration 
\[
0 = V_0 \subset V_1 \subset \dotsm \subset V_m = V
\]
such that $\mu(V_i/V_{i-1}) < \mu(V_{i+1}/V_i)$. One can also associate to $V$ a decreasingly filtered $K$-vector space
\[
D_{HT}(V) := (V \otimes_{\bQ_p} B_{HT})^{\Gamma_K},
\]
where the filtration arises from the $\bZ$-grading of $B_{HT}$, observing that $\deg(V) = -\deg(D_{HT}(V))$.

Recall that a \emph{filtered $\phz$-module} over the $p$-adic field $K$ refers to a finite dimensional $K_0$-vector space $D$ with a Frobenius semi-linear isomorphism $\phz_D : D \ra D$ (an isocrystal) along with a decreasing filtration $\Fil_D^\bullet$ on $D_K := D \otimes_{K_0} K$.

When $V$ is Hodge--Tate and also has finite $E$-height, then $D_{HT}(V)$ can be recovered in another way. Firstly, we write $\fM = \fM(V)$ for a member of the isogeny class of flat Kisin modules associated to $V\vert_{\Gamma_{K_\infty}}$. There is a filtered isocrystal $D = D(\fM)$ associated to $\fM[1/p]$ following \cite[\S1.2.5]{crfc}. We will use homomorphisms $\fS \ra K_0$ (induced by the quotient $\fS \rsurj W(k)$ modulo $u$) and $\fS \ra K$ (induced by the quotient $\fS \rsurj \cO_K$ modulo $E(u)$) to write down $D(\fM)$. Namely, $D(\fM) := (D, \phz_D, \Fil^\bullet_D)$, where $D = \fM \otimes_{\fS} K_0$, $\phz_D = \phz_\fM \otimes_\fS K_0$, and the decreasing filtration on $D_K$ is constructed in \cite[\S1.2.7]{crfc}. 

We have now produced a filtered $K$-vector space in two ways, having started with a Hodge--Tate and finite $E$-height representation $V$ of $\Gamma_K$: $D_{HT}(V)$ and $D(\fM(V))_K$. When $V$ is, moreover, crystalline, these two filtered vector spaces are canonically isomorphic. 
\begin{prop}
Let $V$ be a crystalline $\Gamma_K$-representation. There is a canonical isomorphism of filtered vector spaces $D_{HT}(V) \simeq D(\fM(V))_K$. 
\end{prop}
\begin{proof}
Under the assumption that $V$ is crystalline, there exists a weakly admissible filtered isocrystal $D' = D_\mathrm{cris}(V)$. Then $D'_K$ is canonically isomorphic to $D_{HT}(V)$ and $D'$ is also canonically isomorphic to $D(\fM(V))$ \cite[Cor.\ 1.3.15]{crfc}.
\end{proof}

\begin{rem}
\label{rem:HN_motive}
In general, we do not know how to relate the filtered vector spaces $D_{HT}(V)$ and $D(\fM(V))_K$ for arbitrary $V$ that are Hodge--Tate and have finite $E$-height. So we take the approach that $D(\fM(V))_K$ is a replacement of the Hodge filtration for any representation $V$ of $\Gamma_{K_\infty}$ of finite $E$-height. 
\end{rem}

We work in the category of (finite-dimensional $\Q_p$-linear) representations of $\Gamma_{K_\infty}$ of bounded $E$-height, $\Rep_{K_\infty}^\mathrm{bh}$, which is equivalent to $\Mod_\fS^{\phz} \otimes \bQ_p$. Consequently, we may apply the HN-theory on $\Mod_\fS^{\phz} \otimes \bQ_p$ of \eqref{eq:HN_filt_isog} to $\Rep_{K_\infty}^\mathrm{bh}$. 

When one restricts to flat Kisin modules with height in $[a,a+1]$, all of which arise from a crystalline representation, Fargues has shown that $\HN(\fM(V)[1/p]) = \HN(HT(V))$ \cite[\S5.3, Cor.\ 2]{fargues2019}. The following theorem is as much as we can expect along those lines, in general. 
\begin{thm}
\label{thm:HN-HT}
Let $V$ be a crystalline $\Gamma_K$-representation. Then 
\[
\HN(\fM(V)[1/p]) \geq \HN(HT(V)).
\] 
\end{thm}
\begin{proof}
We know that $\deg_\fS(\fM(W)) = \deg_{HT}(W)$ for all crystalline $\Gamma_K$-representations $W$. Consequently, the endpoints of the polygons are the same. The theorem will follow from showing that the lattice of sub-$\Gamma_K$-representations of $V$ injects into the lattice of strict subobjects $\fM(V)$. This follows directly from \cite[Cor.\ 1.3.15]{crfc}, which states that $V \mapsto \fM(V)[1/p]$ is a fully faithful functor from crystalline representations to the isogeny category of Kisin modules of bounded height. 
\end{proof}

\begin{rem}
In analogy with Proposition \ref{HNoverHodge}, it is clear that we can prove the comparison of polygons ``HN over Hodge'' for the Hodge--Tate HN-theory of Definition \ref{defn:HTHN} that yields $\HN(HT(V))$. As a result, the inequality of Theorem \ref{thm:HN-HT} can be extended by the comparison $\HN(HT(V)) \geq P_\mu$, where $\mu$ is the Hodge type of $V$. This is a lift to characteristic $0$ of the module $p$ ``HN over Hodge'' result for Kisin modules, which is stated in Proposition \ref{HNoverHodge}. Notice also that both the HN and Hodge polygons increase when transitioning from characteristic $p$ to characteristic $0$: see Theorem \ref{thm:HNP_limit} and Remark \ref{rem:lifting_Hodge}. 
\end{rem}

\acknowledgements{
We owe a debt to Laurent Fargues, as the influence of his work will be clear to the reader. We also appreciate Jay Pottharst's exposition \cite{pottharst} of the formal structure of HN-theory, and thank him for making his manuscript available on arXiv. Likewise, we appreciate Daniel Halpern-Leistner's exposition in \cite{HL2014} of results we needed to cull from geometric invariant theory. We thank Christophe Cornut for a very careful reading of the text and several detailed emails that helped us correct and refine the proofs in \S\ref{sec:tensor}. We thank an anonymous referee for making us aware of the references \cite{kirwan1984, ness1984} on geometric invariant theory. We thank Madhav Nori for helpful conversations. We would like to recognize the mathematics departments at the University of Chicago and Brandeis University for providing excellent working conditions as this project was carried out.
}

\section{Appendix: Examples of HN-polygons} \label{sec:examples}

In this appendix, we give some examples of what the different components constructed in Theorem \ref{thm:HNcomp} correspond to concretely  beyond $\GL_2$.   We will assume for simplicity that $K = \Qp$.

 In the following figures, for different Hodge polygons $\nu$, we draw the set of all possible HN-polygons.
\begin{itemize}[leftmargin=2em]
\item For any $\cM$ of the appropriate dimension, the strata of $X_{\cM}^{\nu}$ will be indexed by this finite set of polygons.  
\item The Hodge polygon $\nu$ appears in black. 
\item We color the polygons the same if the corresponding strata can appear on the same component of $X_{\cM}^{\nu}$ according to Theorem \ref{thm:HNcomp}.
\end{itemize} 

\begin{rem}
For a particular \'etale $\phi$-module $\cM$ or \'etale $\phi$-module with $G$-structure $\cP$, not all strata will appear. In other words, the $X^{\nu}_{\cP, J}$ appearing in $X_{\cP}^{\nu}$ can be empty. However, in the figures, we include all possibilities which could potentially occur.  
\end{rem}

\begin{figure}[h]
\caption{Components for $\GL_3$, $\nu = (0,0,1)$}
\centering
\includegraphics[]{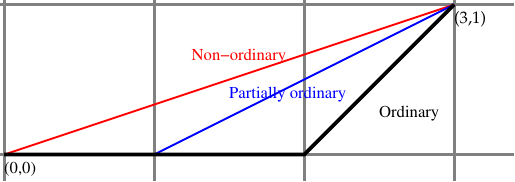}
\end{figure}

\bigskip

\begin{figure}[h]
\caption{Components for $\GL_3$, $\nu = (-1,0,1)$}
\centering
\includegraphics[]{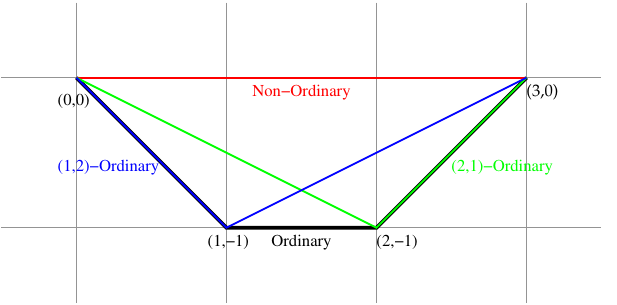}
\end{figure}

\bigskip

\begin{figure}[h]
\caption{Components for $\mathrm{GL}_4$, $\nu = (-1,0,0,1)$}
\centering
\includegraphics[]{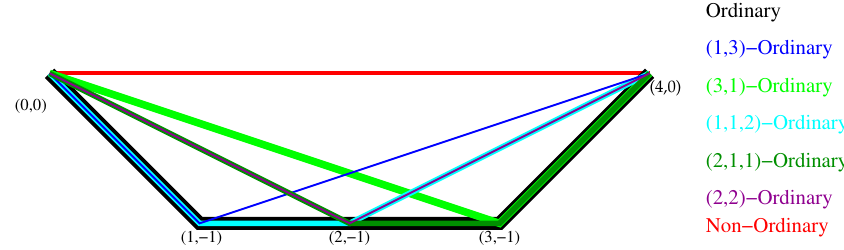}
\end{figure}

\bigskip

\begin{figure}[h]
\caption{Components for $\mathrm{GSp}_4$, $\nu = (-2,-1,1,2)$}
\centering
\includegraphics[]{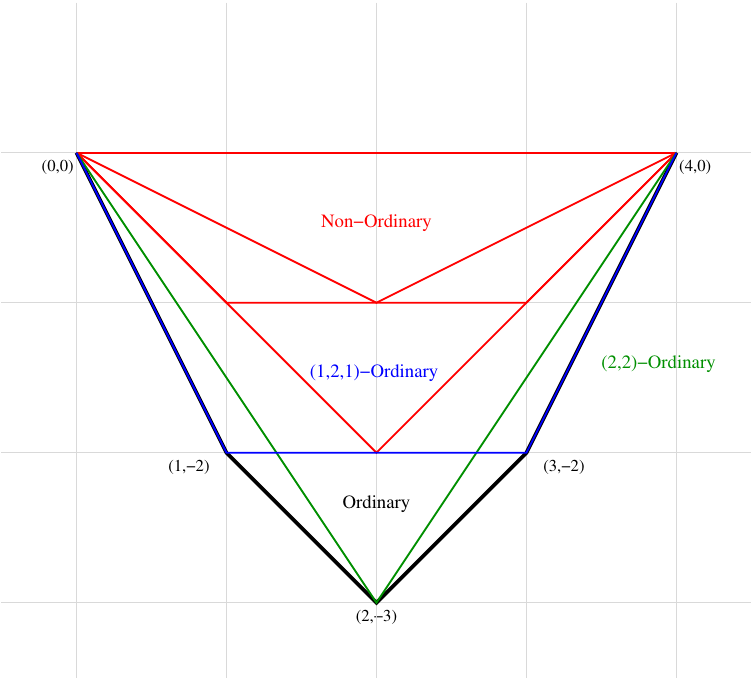}
\end{figure}

\clearpage

\bibliographystyle{amsalpha}
\bibliography{CWEbib-20_01-HN}

\end{document}